\documentclass[a4paper]{amsart}

\usepackage[a4paper,width={16cm},left=2.5cm,bottom=3cm, top=3cm]{geometry}

\usepackage{latexsym}
\usepackage{a4wide}
\usepackage{amscd}
\usepackage{graphics}
\usepackage{amsmath}
\usepackage{amssymb}
\usepackage{bm}
\usepackage{mathrsfs}
\usepackage{enumerate}
\usepackage{pgf,tikz}
\usepackage[utf8]{inputenc}
\usepackage{csquotes}

\usepackage{hyperref}
\usepackage{cleveref}
\hypersetup{
	linktocpage=true,
	colorlinks=true,
	linkcolor=blue!70!black,
	bookmarks=true,
	citecolor=orange!40!red,
	urlcolor=magenta!80!black,
}

\numberwithin{equation}{section} 

\newtheorem{theorem}{Theorem}[section]
\newtheorem{corollary}[theorem]{Corollary}
\newtheorem{lemma}[theorem]{Lemma}
\newtheorem{proposition}[theorem]{Proposition}
 \theoremstyle{definition}

 \newtheorem{remark}[theorem]{Remark}

\renewcommand{\S}{\mathbb{S}}	
\newcommand{\R}{\mathbb{R}}	
\newcommand{\N}{\mathbb{N}} 
\newcommand{\dx}{\,\mathrm{d}x}	
\newcommand{\dy}{\,\mathrm{d}y}	
\newcommand{\ds}{\,\mathrm{d}S}	
\newcommand{\dr}{\,\mathrm{d}r}	
\renewcommand{\d}{\mathrm{d}}
\newcommand{\e}{\varepsilon}	
\newcommand{\nnu}{\bm{\nu}}  

\newcommand{\norm}[1]{\left\lVert #1 \right\lVert}
\newcommand{\abs}[1]{\left| #1 \right|}
\newcommand{\sub}{\subseteq}

\newcommand{\xint}[1]{\xi^{\textup{int}}_{#1,R,\e}} 

\DeclareMathOperator{\supp}{\mathrm{supp}}

\DeclareMathOperator{\dive}{\mathrm{div}}

\newenvironment{bvp}{\left\{\begin{aligned}  }{\end{aligned}\right.}

\begin{document}

\title[Eigenvalues with moving mixed boundary conditions]{Eigenvalues
  of the Laplacian with moving mixed boundary conditions: the case of
  disappearing Neumann region}

\author[V. Felli]{Veronica Felli}
\address{Veronica Felli
  \newline \indent Dipartimento di Matematica e Applicazioni
  \newline \indent
Universit\`a degli Studi di Milano–Bicocca
\newline\indent Via Cozzi 55, 20125 Milano, Italy}
\email{veronica.felli@unimib.it}

\author[B. Noris]{Benedetta Noris}
\address{Benedetta Noris
  \newline \indent Dipartimento di Matematica
  \newline \indent Politecnico di Milano\newline\indent Piazza Leonardo da Vinci 32, 20133 Milano, Italy}
\email{benedetta.noris@polimi.it}

\author[R. Ognibene]{Roberto Ognibene}
\address{Roberto Ognibene
	\newline \indent Dipartimento di Matematica
	\newline \indent Universit\`a di Pisa
	\newline\indent  Largo Bruno Pontecorvo, 5, 56127 Pisa, Italy}
\email{roberto.ognibene@dm.unipi.it}

\date{March 9, 2022}
\begin{abstract}
We deal with eigenvalue problems for the Laplacian with varying mixed
  boundary conditions, consisting in  homogeneous Neumann conditions on
  a vanishing portion of the boundary and Dirichlet conditions on the
  complement. By the study of an Almgren-type frequency function, we derive upper and lower bounds of the
   eigenvalue variation and sharp estimates in the case of a strictly
   star-shaped Neumann region.   
\end{abstract}
	
	\maketitle
	

\medskip 


\section{Introduction}\label{sec:intr}
The present paper concerns the eigenvalue problem for the Laplacian in
a bounded domain, with mixed Dirichlet-Neumann homogeneous boundary
conditions prescribed on variable portions of the boundary. More
precisely, we focus on
  a perturbative problem characterized by the
  disappearance, in some limiting process,
of the region where Neumann
boundary conditions are imposed. In this situation, the
 eigenvalues of the mixed problem converge to Dirichlet eigenvalues: we
aim to study the rate of this convergence. This paper is the
counterpart of \cite{FNO1}, where the case of Dirichlet disappearing
region is studied.

In the literature there have been several
  contributions on asymptotic
  behaviour of eigenvalues of elliptic boundary
  value problems under singular perturbation of  the boundary
  conditions. Concerning, in particular, the case treated
  in the present paper, i.e. the perturbation of a Dirichlet
  problem by imposing a homogeneous Neumann condition on a
  vanishing portion of the boundary, we mention the results in   
  \cite{Gadyl'shin1992}, where  a full asymptotic
  expansion of perturbed eigenvalues is obtained in dimension $2$, see
  also  \cite{AFL2018};  we mention additionally the paper
    \cite{CP}, concerned with  the spectral stability of the first eigenvalue.
  The complementary problem, i.e. a Neumann problem perturbed with
  a Dirichlet condition on a small part of the boundary, is
  treated by \cite{Gadylshin1993} in dimension 2 and by \cite{FNO1}
  in any dimension $N\geq3$.
The approach developed in \cite{FNO1} is based on a capacity argument
inspired by \cite{Courtois1995} and \cite{AFHL}, where the problem of spectral
stability for the Dirichlet Laplacian in domains with small holes
 is investigated; in particular in \cite{FNO1} the sharp
asymptotic behaviour of perturbed eigenvalues is described in terms of
the Sobolev capacity of the boundary portion where the Dirichlet
condition is imposed. This kind of method does not seem to be effective
in the case of a disappearing Neumann region, being the Dirichlet boundary set
not small. Therefore, in the present paper we treat this case with a
different approach, based on blow-up analysis for scaled eigenfunctions
and energy estimates obtained by monotonicity formulas, in the spirit
of \cite{AFL2018}; we point out that the case of dimension
$N\geq 3$ presents several additional difficulties with respect to the
$2$-dimensional case treated in \cite{AFL2018}, because of the
occurrence of some effects of the geometry of the Neumann region in
the monotonicity argument, see Remark \ref{rmk:dimension}.

Let $\Omega\subset\R^N$, $N\geq 3$, be a bounded open set such that 
$\partial \Omega$ is of class $C^{1,1}$ in a neighbourhood of $0\in\partial\Omega$, namely 
there exist $r_0\in(0,1)$ and $g\in C^{1,1}(\R^{N-1})$ such that 
\begin{equation}\label{eq:1}
B_{r_0}\cap \Omega=\{x\in B_{r_0}\colon x_N>g (x')
\}\quad\text{and}\quad 
B_{r_0}\cap \partial\Omega=\{x\in B_{r_0}\colon x_N=g (x') \},
\end{equation}
where 
$B_{r_0}=\{x=(x_1,x_2,\dots,x_N)\in\R^N:|x|<r_0\}$ is the ball in $\R^N$ centered at the origin  with radius $r_0$ and $x'=(x_1,\dots,x_{N-1})$.
Up to a suitable choice of the Cartesian
coordinate system, it is not restrictive to assume that 
\begin{equation}\label{eq:2}
g(0)=0\quad\text{and}\quad 
\nabla g (0)=0,
\end{equation}
i.e. that $\partial\Omega$ is tangent to the coordinate hyperplane $\{x_N=0\}$ in the origin.

Let $\mathcal V$ be a bounded open set in $\R^N$ such that
\begin{equation}\label{eq:7}
  0\in \mathcal V\quad\text{and}\quad \mathop{\rm diam}(\mathcal V)
  =\sup\{|x-y|:x,y\in\mathcal V\}<r_0.
\end{equation}
For every $\e\in(0,1)$, let
\begin{equation}\label{eq:sigmaeps}
\Sigma_\e:=(\e\mathcal V)\cap\partial\Omega,
\end{equation}
where $\e\mathcal V=\{\e x:x\in \mathcal V\}\subset
B_{r_0}$. Furthermore, we set 
\begin{equation}\label{eq:Sigma}
\Sigma=\mathcal V\cap\partial
\R^N_+=\{x=(x_1,x_2,\dots,x_N)\in\mathcal V:x_N=0\},
\end{equation}
where $\R^N_+=\{(x',x_N)\in\R^N=\R^{N-1}\times\R:x_N>0\}$.

We consider the following eigenvalue problem with mixed boundary conditions
\begin{equation}\label{eq:main_intro}
\begin{cases}
-\Delta u=\lambda u, \quad & \text{in } \Omega, \\
u=0, & \text{on }\partial\Omega\setminus \Sigma_\e, \\
\partial_{\nnu} u=0, & \text{on } \Sigma_\e,
\end{cases}
\end{equation}
and we are interested in the asymptotic behaviour of its eigenvalues
as $\varepsilon\to0^+$, that is when the Neumann region $\Sigma_\e$ is
  disappearing.

In order to write the weak formulation of \eqref{eq:main_intro}, we
first introduce the suitable functional framework.  For any open set
$\omega\sub\R^N$ and for any closed set
$\Gamma\sub\partial \omega$, we define $H^1_{0,\Gamma}(\omega)$
  as the closure in $H^1(\omega)$ of
  $C_c^\infty(\overline{\omega}\setminus\Gamma)$; we refer to
  \cite{Egert-Tolksdorf} for a more detailed analysis of this kind of
  space (see also \cite{bernard}).
 We note that,
if $\Omega$ and $\mathcal V$ are sufficiently regular (for example Lipschitz), then we have the following characterization
\[
H^1_{0,\partial\Omega\setminus \Sigma_\e}(\Omega)=
\{ u\in H^1(\Omega): \, \mathop{\rm Tr}(u)=0 \text{ on } \partial\Omega\setminus \Sigma_\varepsilon\}
\]
for every $\e\in(0,1)$, where $\mathop{\rm Tr}$ denotes the trace
operator (see \cite{bernard}).  We note also that, formally, when
$\varepsilon=0$ the space $H^1_{0,\partial\Omega}(\Omega)$
  coincides with the usual Sobolev space $H^1_0(\Omega)$.

We say that $\lambda\in\R$ is an \emph{eigenvalue} of problem
\eqref{eq:main_intro} if there exists
$\varphi\in H^1_{0,\partial\Omega\setminus \Sigma_\e}(\Omega)$,
$\varphi\not\equiv0$, called an \emph{eigenfunction}, such that
\begin{equation}\label{eq:main_intro_weak}
\int_\Omega \nabla \varphi\cdot\nabla v\dx=
\lambda\int_\Omega \varphi v \dx
\quad \text{ for every } v \in H^1_{0,\partial\Omega\setminus\Sigma_\e}(\Omega).
\end{equation}
From classical spectral theory, \eqref{eq:main_intro_weak} admits a diverging sequence of positive eigenvalues
\[
  0<\lambda_1^\varepsilon<\lambda_2^\varepsilon \leq
  \lambda_3^\varepsilon\leq \cdots\leq
  \lambda_i^\varepsilon\leq\cdots,
\]
where each eigenvalue is repeated according to its
multiplicity. Letting $\N_*:=\N\setminus\{0\}$, we denote by
$\{\varphi_i^\varepsilon\}_{i\in \N_*}$ a corresponding sequence of eigenfunctions
satisfying 
\begin{equation}\label{eq:ortonor}
  \int_\Omega \varphi_i^\varepsilon
  \varphi_j^\varepsilon=\delta_i^j,
\end{equation}
with $\delta_i^j$ denoting the usual \emph{Kronecker
  delta.}
In  the sense clarified in  \eqref{eq:main_intro_weak}, the
   functions 
 $\varphi_i^\varepsilon$ weakly
 solve 
\begin{equation}\label{eq:eqphiie}
	\begin{cases}
	-\Delta \varphi_i^\e=\lambda_i^\e \varphi_i^\e, & \text{in }\Omega, \\
	\varphi_i^\e=0, &\text{on }\partial\Omega\setminus \Sigma_\e, \\
	\partial_{\nnu} \varphi_i^\e =0, &\text{on }\Sigma_\e.
	\end{cases}
      \end{equation}
   In the limit $\varepsilon=0$, once the Neumann region $\Sigma_\e$ 
      has disappeared, we formally recover 
the eigenvalue problem for the standard       Dirichlet Laplacian
      \begin{equation}\label{eq:str_dir}
      	\begin{cases}
      		-\Delta u=\lambda u, &\text{in }\Omega, \\
      		u=0, &\text{on }\partial\Omega,
      	\end{cases}
      \end{equation}
which is well known to admit a diverging sequence of
positive eigenvalues
      \[
      	0<\lambda_1<\lambda_2\leq\lambda_3\leq  \cdots\leq \lambda_i\leq \cdots.
      \]
We  denote by $\{\varphi_i\}_{i\in \N_*}$ a corresponding sequence of eigenfunctions satisfying
      \begin{equation}\label{eq:normaliz_limit_eigenfunction}
      	\int_{\Omega}\varphi_i\varphi_j\dx=\delta_i^j.
      \end{equation}
      More precisely, $\lambda_i$ and $\varphi_i$ solve
      \eqref{eq:str_dir} in the sense that
      $\varphi_i\in H^1_0(\Omega)$ and
 \begin{equation}\label{eq:eqphi_0}
   \int_{\Omega}\nabla\varphi_i\cdot\nabla v\dx=\lambda_i\int_\Omega\varphi_i v\dx
   \quad\text{for all }v \in H^1_0(\Omega).
 \end{equation}
In Section \ref{sec:conv-eigen} we prove that, for all $i\in\N_*$,
\[
	\lambda_i^\varepsilon\to \lambda_i\quad\text{as }\varepsilon\to 0.
\]
The main goal of the present paper is to detect
the sharp rate of the above convergence.

The vanishing rate of the eigenvalue
  variation $\lambda_i^{\e}-\lambda_i$ 
  turns out to be strongly related to  
the behaviour of the Dirichlet
eigenfunctions $\varphi_i$, locally near the point $0\in\partial\Omega$.
 We can derive  from \cite{FF-proca} the classification of
  possible vanishing orders of $\varphi_i$ at the boundary: for every
$i\in \N_*$, there exist $\gamma_i\in \N_*$,
$\Psi_i\in H^1_0(\mathbb{S}^{N-1}_+)$, with $\Psi_i\not\equiv0$, such
that
\begin{equation}\label{eq:gamma_i_def}
  \frac{\varphi_i(rx)}{r^{\gamma_i}} \to |x|^{\gamma_i}\Psi_i\left(\frac{x}{|x|}\right) \quad\text{in } H^1(B_\rho)\text{ as } r\to 0, \text{ for every } \rho>0,
\end{equation}
where $\mathbb{S}^{N-1}_+=\{(x_1,\ldots,x_N)\in\R^N:  |x|=1,\, x_N>0
\}$ and $\varphi_i$, respectively $\Psi_i$, are
trivially extended outside $\Omega$, respectively $\mathbb{S}^{N-1}_+$. Moreover the function $\Psi_i$ is the restriction to
$\mathbb{S}^{N-1}_+$ of a spherical harmonic odd with respect to the
equator $x_N=0$  and the $\gamma_i$-homogeneous function $\psi_i$
with angular profile $\Psi_i$, i.e. 
\begin{equation}\label{eq:psi_i_def}
\psi_i(x)=|x|^{\gamma_i} \Psi_i\left(\frac{x}{|x|}\right),
\end{equation}
is a harmonic homogeneous  polynomial of degree $\gamma_i$ vanishing
on $\partial\R^N_+$. In
particular, being $\psi_i$ harmonic, nontrivial and vanishing on
$\partial\R^N_+$, we have that $\partial_{x_N}\psi_i\not\equiv0$ on
$\Sigma$.

 In the following we fix $n_0\in \N_*$ such that
\begin{equation}\label{eq:simple_hp}
\lambda_{n_0} \text{ is simple}
\end{equation}
as an eigenvalue of the  standard Dirichlet Laplacian in $\Omega$. Moreover, hereafter we denote 
\begin{equation}\label{eq:notation-n_0}
	\gamma:=\gamma_{n_0}
\end{equation}
and 
\begin{equation}\label{eq:def_psi_0}
	\Psi:=\Psi_{n_0}, \quad \psi:=\psi_{n_0}.
      \end{equation}
Under assumption \eqref{eq:simple_hp} it is possible to choose the
eigenfunctions  $\varphi_{n_0}^\varepsilon$, solving \eqref{eq:eqphiie}
with $i=n_0$, in such a way that
\begin{equation*}
  \varphi_{n_0}^\varepsilon\to \varphi_{n_0}\quad\text{in }H^1(\Omega)\quad\text{as }\varepsilon\to0,
\end{equation*}
see Proposition \ref{prop:conv_eigenfunctions}.

The scaled shape \eqref{eq:Sigma} of the Neumann disappearing region
emerges in the asymptotic expansion of the eigenvalue variation in the
guise of a coefficient $\mathcal C_{n_0}(\Sigma)$ admitting a
variational characterization.
We define $\mathcal C_{n_0}(\Pi)$ for any bounded open subset
$\Pi\subset\partial\R^N_+$ such that $0\in\Pi$.  Letting
$\mathcal{D}^{1,2}(\R^N_+\cup \Pi)$ be the completion of
$C^\infty_c(\R^N_+\cup \Pi)$ with respect to the norm
\begin{equation*}
  \|u\|_{\mathcal{D}^{1,2}(\R^N_+\cup
    \Pi)}=\sqrt{\int_{\R^N_+}|\nabla u|^2\dx},
\end{equation*}
we define
\begin{equation}\label{eq:58}
  \mathcal C_{n_0}(\Pi):=-2\min\Big\{J_\Pi(u):u\in \mathcal{D}^{1,2}(\R^N_+\cup
  \Pi)\Big\},
\end{equation}
where
\begin{equation}\label{eq:57}
  J_\Pi: \mathcal{D}^{1,2}(\R^N_+\cup
  \Pi)\to\R,\quad J_\Pi(u):=\frac{1}{2}\int_{\R^N_+}\abs{\nabla u}^2\dx+\int_{\Pi}u\partial_{\nnu} \psi\dx',
\end{equation}
$\psi$ is defined in \eqref{eq:def_psi_0}, and
$\nnu=(0,0,\dots,0,-1)$ is the vertical downward unit vector.  By
classical variational methods one can easily prove that the minimum in
\eqref{eq:58} is attained (by the function $w_{0,\Pi}$ defined in
\eqref{eq:def_w_0}) and $\mathcal C_{n_0}(\Pi)>0$, see Section
\ref{sec:limit-profiles}.  We define
\begin{equation}
  \label{eq:64}
  \mathcal C_{n_0}=\mathcal C_{n_0} (B_1')>0,
\end{equation}
where we are denoting, for all $r>0$,
\begin{equation}\label{eq:72}
  B_r' :=B_r\cap\partial \R^N_+.
\end{equation}
Our first main result provides asymptotic lower and
  upper bounds of the eigenvalue variation as $\e\to0$.

\begin{theorem}\label{thm:nonstarshaped}
  Let $\mathcal V\subset \R^N$ be a bounded open set satisfying
  \eqref{eq:7} and $0<r_{\mathcal V}<R_{\mathcal V}<r_0$ be such that
  $B_{r_{\mathcal V}}\subset\mathcal V\subset B_{R_{\mathcal
      V}}$. Then
\begin{equation*}
  \mathcal C_{n_0}r_{\mathcal V}^{N+2\gamma-2}\leq
  \liminf_{\e\to0}\frac{\lambda_{n_0}-\lambda_{n_0}^\e}{\e^{N+2\gamma-2}}
  \leq   \limsup_{\e\to0}\frac{\lambda_{n_0}-\lambda_{n_0}^\e}{\e^{N+2\gamma-2}}
  \leq \mathcal C_{n_0}R_{\mathcal V}^{N+2\gamma-2}
\end{equation*}
with $\mathcal C_{n_0}$ as in \eqref{eq:64} and $\gamma$ as in \eqref{eq:notation-n_0}.
\end{theorem}
The proof of Theorem \ref{thm:nonstarshaped} will be obtained by
comparison of the eigenvalues $\lambda_{n_0}^\e$ of problem
\eqref{eq:eqphiie} with the eigenvalues of the analogous problems with
$\mathcal V$ replaced by the balls $B_{R_{\mathcal
    V}}$ and $B_{r_{\mathcal
    V}}$, for which a more precise asymptotic expansion can be derived, exploiting the star-shapedness of the Neumann region. 

We now state the sharper asymptotic estimates which we are able to obtain under 
stronger regularity and geometric assumptions on the set $\mathcal V$.  
Let us assume, in addition to \eqref{eq:7}, that 
\begin{equation}\label{eq:V-C11}
\mathcal V \text{ is of class }C^{1,1}
\end{equation}
and 
$\mathcal V$ is strictly star-shaped with respect to the origin,
i.e.
\begin{equation}\label{eq:V-star-shaped}
 \text{there exists
    $\sigma>0$ such that 
  }x\cdot\nnu(x)\geq\sigma  \quad\text{for every }x\in\partial \mathcal V,
\end{equation}
where $\nnu(x)$ is the exterior unit normal vector at
$x\in\partial\mathcal V$. We observe that the notion of \emph{strict
  star-shapedness} given in \eqref{eq:V-star-shaped} is equivalent to
the notion of \emph{star-shapedness with respect to a ball} discussed
in \cite[Section 1.1.8]{mazja}, see \cite[Lemma 1]{Boulkhemair}
and \cite[Lemma 3.2]{kim-kwon}.

\begin{theorem}\label{thm:main}
If $\mathcal V$ satisfies \eqref{eq:V-C11} and
\eqref{eq:V-star-shaped} in addition to \eqref{eq:7}, then the following asymptotic expansion holds
\begin{equation*}
  \lambda_{n_0}^\e=\lambda_{n_0}-\mathcal C_{n_0}(\Sigma)
  \e^{N+2\gamma-2}+o(\e^{N+2\gamma-2})\quad\text{as
    }\e\to0,
\end{equation*}
	with  $\mathcal C_{n_0}(\Sigma)$  as in
          \eqref{eq:58} and $\Sigma$ as in \eqref{eq:Sigma}.
\end{theorem}

The proof of  Theorem \ref{thm:main} is obtained through 
sharp estimates from above and below of the 
Rayleigh quotients for the eigenvalues $\lambda_{n_0}^\e$ and
$\lambda_{n_0}$, which in turn require energy bounds, uniform with
respect to $\e$,  on
eigenfunctions, 
provided by an Almgren-type
monotonicity argument. The last step in the achievement of sharp eigenvalue estimates consists of a blow-up analysis for scaled eigenfunctions.

The Almgren frequency function at a point is given by the ratio
  of the local energy over the mass near that point, see
  \cite{Almgren}; we refer to formula
  \eqref{eq:73} in Section \ref{sec:energy-estimates-via} for the
  precise definition of the frequency for our problem. Monotonicity of this quotient
  implies a uniform control of the local energies and, in the
  classical case of harmonic functions, such a monotonicity easily
  follows from the positivity of the derivative. For solutions $u$ of
  elliptic equations of the type
\[
	-\Delta u = Vu,
\]
with $V$ bounded, the frequency is no more monotone because of the
presence of the potential. However, the \enquote{perturbed frequency}
\[
	\mathcal{N}_V(r):= \frac{\displaystyle r\int_{\abs{x-x_0}<r}(\abs{\nabla u}^2-Vu^2)\dx}{\displaystyle \int_{\abs{x-x_0}=r} u^2\ds}
\]
still enjoys some monotonicity properties, in the form of an estimate of the type
\[
	\mathcal{N}'_V(r)\geq -\mathrm{const}\, \mathcal{N}_V(r),
\]
which allows proving boundedness of $\mathcal{N}_V$ and then energy
estimates for $u$. In this spirit, here we mean to prove boundedness
of the frequency of eigenfunctions $\varphi_i^\e$ at the origin (which
belongs to the boundary), uniformly with respect to the parameter
$\e$, by establishing its perturbed monotonicity through an estimate
from below of its derivative.  Since, in the case we are considering
here, all neighbourhoods of the origin contain portions of the
boundary, some additional boundary terms appear in the derivative of
the frequency; star-shapedness assumption \eqref{eq:V-star-shaped}
forces these remainder terms to have a sign which is favorable to the
desired estimate. On the other hand, the lack of regularity of the
eigenfunctions $\varphi_i^\e$ at Dirichlet-Neumann junctions prevents
us from a direct differentiation of the frequency function, which
requires a Pohozaev-type identity based on the integration of the
Rellich-Necas identity \eqref{eq:rellich_necas}. For what concerns
this last issue, considerable differences appear between the cases
$N=2$ and $N\geq 3$, as explained in the following Remark.

\begin{remark}\label{rmk:dimension}
  With respect to the $2$-dimensional case treated in \cite{AFL2018},
  significant new difficulties arise, mainly due to regularity issues
  for mixed boundary value problems like \eqref{eq:eqphiie}, which
  turn out to be more delicate in dimension $N\geq3$ because of the
  positive dimension of the junction set $\partial \Sigma_\e$ and some
  role played by the geometry of $\Sigma_\e$, in particular in
  connection with its star-shapedness. Indeed, when $N=2$ the
  interface $\partial\Sigma_\e$ has zero dimension (basically, it
  consists of a couple of points) and it is possible to perform an
  approximation just by removing a small neighbourhood of the junction
  points, thus allowing quite explicit computations in the derivation
  of Almgren monotonicity formulas (see \cite[Lemma C.5]{AFL2018}). In
  higher dimensions, we overcome the difficulties produced by lack of
  regularity of solutions by constructing a sequence of approximating
  problems, for which enough regularity is available to derive
  Pohozaev-type identities, needed, in turn, to obtain Almgren-type
  monotonicity formulas and consequently to perform blow-up
  analysis. In particular, the geometry of the boundary manifests in
  the form of some extra remainder terms appearing in the
  Pohozaev-type identity for the regularized problem and depending on
  the mutual orientation of normal and position vectors, whose control
  motivates here the geometric assumption \eqref{eq:V-star-shaped},
  which is, in fact, a star-shapedness condition on the Neumann
  region, see Proposition \ref{p:poho} and, in particular,
  \eqref{eq:starshape_approx}.
\end{remark}

Under the same assumptions of Theorem \ref{thm:main},
the blow-up analysis performed in Section
    \ref{sec:blow-up-analysis} allows us to describe the behaviour of
  perturbed eigenfunctions $\varphi_{n_0}^\e$
when they are scaled at
the origin, i.e. at the point around which $\Sigma_\e$ is shrinking,
thus yielding the following result.
	\begin{theorem}\label{thm:sharp_eigenfunction}
		Let $\mathcal{V}$ satisfy \eqref{eq:7},
                \eqref{eq:V-C11} and \eqref{eq:V-star-shaped}. Let
                $U=\psi+w_{0,\Sigma}$, where $\psi$ is as in
                \eqref{eq:def_psi_0} and $w_{0,\Sigma}$ (defined in
                \eqref{eq:def_w_0}) is the unique minimizer of the
                functional $J_{\Sigma}$ introduced in \eqref{eq:57}. Then, for any $R>0$ sufficiently large,
		\begin{gather*}
			\varepsilon^{-N-2\gamma}\int_{\Omega\cap B_{R\varepsilon}}\abs{\varphi_{n_0}^\varepsilon}^2\dx\to \int_{B_R^+} U^2\dx \\
			\varepsilon^{-N-2\gamma+2}\int_{\Omega\cap B_{R\varepsilon}}\abs{\nabla\varphi_{n_0}^\varepsilon}^2\dx\to \int_{B_R^+} \abs{\nabla U}^2\dx	,
		\end{gather*}
		as $\varepsilon\to 0$.
\end{theorem}

The paper is structured as follows. In Section \ref{sec:conv-eigen} we
prove convergence of eigenvalues and eigenfunctions as $\e\to0$ to
eigenelements of the unperturbed problem. In Section
\ref{sec:limit-profiles} we construct the limit profiles, which will
appear in the blow-up analysis, by minimization of the functional
introduced in \eqref{eq:57}. In Section \ref{sec:an-equiv-probl} we
introduce  an equivalent auxiliary problem 
obtained by deforming the boundary of $\Omega$ into a straight
hyperplane; to this aim we use a particular 
diffeomorphism, introduced in  \cite{Adolfsson1997}
 and made on purpose to ensure that the equation is conserved by reflection
 through the straightened boundary.
Section \ref{sec:pohoz-type-ineq} is devoted to  a
Pohozaev-type identity for the approximating problems, which is then used in Section
\ref{sec:energy-estimates-via} to develop a monotonicity argument, from
which energy estimates follow.
In Sections \ref{sec:upper-bound-texorpdf} and
\ref{sec:lower-bound-texorpdf} we prove sharp upper and lower bounds
for the eigenvalue variation, while Section \ref{sec:blow-up-analysis}
is
devoted to  a blow-up analysis for scaled eigenfunctions. In Section
\ref{sec:proof_main} we combine the lower/upper estimates on the
eigenvalue variation and the blow analysis to prove Theorem
\ref{thm:main}, which is then combined with a comparison and scaling
argument to prove Theorem \ref{thm:nonstarshaped}  in Section
\ref{sec:proof-theor-refthm:m}.
Finally, in
the appendix 
we recall some Poincaré-type inequalities and an abstract lemma on maxima of quadratic forms.

\section{Convergence of eigenelements}\label{sec:conv-eigen}

In the following we tacitly assume that the hypotheses on $\Omega$ set out in the Introduction and
assumption \eqref{eq:7} on $\mathcal V$ are satisfied;
  consequently we let 
$\Sigma_\e$ be as in \eqref{eq:sigmaeps}. In this section
we prove that the eigenvalues and eigenfunctions of the perturbed
problem converge, as $\e\to0$, to the corresponding unperturbed
eigenelements.

\begin{lemma}\label{lemma:prelim_conv}
  For any $\e\in(0,1)$, let $\lambda_\e\in\R$ be an eigenvalue of
  problem \eqref{eq:main_intro_weak} and 
  $\varphi_\e\!\in \!H^1_{0,\partial\Omega\setminus \Sigma_\e}\!(\Omega)$ be
  an associated eigenfunction such that
  $\int_\Omega\varphi_\e^2\dx=1$. Let us assume that there exists a
  decreasing sequence $(\e_n)_n\sub(0,1)$ and a real number
  $\lambda^*$ such that $\e_n\to 0$ and $\lambda_{\e_n}\to\lambda^*$
  as $n\to\infty$. Then there exist a subsequence $(\e_{n_i})_i$ and
  $\varphi^*\in H^1_0(\Omega)$ such that
\begin{equation*}
  \varphi_{\e_{n_i}}\rightharpoonup
  \varphi^*\text{ weakly in }H^1(\Omega) \quad
  \text{and}\quad
  \varphi_{\e_{n_i}}\to \varphi^*\text{ strongly in }L^2(\Omega),
\end{equation*}
as $i\to\infty$. Moreover
$\lambda^*$ is an eigenvalue of the Dirichlet
  Laplacian in $\Omega$ with $\varphi^*$ as an eigenfunction, in the
  sense of \eqref{eq:eqphi_0}, and $\int_\Omega\abs{\varphi^*}^2\dx=1$.
\end{lemma}
\begin{proof}
By hypothesis we have that
\[
  \int_\Omega\varphi_{\e_n}^2\dx=1\quad\text{and}\quad\int_\Omega\abs{\nabla\varphi_{\e_n}}^2\dx
  =\lambda_{\e_n}=\lambda^*+o(1)
\]
as $n\to\infty$. Therefore there exist a subsequence $(\e_{n_i})_i$ and $\varphi^*\in H^1(\Omega)$ such that
\begin{equation*}
  \varphi_{\e_{n_i}}\rightharpoonup \varphi^* \text{ weakly in
  }H^1(\Omega)\quad\text{and}\quad \varphi_{\e_{n_i}}\to \varphi^*
  \text{ strongly in }L^2(\Omega),
\end{equation*}
as $i\to \infty$. We first aim at showing that
$\varphi^*\in H^1_0(\Omega)$.  In order to do this, let $r_2>r_1>0$
and let $\zeta \in C^\infty_c(\R^N)$ be such that
$\text{supp}(\zeta) \subset B_{r_2}$ and $\zeta(x)=1$ for every
$x\in B_{r_1}$. For every $\delta>0$ and $x\in \R^N$, we define
$\zeta_\delta(x)=\zeta(x/\delta)$. First we notice that
$(1-\zeta_\delta)\varphi^* \in H^1_0(\Omega)$ for every
$\delta>0$. Indeed, by approximation of $\varphi_{\e_{n_i}}$ with
$C^\infty_c(\Omega\cup\Sigma_{\e_{n_i}})$-functions, we see that, for
every fixed $\delta>0$ and for $i$ large enough,
$(1-\zeta_\delta)\varphi_{\e_{n_i}} \in H^1_0(\Omega)$; moreover,
$(1-\zeta_\delta)\varphi_{\e_{n_i}} \rightharpoonup
(1-\zeta_\delta)\varphi^*$ weakly in $H^1(\Omega)$ as $i\to\infty$,
thus implying that
  $(1-\zeta_\delta)\varphi^*\in H^1_0(\Omega)$. Secondly, we have
that, as $\delta\to0$, 
	\begin{multline*}
          \|(1-\zeta_\delta)\varphi^*-\varphi^*\|_{H^1(\Omega)}^2=
          \|\zeta_\delta \varphi^*\|_{H^1(\Omega)}^2 \leq \int_\Omega
          \left( 2|\nabla\zeta_\delta|^2(\varphi^*)^2
            +2\zeta_\delta^2|\nabla\varphi^*|^2+\zeta_\delta^2(\varphi^*)^2
          \right)\dx  \\
          \leq 2 \delta^{-2} \|\nabla \zeta\|_{L^\infty(\R^N)}^2
          \int_{B_{\delta r_2}\setminus B_{\delta r_1}} (\varphi^*)^2
          \dx +o(1) \leq C \left( \int_{B_{\delta r_2}\setminus
              B_{\delta r_1}} (\varphi^*)^{2^*} \dx \right)^{2/2^*}
          +o(1) = o(1),
	\end{multline*}
        for some $C>0$, with $2^*=2N/(N-2)$.  Hence
        $(1-\zeta_\delta)\varphi^* \in H^1_0(\Omega)$ for every
        $\delta>0$ and converges to $\varphi^*$ in $H^1(\Omega)$ as
        $\delta\to0$, so that $\varphi^*\in H^1_0(\Omega)$.
	
By strong $L^2(\Omega)$-convergence, we have that
$\int_\Omega\abs{\varphi^*}\dx=1$. Finally, by hypothesis we have
that, 	for every $i$,
\begin{equation*}
  \int_\Omega\nabla\varphi_{\e_{n_i}}\cdot\nabla\phi\dx=\lambda_{\e_{n_i}}
  \int_\Omega\varphi_{\e_{n_i}}\phi\dx,
\end{equation*}
for
    all $\phi\in
  H^1_{0,\partial\Omega\setminus\Sigma_{\e_{n_i}}}(\Omega)$
 and, in particular, for any $\phi\in H^1_0(\Omega)$. Passing to the limit as $i\to\infty$ in the previous equation for $\phi\in H^1_0(\Omega)$ proves that $\varphi^*$ and $\lambda^*$ satisfy \eqref{eq:eqphi_0} thus completing the proof.
\end{proof}

\begin{remark}\label{rem:conv-eigen-1}
  Some basic relations among the families of
    perturbed eigenvalues and between the perturbed and unperturbed
    sequences can be easily observed.
 The eigenvalues $\lambda_i^\e$ admit the following classical   
	Min-Max variational characterization 
	\begin{equation}\label{eq:min_max}
		\lambda_i^\varepsilon =\min \left\{
		\max_{u\in F_i} \frac{\|\nabla u\|_{L^2(\Omega)}^2}{\|u\|_{L^2(\Omega)}^2}:\, F_i\subset H^1_{0,\partial\Omega\setminus \Sigma_\e}(\Omega) \text{ $i$-dimensional subspace} 
		\right\}.
	\end{equation}
	By \eqref{eq:7}, for every
          $\e_1>0$ there exists 	$0<\e_2<\e_1$ such that 
	$H^1_{0,\partial\Omega\setminus \Sigma_{\e_2}}(\Omega) \subset
	H^1_{0,\partial\Omega\setminus \Sigma_{\e_1}}(\Omega)$. Then for
        every sequence $\e_n\to0$ there exists a decreasing subsequence
        $\{\e_{n_k}\}$ such that, for
	every $i\in \N_*$, 
	\begin{equation*}\label{eq:28}
		\lambda_i^{\varepsilon_{n_{k+1}}}\geq \lambda_i^{\e_{n_k}} \quad\text{for every } k.
	\end{equation*}
Moreover, since 
	$H^1_0(\Omega)\subset H^1_{0,\partial\Omega\setminus
		\Sigma_\e}(\Omega)$ for every $\varepsilon>0$, we readily get, for
	every $i\in \N_*$,
	\begin{equation}\label{eq:28}
		\lambda_i\geq\lambda_i^\varepsilon \quad\text{for every }\varepsilon>0.
	\end{equation}
\end{remark}

\begin{proposition}\label{propo:conv_eigenvalues}
	For any $i\in\N_*$, $\lambda_i^\e\to\lambda_i$ as $\e\to 0$.
\end{proposition}
\begin{proof}
By Remark  \ref{rem:conv-eigen-1} and Urysohn's
  subsequence principle, it is enough to prove the convergence along
  sequences $\e_n\to0$ for which $n\mapsto\lambda^i_{\e_n}$ is
  increasing; then, by \eqref{eq:28}, it is not restrictive to assume that $\e\mapsto
  \lambda^i_\e$ is decreasing for any $i\in\N_*$ and admits a limit $\lambda_i^*\leq\lambda_i$ as
  $\e\to 0$. We now prove that, for any $i\in\N_*$,
  $\lambda_i\leq\lambda_i^*$.  We argue by induction
  on $i$. From Lemma \ref{lemma:prelim_conv} we know that
  $\lambda_1^*$ is an eigenvalue of the unperturbed problem so that,
  $\lambda_1^*\geq \lambda_1$. Let us now assume that
\begin{equation}\label{eq:u*induction}
\lambda_j^*=\lambda_j \qquad\text{for all } j=1,\ldots,i-1.
\end{equation}	
Let $\varphi_1^\e,\ldots,\varphi_i^\e$ be a family of perturbed
eigenfunctions as in \eqref{eq:ortonor}. By Lemma
\ref{lemma:prelim_conv} there exist a sequence $\e_n\to0$ as
$n\to\infty$ and functions $u_1^*,\ldots,u_i^*$, that are
eigenfunctions of the unperturbed problem \eqref{eq:eqphi_0}, such
that
\begin{equation}\label{eq:u*convergence}
\varphi_j^{\e_n}\rightharpoonup u_j^* \text{ weakly in }H^1(\Omega)
\quad\text{and}\quad
\varphi_j^{\e_n}\to u_j^* \text{ strongly in }L^2(\Omega),
\end{equation}	
as $n\to\infty$, for every $j=1,\ldots,i$.

On the one hand, by passing to the limit as $n\to\infty$ in
\eqref{eq:ortonor}, we obtain
\begin{equation}\label{eq:delta_kronecker}
\int_\Omega u_j^*u_l^*\dx=\delta_j^l.
\end{equation}
for all $j,l\in \{1,\ldots,i\}$. On the other hand, by
\eqref{eq:u*induction} and \eqref{eq:u*convergence}, for every
$j=1,\ldots,i-1$, $u_j^*$ is a $L^2(\Omega)$-normalized eigenfunction
corresponding to the eigenvalue $\lambda_j$ . Therefore, in view also
of \eqref{eq:delta_kronecker},
$u_i^*\in\mathop{\rm span}\{u_1^*,\dots,u_{i-1}^*\}^\perp$.  From the
iterative variational characterization of the eigenvalues (see
e.g. \cite[Section 11.5]{Jost2013}) we have that
\[
  \lambda_i^*=\int_\Omega\abs{\nabla u_i^*}^2\dx\geq\min_{u\in
    \mathop{\rm span}\{u_1^*,\dots,u_{i-1}^*\}^\perp} \frac{\int_\Omega\abs{\nabla
      u}^2\dx}{\int_\Omega u^2\dx}=\lambda_i
\]
	and this concludes the proof.
\end{proof}

\begin{proposition}\label{prop:conv_eigenfunctions}
  Let $\lambda_i$ be a simple eigenvalue of \eqref{eq:eqphi_0} and let
  $\varphi_i$ be a $L^2(\Omega)$-normalized associated
  eigenfunction. For any $\e\in(0,1)$ let $\lambda_i^\e$ be
  the $i$-th
  eigenvalue of \eqref{eq:eqphiie} and let $\varphi_i^\e$ be a
  $L^2(\Omega)$-normalized associated eigenfunction such that
	\begin{equation}\label{eq:sign_assumption}
		\int_{\Omega}\varphi_i^\e\varphi_i\dx\geq 0.
	\end{equation}
	Then $\varphi_i^\e\to\varphi_i$ in $H^1(\Omega)$ as $\e\to 0$.
\end{proposition}
\begin{proof}
  From Lemma \ref{lemma:prelim_conv} and Proposition
  \ref{propo:conv_eigenvalues} we infer that there exist a sequence
  $\e_n\to0$ as $n\to\infty$ and $\varphi^*\in H^1_0(\Omega)$,
  eigenfunction associated to $\lambda_i$, such that
\begin{equation*}\label{eq:phi*convergence}
  \varphi_i^{\e_n}\rightharpoonup \varphi^* \text{ weakly in }H^1(\Omega)
  \quad\text{and}\quad
  \varphi_i^{\e_n}\to \varphi^* \text{ strongly in }L^2(\Omega)
\end{equation*}	
as $n\to\infty$. In particular, by strong $L^2(\Omega)$-convergence,
$\varphi^*$ is $L^2(\Omega)$-normalized. Being $\lambda_i$ simple, we
have that either $\varphi^*=\varphi_i$ or $\varphi^*=-\varphi_i$. The
assumption \eqref{eq:sign_assumption} rules out the second
possibility, allowing us to conclude that $\varphi^*=\varphi_i$.

Finally, in view of Proposition \ref{propo:conv_eigenvalues},
\[
  \norm{\varphi_i^{\e_n}}_{H^1(\Omega)}^2=\lambda_i^{\e_n}+1\to\lambda_i+1=\norm{\varphi_i}_{H^1(\Omega)}^2
\]
as $n\to\infty$. Hence $\varphi_i^{\e_n}\to\varphi_i$ strongly in
$H^1(\Omega)$ as $n\to\infty$. By Urysohn's subsequence principle the
convergence holds as $\e\to 0$, thus concluding the
proof.
\end{proof}

\section{Limit profiles}\label{sec:limit-profiles}

We now introduce the functions that appear as limit profiles in the
blow-up analysis. From here on, for any $R>0$, we denote by $\eta_R$
a cut-off function such that 
\begin{equation}\label{eq:def_cutoff}
         \eta_R\in C^\infty(\overline{\R^N_+}),\quad 0\leq
          \eta_R\leq 1,\quad
          \abs{\nabla \eta_R}\leq \frac{4}{R},\quad 
          \eta_R(x):=\begin{cases}
            1, &\text{if }\abs{x}\geq R, \\
            0, &\text{if }\abs{x}\leq R/2.
		\end{cases} 	
\end{equation}

\begin{lemma}\label{lemma:def_w}
Let $\Pi$ be a bounded open subset of $\partial\R^N_+$
  such that $0\in\Pi$ and let
  $f\in L^2(\Pi)$. Then there exists a unique function 
  $w=w(f,\Pi)\in\mathcal{D}^{1,2}(\R^N_+\cup \Pi)$ such that
\begin{equation*}
\begin{cases}
-\Delta w=0, &\text{in }\R^N_+, \\
w=0, &\text{on }\partial\R^N_+\setminus\Pi, \\
\partial_{\nnu}w=f, & \text{on }\Pi,
\end{cases}
\end{equation*}
	where $\nnu=(0,\dots,0,-1)$, in a weak sense, that is
\begin{equation}\label{eq:59}
  \int_{\R^N_+}\nabla w\cdot\nabla
  v\dx=\int_{\Pi}fv\dx'\quad\text{for all
  }v\in\mathcal{D}^{1,2}(\R^N_+\cup\Pi).
\end{equation}
\end{lemma}
\begin{proof}
	The result is a direct consequence of the Lax-Milgram Theorem.
\end{proof}

Given $\psi\in C^\infty(\overline{\R^N_+})$ as in \eqref{eq:def_psi_0}, we define
\begin{equation}\label{eq:def_w_0}
	w_{0,\Pi}:=w(-\partial_{\nnu}\psi,\Pi),
\end{equation}
where $w(\cdot,\cdot)$ is defined in Lemma \ref{lemma:def_w}. We point out
that the function $w_{0,\Pi}\in\mathcal{D}^{1,2}(\R^N_+\cup \Pi)$ is the
unique minimizer, among all the possible $u\in
\mathcal{D}^{1,2}(\R^N_+\cup \Pi)$, of the functional $J_\Pi$ defined
in \eqref{eq:57}. 
We denote
\begin{equation}\label{eq:def_m}
  m_{n_0}(\Pi):=J(w_{0,\Pi})=\min_{u\in \mathcal{D}^{1,2}(\R^N_+\cup \Pi)} J_\Pi(u).
\end{equation}
 Since, for any bounded open set $\Pi\subset\partial\R^N_+$,  $\partial_{x_N}\psi\not\equiv0$ on
$\Pi$, we have that $w_{0,\Pi}\not\equiv0$ and hence, choosing $v=w_{0,\Pi}$ in \eqref{eq:59}, we obtain that
\begin{equation}\label{eq:m_int_by_parts}
  m_{n_0}(\Pi)={\frac12}
  \int_{\Pi}w_{0,\Pi}\partial_{\nnu}\psi\dx'=-\frac{1}{2}\int_{\R^N_+}\abs{\nabla w_{0,\Pi}}^2\dx<0.
\end{equation}
 The following lemma is a consequence of the homogeneity of the function $\psi$.
  \begin{lemma}\label{l:scaling-m0}
    For every $r>0$ we have that
    $m_{n_0}(B_r')=r^{N+2\gamma-2}m_{n_0}(B_1')$, with $\gamma$ as in \eqref{eq:notation-n_0}.
  \end{lemma}
  \begin{proof}
    Since
    $\partial_{\nnu}\psi(r
    x')=r^{\gamma-1}\partial_{\nnu}\psi(x')$, a
    scaling argument easily yields that
    \begin{equation*}
      w_{0, B_r'}(x)=r^\gamma w_{0, B_1'}\Big(\frac xr\Big).
    \end{equation*}
Hence, by a change of variable, we obtain that
\begin{align*}
  m_{n_0}(B_r')&=-\frac{1}{2}\int_{\R^N_+}\abs{\nabla w_{0,B_r'}(x)}^2\dx=
  -\frac{1}{2}r^{2(\gamma-1)}\int_{\R^N_+}\abs{\nabla
                 w_{0,B_1'}(x/r)}^2\dx\\
  &=
 -\frac{1}{2}r^{N+2\gamma-2}\int_{\R^N_+}\abs{\nabla
    w_{0,B_1'}(y)}^2\dy=
    r^{N+2\gamma-2}m_{n_0}(B_1'),
\end{align*}
thus concluding that proof.
\end{proof}
Hereafter in this section, we let $\Sigma$ be as in \eqref{eq:Sigma},
  for some $\mathcal V$ satisfying \eqref{eq:7}, \eqref{eq:V-C11} and
  \eqref{eq:V-star-shaped}, and define 
\begin{equation*}
	w_0:=w_{0,\Sigma}
      \end{equation*}
      and 
\begin{equation}\label{eq:def_U}
	U=w_0+\psi.
\end{equation}
One can see that $U\in \psi+\mathcal{D}^{1,2}(\R^N_+\cup \Sigma)$
 weakly satisfies the following boundary value problem
\begin{equation}\label{eq:63}
	\begin{cases}
		-\Delta U=0, &\text{in }\R^N_+, \\
		U=0, &\text{on }\partial\R^N_+\setminus\Sigma, \\
		\partial_{\nnu} U=0, &\text{on }\Sigma,
	\end{cases}
\end{equation}
 i.e. $\int_{\R^N_+}\nabla U\cdot\nabla v\dx=0$ for all
  $v\in \mathcal{D}^{1,2}(\R^N_+\cup \Sigma)$.

The two following existence results Lemma \ref{lemma:U_R} and Lemma
\ref{lemma:Z_R} can be easily proved by standard minimization
methods.  Henceforth we denote, for all $R>0$, 
\begin{equation}\label{eq:notation}
	B_R^+:=B_R\cap \R^N_+,\quad
	\text{and}\quad S_R^+:=\partial B_R\cap \R^N_+.
\end{equation}
\begin{lemma}\label{lemma:U_R}
  Let $\Sigma\sub\partial\R^N_+$ be as in \eqref{eq:Sigma}. For every
  $R>1$ there exists a unique function 
  $U_R\in \psi+H^1_{0,\partial B_R^+\setminus \Sigma}(B_R^+)$
  achieving
\[
  \min\left\{ \|\nabla u\|^2_{L^2(B_R^+)}: \, u\in
    \psi+H^1_{0,\partial B_R^+\setminus \Sigma}(B_R^+) \right\}.
\]
Moreover, $U_R$ weakly solves
\begin{equation}\label{eq:60}
\begin{cases}
  -\Delta U_R=0,  & \text{in } B_R^+, \\
  U_R= \psi,  &\text{on }\partial B_R^+\setminus \Sigma, \\
  \partial_{\nnu} U_R=0, & \text{on } \Sigma,
\end{cases}
\end{equation}
	that is	
\begin{equation}\label{eq:U_R_eq}
\begin{cases}
  U_R-\psi\in
  H^1_{0,\partial B_R^+\setminus \Sigma}(B_R^+),\\
  \int_{B_R^+} \nabla U_R\cdot\nabla\phi\dx=0 \quad\text{for all }
  \phi\in H^1_{0,\partial B_R^+\setminus \Sigma}(B_R^+).
\end{cases}
\end{equation}
\end{lemma}

\begin{lemma}\label{lemma:Z_R}
  Let $\Sigma\sub\partial\R^N_+$ be as in \eqref{eq:Sigma}, $R>2$,
$\eta_R$ as in \eqref{eq:def_cutoff}, and $U$ as in
    \eqref{eq:def_U}. Then there exists a unique
  $Z_R\in \eta_R U+H^1_0(B_R^+)$ achieving
\[
  \min\left\{ \|\nabla u\|^2_{L^2(B_R^+)}: \, u\in \eta_R
    U+H^1_0(B_R^+) \right\}.
\]
Moreover, $Z_R$ weakly solves
\begin{equation}\label{eq:69}
\begin{cases}
  -\Delta Z_R=0,  & \text{in } B_R^+, \\
  Z_R= U,  &\text{on }S_R^+, \\
  Z_R=0, & \text{on } B_R',
\end{cases}
\end{equation}
	that is	
	\begin{equation*}\label{eq:Z_R_eq}
		\begin{cases}
			Z_R\in \eta_R U+H^1_0(B_R^+),\\
			\int_{B_R^+} \nabla Z_R\cdot\nabla\phi\dx=0 
			\quad\text{for all } \phi\in H^1_0(B_R^+).
		\end{cases}
	\end{equation*}
\end{lemma}

The function $U_R$ naturally appears as a limit profile of a scaled
convenient competitor in the estimate of the eigenvalue variation
$\lambda_{n_0}-\lambda_{n_0}^\e$ from below. Indeed, to estimate
$\lambda_{n_0}^\e$ from above with the most precise approximation of
$\lambda_{n_0}$, we are led to test the Rayleigh quotient for
$\lambda_{n_0}^\e$ with test functions obtained by modifying the limit
eigenfunction $\varphi_{n_0}$ (inside balls $B_{R\e}$) into a solution
of the mixed boundary problem, in the less expensive way from the
energetic point of view. By the Dirichlet principle, among all
functions satisfying some prescribed boundary conditions, the energy
is minimized by harmonic ones, so that the above defined harmonic
functions $U_R$ turn out to be the blown-up limit profile of best
competitors. In a similar way the function $Z_R$ is the limit profile
of scaled best competitors for the estimate of the eigenvalue
variation $\lambda_{n_0}-\lambda_{n_0}^\e$ from above.

The function $U_R$, introduced in Lemma \ref{lemma:U_R}, is locally a
good approximation of the limit profile $U$, defined in \eqref{eq:def_U},
for large values of $R$.

\begin{lemma}\label{lemma:U_R_conv}
  For any $r>2$, $U_R\to U$ in $H^1(B_r^+)$ as $R\to+\infty$.
\end{lemma}
\begin{proof}
  Let $R>r$ and $\eta_R$ as in \eqref{eq:def_cutoff}. The function
  $U_R-U\in H^1(B_r^+)$ weakly solves
\[
\begin{cases}
  -\Delta (U_R-U)=0, &\text{in }B_R^+, \\
  U_R-U=0, &\text{on }B_R'\setminus \Sigma, \\
  \partial_{\nnu}(U_R-U)=0, &\text{on }\Sigma, \\
  U_R-U=\psi-U, &\text{on }S_R^+.
\end{cases}
\]	
By the Dirichlet principle, we observe that $\eta_R (\psi-U)\in H^1(B_R^+)$ has higher energy than $U_R-U$ in $B_R^+$. Hence
\begin{align*}
  \int_{B_r^+}\abs{\nabla (U_R-U)}^2\dx
  &\leq
    \int_{B_R^+}\abs{\nabla (U_R-U)}^2\dx\leq
    \int_{B_R^+}\abs{\nabla( \eta_R(\psi-U))}^2\dx \\
  &\leq 2\int_{B_R^+}\abs{\nabla \eta_R}^2\abs{\psi-U}^2\dx+
    2\int_{B_R^+}\eta_R^2\abs{\nabla(\psi-U)}^2\dx \\
  &\leq \frac{32}{R^2}
    \int_{B_R^+\setminus B_{R/2}^+}\abs{\psi-U}^2\dx+{2}
    \int_{B_R^+\setminus B_{R/2}^+}\abs{\nabla(\psi-U)}^2\dx \\
  &\leq {32}\int_{B_R^+\setminus
    B_{R/2}^+}
    \frac{\abs{\psi-U}^2}{\abs{x}^2}\dx+{2}
    \int_{B_R^+\setminus B_{R/2}^+}\abs{\nabla(\psi-U)}^2\dx .
\end{align*}
The last two terms vanish as $R\to+\infty$  thanks to the validity of the Hardy
inequality and the fact that 
$\psi-U\in\mathcal{D}^{1,2}(\R^N_+\cup\Sigma)$. Since $B_r^+\setminus \Sigma$ has
positive $(N-1)$-dimensional measure, by Poincaré inequality we may
conclude the proof.
\end{proof}

\section{An equivalent problem on a domain with a flat crack}\label{sec:an-equiv-probl}

Sections \ref{sec:an-equiv-probl} to
  \ref{sec:proof_main} are devoted to the development of the
  monotonicity formula and the consequent energy and eigenvalue
  estimates needed to prove Theorem \ref{thm:main} and requiring regularity and star-shapedness assumptions on the open set
  $\mathcal V$.
Then throughout Sections \ref{sec:an-equiv-probl} to
\ref{sec:proof_main} we tacitly assume hypotheses
\eqref{eq:7}, \eqref{eq:V-C11}, and \eqref{eq:V-star-shaped} on
$\mathcal V$, besides the assumptions 
on $\Omega$ set out in the Introduction;
  consequently we let 
$\Sigma_\e$ be as in \eqref{eq:sigmaeps} and $\Sigma$ as in \eqref{eq:Sigma}. 

In the present section, we first introduce an equivalent problem,
obtained by straightening the boundary of $\Omega$ locally around the
origin. Then, we prove that the star-shapedness of the Neumann region
is preserved by such a transformation, see Section
\ref{subsec:starshaped}.

\subsection{Flattening the boundary of the domain}
In this section, we consider a particular diffeomorphism straightening
the boundary of $\Omega$ near $0$, first introduced in
\cite{Adolfsson1997}, see also \cite{FNO1}.  Let
$g\in C^{1,1}(\R^{N-1})$ be as in \eqref{eq:1}. Let
$\zeta\in C_c^\infty(\R^{N-1})$ be such that $\supp \zeta\subset B_1'$,
$\zeta\geq 0$ in $\R^{N-1}$, $\int_{\R^{N-1}}\zeta(y')\,dy'=1$
 (see \eqref{eq:72} for the notation $B_{1}'$).
For every $\delta>0$ we consider the mollifier
	\begin{equation*}
	\zeta_{\delta}(y')= \delta^{-N+1}\zeta\left(\frac{y'}{\delta}\right).
	\end{equation*}
	For all $j=1,\dots,N-1$,   we define
	\[
G_j(y',y_N) = \begin{cases}
\left( \zeta_{y_N}\star \dfrac{\partial
            g}{\partial y_j} \right)(y'),&\text{if }y'\in\R^{N-1},\
        y_N>0,\\[10pt]
\dfrac{\partial
            g}{\partial y_j} (y'),&\text{if } y'\in\R^{N-1},\
        y_N=0,
      \end{cases}
\]
where $\star$ denotes the convolution product.  We observe that, for
all $j=1,\dots,N-1$, $G_j\in C^\infty(\R^N_+)$, $G_j$ is Lipschitz
continuous in $\overline{\R^N_+}$, and
$\frac{\partial G_j}{\partial y_i}\in L^\infty(\R^N_+)$ for every
$i\in\{1,\dots,N\}$.  Furthermore we have that, for all
$j=1,\dots,N-1$ and $i=1,\dots,N$,
\[
y_N \frac{\partial
            G_j}{\partial y_i}\quad\text{is Lipschitz continuous in
$\overline{\R^N_+}$}.
\]
Let,  for every $j=1,\dots,N-1$,
	\begin{equation*}
\widetilde G_j:\R^N\to\R,\quad \widetilde G_j(y',y_N):=G_j(y',|y_N|)
	\end{equation*}
and 	\begin{equation*}
F_j:\R^N\to\R,\quad 	F_j(y',y_N)=y_j-y_N\widetilde G_j(y',y_N).
	\end{equation*}
        It follows that $\widetilde G_j$ is Lipschitz continuous in $\R^N$ and
        $F_j$ belongs to $C^{1,1}(\R^N)$ (i.e. it is continuously
        differentiable with Lipschitz gradient) for all
        $j=1,\dots,N-1$.

        In particular, defining 
        $\widetilde G:\R^N\to\R^{N-1}$ as 
\[
\widetilde G(y',y_N)=( \widetilde G_1(y',y_N), \widetilde G_2(y',y_N),\dots,
        \widetilde G_{N-1}(y',y_N)),
\]
 we have that 
\[
J_{\widetilde G}\in L^\infty(\R^N,\R^{N(N-1)}),
\]
where $J_{\widetilde G}(y',y_N)$ is the Jacobian matrix of $\widetilde G$ at $(y',y_N)$, and 
\begin{equation}\label{eq:3}
|\widetilde G(y',y_N)-\nabla g(y')|\leq C\,
|y_N|\quad\text{for all }(y',y_N)
\in  \R^N,
\end{equation}
for some constant $C>0$ independent of $(y',y_N)$.

Let $F: \R^N\to\R^N$ be defined as follows
\begin{equation}\label{eq:F_def}
  F(y',y_N):=(F_1(y',y_N),\dots,F_{N-1}(y',y_N),y_N+g(y')).
\end{equation}
Using the above function $F$, we are going to construct a
diffeomorphism which straightens the boundary. To prove Lemma
\ref{l:ay-nu}, which will be crucial in the monotonicity argument, we
will need a quite precise quantification of the behaviour of all
entries of the Jacobian matrix of $F$. Hence, by direct computations
and \eqref{eq:3} we have that
\begin{align*}
  J_{F}(y',y_N)&=		\begin{pmatrix}
    1-y_N\frac{\partial \widetilde G_1}{\partial y_1} &
    -y_N\frac{\partial \widetilde G_1}{\partial y_2} & 
    \cdots & -y_N\frac{\partial \widetilde G_1}{\partial y_{N-1}} &
    -\widetilde G_1
    -y_N \frac{\partial
      \widetilde G_1}{\partial y_N}
    \\[5pt]
    -y_N\frac{\partial \widetilde G_2}{\partial y_1} &
    1-y_N\frac{\partial \widetilde G_2}{\partial y_2} &
    \cdots & -y_N\frac{\partial \widetilde G_2}{\partial
      y_{N-1}}
    & - \widetilde G_2 -y_N \frac{\partial
      \widetilde G_2}{\partial y_N} \\
    \vdots & \vdots & \ddots & \vdots & \vdots \\[5pt]
    -y_N\frac{\partial \widetilde G_{N-1}}{\partial y_1} &
    -y_N\frac{\partial \widetilde G_{N-1}}{\partial y_2} &
    \cdots & 1-y_N\frac{\partial \widetilde G_{N-1}}{\partial
      y_{N-1}}
    & - \widetilde G_{N-1} -y_N \frac{\partial
      \widetilde G_{N-1}}{\partial y_N} \\[5pt]
    \frac{\partial g}{\partial y_1}(y') &
    \frac{\partial g}{\partial y_2}(y') &
    \cdots & \frac{\partial g}{\partial y_{N-1}}(y') & 1
  \end{pmatrix}\\[5pt]
               &= \left( \renewcommand{\arraystretch}{1.5}
 \begin{array}{c|c}
    I_{N-1}-y_NJ_{\widetilde G}&-\nabla g(y')+O(y_N)\\\hline
(\nabla g(y'))^T&1
  \end{array}\right)
,
\end{align*}
where $I_{N-1}$ is the $(N-1)\times(N-1)$ identity
  matrix,
  $\nabla g(y')$ is a column vector and $(\nabla g(y'))^T$ denotes
  its transpose; henceforth, the notation $O(y_N)$ will be used to
  denote blocks of matrices with all entries being $O(y_N)$ as
  $y_N\to 0$ uniformly with respect to $y'$.

From \eqref{eq:2} and the assumption that $g\in C^{1,1}(\R^{N-1})$ we
deduce that $\nabla g(y')=O(|y'|)$ as $|y'|\to0$, so that
\begin{equation}\label{eq:det-jac}
\det 	J_F(y',y_N)=1+|\nabla g(y')|^2+O(y_N)=
1+O(|y'|^2)+O(y_N)
\end{equation}
as $y_N\to0$ and $|y'|\to0$.

In particular, $\det J_{F}(0)=1\ne0$. Hence, by the Inverse Function
Theorem, $F$ is invertible in a neighbourhood of the origin, i.e.
there exists $r_1\in (0,r_0)$ such that $F$ is a diffeomorphism of
class $C^{1,1}$ from $B_{r_1}$ to $\mathcal U=F(B_{r_1})$ for some
$\mathcal U$ open neighbourhood of $0$. Moreover, it is possible to
choose $r_1$ sufficiently small so that
\begin{equation}\label{eq:r1_def}
     F^{-1}(\mathcal U\cap \Omega)=\R^N_+\cap  B_{r_1}= B_{r_1}^+ \qquad\text{and}\qquad
	F^{-1}(\mathcal U\cap \partial\Omega)=\partial\R^N_+\cap B_{r_1}=B_{r_1}',
\end{equation}
i.e. near the origin the image of $\Omega$ through $F^{-1}$
        has flat boundary (lying on $\partial\R^N_+$). Let 
\begin{equation}\label{eq:6}
\Phi:\mathcal U\to B_{r_1},\quad   \Phi:=F^{-1}.
\end{equation}
 From the fact that 
\begin{equation*}
  \Phi\in C^{1,1}(\mathcal U,B_{r_1}), \quad
  \Phi^{-1}\in C^{1,1}(B_{r_1},\mathcal U), \quad
 \Phi(0)=\Phi^{-1}(0)=0, \quad J_\Phi(0)=J_{\Phi^{-1}}(0)=I_N,
\end{equation*}
it follows that 
\begin{align}
\label{eq:11}&  J_\Phi(x)=I_N+O(|x|)\quad\text{and}\quad 
\Phi(x)=x+O(|x|^2)\quad\text{as }|x|\to0,\\
\label{eq:10}
&  J_{\Phi^{-1}}(x)=I_N+O(|x|)\quad\text{and}\quad 
\Phi^{-1}(x)=x+O(|x|^2)\quad\text{as }|x|\to0.
\end{align}
Let $i\in\N_*$ and $\bar\e\in(0,1)$ be such that
  $\e\mathcal V\subset \mathcal U$ for all $\e\in(0,\bar \e)$ (so that
  $\Phi(\e\mathcal V)\subset B_{r_1}$ for all $\e\in(0,\bar \e)$). For $y\in \Phi(\mathcal
U\cap\Omega)=B_{r_1}^+$, we define
\[
u_i^\e(y):=\varphi_i^\e(\Phi^{-1}(y))=\varphi_i^\e(F(y)).
\]
From \eqref{eq:eqphiie} we deduce that 
	\begin{equation}\label{eq:uie}
\int_{B_{r_1}^+}A(y)\nabla u_i^\e(y)\cdot\nabla\phi(y)\dy=\lambda_i^\e\int_{B_{r_1}^+}p(y) u_i^\e(y)\phi(y)\dy
	\end{equation}
	for all $\phi\in 
H^1_{0,\partial B_{r_1}^+\setminus \widetilde\Sigma_\e}(B_{r_1}^+)$,
where 
\begin{equation}\label{eq:sigmatilde}
\widetilde\Sigma_\e=\Phi(\Sigma_\e)=
\Phi((\e\mathcal V)\cap\partial\Omega)
=\Phi(\e\mathcal V)\cap\partial\R^N_+, \quad 	p(y)=|\det J_{F}(y)|,
\end{equation}
and
\begin{equation}\label{eq:A}
	A(y)=(J_{F}(y))^{-1}((J_{F}(y))^{-1})^T|\det J_{F}(y)|.
\end{equation}
Notice that
  $u_i^\e \in H^1_{0,B_{r_1}' \setminus
    \widetilde\Sigma_\e}(B_{r_1}^+)$ for every
  $\e\in(0,\bar \e)$.  We
observe that \eqref{eq:uie} is the weak formulation of the problem
\begin{equation*}
  \begin{cases}
    -\dive(A(y)\nabla  u_i^\e(y))=\lambda_i^\e \,p(y) u_i^\e(y), &\text{in }B_{r_1}^+,\\
    u_i^\e=0,&\text{on }\widetilde \Gamma_{\e,r_1},\\
    A(y) \nabla u_i^\e(y)\cdot{\nnu}=0,&\text{on
    }\widetilde\Sigma_\e,
  \end{cases}
\end{equation*}
where 
\[
\widetilde \Gamma_{\e,r_1} =\overline{B_{r_1}'}\setminus\widetilde\Sigma_\e
\]
and $\nnu$ is the exterior unit normal vector on
$\widetilde\Sigma_\e$, which is equal to $-\bm{e}_N$ with $\bm{e}_N=(0,0,\dots,1)$
since $\widetilde\Sigma_\e\subset \partial\R^N_+$. 

From \eqref{eq:A} it follows directly that $A$ is symmetric.
Moreover, by direct computations we have that 
\[
A(y)=\alpha(y) B(y)(B(y))^T
\]
where, taking into account \eqref{eq:det-jac}, 
\begin{equation}\label{eq:alpha}
\alpha(y)=\frac{1}{|\det J_{F}(y)|}=
1+O(|y'|^2)+O(y_N)
\end{equation}
as $y_N\to0$ and $|y'|\to0$, and 
\[
B=
\left( \renewcommand{\arraystretch}{1.8}
 \begin{array}{c|c}
{\scriptsize\begin{matrix}
	1+\mathop{\sum}\limits_{j\neq1}\big|\frac{\partial g}{\partial
          y_j}\big|^2\!+\!O(y_N) &
-\frac{\partial g}{\partial y_1}\frac{\partial g}{\partial y_2}\!+\!O(y_N)& 
\cdots & -\frac{\partial g}{\partial y_1}\frac{\partial g}{\partial y_{N-1}}\!+\!O(y_N)
 \\[-3pt]
-\frac{\partial g}{\partial y_2}\frac{\partial g}{\partial
  y_1}\!+\!O(y_N)&	
1+\mathop{\sum}\limits_{j\neq2}\big|\frac{\partial g}{\partial
          y_j}\big|^2\!+\!O(y_N)& 
\cdots & -\frac{\partial g}{\partial y_2}\frac{\partial g}{\partial y_{N-1}}\!+\!O(y_N)\\[-3pt]
	\vdots & \vdots & \ddots & \vdots \\[-3pt]
-\frac{\partial g}{\partial y_{N-1}}\frac{\partial g}{\partial
  y_1}\!+\!O(y_N)	
& -\frac{\partial g}{\partial y_{N-1}}\frac{\partial g}{\partial y_2}\!+\!O(y_N)&
\cdots & 1+\mathop{\sum}\limits_{j\neq{N-1}}\big|\frac{\partial g}{\partial
          y_j}\big|^2\!+\!O(y_N)
\end{matrix} }
&\nabla g+O(y_N)\\\hline
-(\nabla g)^T+O(y_N)&1+O(y_N)
  \end{array}\right).
\]
Hence we have that 
\begin{equation}\label{eq:5}
A(y',y_N)=\alpha(y',y_N)
\left( \renewcommand{\arraystretch}{1.5}
 \begin{array}{c|c}
    I_{N-1}+O(|y'|^2)+O(y_N)&O(y_N)\\\hline
O(y_N) &1 +O(|y'|^2)+O(y_N)
  \end{array}\right)
\end{equation}
where here $O(y_N)$, respectively $O(|y'|^2)$, denotes
blocks of matrices with all entries being $O(y_N)$ as $y_N\to 0$, 
respectively $O(|y'|^2)$ as $|y'|\to 0$.

From \eqref{eq:alpha} and \eqref{eq:5} it follows that, if $r_1$ is
chosen sufficiently small, $\alpha(y)>0$ and $A$ is uniformly
elliptic in $B_{r_1}^+$. 
Moreover, by \eqref{eq:A} and the fact that $F\in
C^{1,1}(B_{r_1},\R^N)$, we have that, if we denote
$A(y)=(a_{i,j}(y))_{i,j=1,\dots,N}$, then 
\begin{equation}\label{eq:4}
a_{i,j}\in
C^{0,1}(B_{r_1}^+\cup B_{r_1}'),
\end{equation}
while \eqref{eq:5} ensures that
	\begin{equation}\label{eq:A_coeff}
	a_{i,N}(y',0)=a_{N,i}(y',0)= 0 \quad\text{for all }i=1,\dots,N-1.
	\end{equation}
The even reflection through the hyperplane $\{y_N=0\}$  of $u_i^\e$ 
\begin{equation}\label{eq:veps_def}
v_i^\e(y',y_N)=u_i^\e(y',|y_N|), \quad (y',y_N)\in B_{r_1},
\end{equation}
belongs to $H^1_{0, \widetilde \Gamma_{\e,r_1} }(B_{r_1})$ and 
\begin{equation}\label{eq:vie}
  \int_{B_{r_1}}\widetilde A(y)\nabla
  v_i^\e(y)\cdot\nabla\phi(y)\dy=\lambda_i^\e\int_{B_{r_1}}\widetilde p(y) v_i^\e(y)\phi(y)\dy
\end{equation}
	for all $\phi\in 
H^1_{0,\partial B_{r_1}\cup \widetilde \Gamma_{\e,r_1}}(B_{r_1})$,
i.e. $v_i^\e$ weakly solves 
\begin{equation}\label{eq:reflected}
  \begin{cases}
  -\dive (\widetilde{A}(y)\nabla v_i^\e(y))=\lambda_i^\e
  \widetilde p(y) v_i^\e(y),&\text{in }B_{r_1}\setminus \widetilde \Gamma_{\e,r_1},\\
  v_i^\e=0,&\text{on }\widetilde \Gamma_{\e,r_1},
\end{cases}
\end{equation}
where
\begin{equation}\label{eq:tilde}
\widetilde p(y',y_N)=\begin{cases}
  p(y',y_N), &\text{if }y_N\geq0, \\
  p(y',-y_N), &\text{if }y_N<0,
\end{cases}\quad\text{and}\quad	
\widetilde{A}(y',y_N):=\begin{cases}
  A(y',y_N), &\text{if }y_N>0, \\
  QA(y',-y_N) Q, &\text{if }y_N<0,
\end{cases}
\end{equation}
	with
\[
	Q:=
\left( \renewcommand{\arraystretch}{1.2}
 \begin{array}{c|c}
    I_{N-1}&
             \begin{matrix}
               0\\\vdots\\0
             \end{matrix}
\\\hline
   \begin{array}{cccc}
     0&  \dots &0
   \end{array}
&-1
  \end{array}\right).
\]
We observe that \eqref{eq:4} and \eqref{eq:A_coeff}
  ensure that the coefficients of the matrix $\widetilde{A}$ are
  Lipschitz continuous in $B_{r_1}$.

\begin{remark}\label{rem:Mosco}
  From \eqref{eq:11}--\eqref{eq:10} we easily deduce that, as
  $\e\to0$,
  $\R^{N}\setminus
  \big(\partial\R^N_+\setminus\big(\frac1\e\widetilde\Sigma_\e\big)\big)$
  converges in the sense of Mosco (see \cite{Daners2003,Mosco1969}) to
  the set $\R^{N}\setminus \big(\partial\R^N_+\setminus\Sigma\big)$,
  where $\Sigma$ is defined in \eqref{eq:Sigma}.  In particular, for
  every $R>0$, the weak limit points in $H^1(B_R^+)$ as $\e\to 0$ of
  a family of functions $\{w_\e\}_\e$ with
  $w_\e\in H^1_{0,\partial B_{R}^+\setminus (\frac1\e\widetilde
    \Sigma_\e)}(B_{R}^+)$ belong to
  $H^1_{0,\partial B_{R}^+\setminus \Sigma}(B_{R}^+)$.
\end{remark}

\subsection{Star-shapedness of the transformed domains}\label{subsec:starshaped}

We are interested in proving that the star-shapedness of the set
$\Sigma_\e$ is preserved under the transformation $\Phi$ introduced in
\eqref{eq:6}. To this aim, we first observe that $\Phi(\e\mathcal V)$
is strictly star-shaped provided $\e$ is sufficiently small.

\begin{lemma}\label{l:eV-star-shaped}
  Let $\mathcal V$ be a bounded open set in $\R^N$ satisfying
  \eqref{eq:7}, \eqref{eq:V-C11}, and \eqref{eq:V-star-shaped}.
Then
\begin{enumerate}[\rm (i)]
\item there exists a function $\rho:\S^{N-1}\to\R$ of class $C^{1,1}$
such that $\rho\geq 0$, 
\begin{equation}\label{eq:16}
\mathcal V=\{r\theta:\theta\in
\S^{N-1}\text{ and }0\leq r< \rho(\theta)\},
\quad \text{and}\quad  \partial \mathcal V=\{\rho(\theta)\theta:\theta\in
\S^{N-1}\};
\end{equation}
\item letting $\Phi$ be as in
  \eqref{eq:6}, $\Phi(\e\mathcal V)$ is strictly star-shaped with
  respect to $0$ provided $\e$ is sufficiently small.
\end{enumerate}
\end{lemma}
\begin{proof}
 (i) Let us define, for all $\theta\in\S^{N-1}$,
    $\rho(\theta)=\sup\{r>0:r\theta\in\mathcal V\}$. Assumption
    \eqref{eq:V-star-shaped} implies that \eqref{eq:16} is satisfied.
 To show that the function $\rho$ is of class $C^{1,1}$ we use the Implicit
Function Theorem. To this aim, let us fix $P_0\in\partial\mathcal V$,
so that $P_0=\rho(\theta_0)\theta_0$ for some
$\theta_0\in\S^{N-1}$. Up to a rotation, it is not restrictive to
assume that $\theta_0=\bm{e_N}=(0,0,\dots,1)$. 
Let $\widetilde\rho:B'_{1}\to \R$,
$\widetilde\rho(x')=\rho(x',\sqrt{1-|x'|^2})$. We observe that
$\widetilde\rho(0)=\rho(\theta_0)$ and that $\widetilde\rho$ is the
composition of $\rho$ with a smooth local parametrization of
$\S^{N-1}$ near $0$; hence $\rho$ is of class $C^{1,1}$
in a neighbourhood $\theta_0$ if and only if 
$\widetilde\rho$ is of class $C^{1,1}$
in a neighbourhood $0$. Since $\mathcal V$ is of class $C^{1,1}$,
there exist $\delta>0$ and a $C^{1,1}$--function
$\varphi:\R^{N-1}\to\R$ such that 
\begin{align*}
&\mathcal V\cap B_\delta(P_0)=\{(x',x_N)\in
B(P_0,\delta):x_N<\varphi(x')\}, \\
&\partial\mathcal V\cap B_\delta(P_0)=\{(x',x_N)\in
B_\delta(P_0):x_N=\varphi(x')\}.
\end{align*}
Let $H:(0,+\infty)\times B'_1\to\R$,  $H(r,x')=\varphi(r
x')-r\sqrt{1-|x'|^2}$. We have that $H$ is of class $C^{1,1}$ in a
neighbourhood $(\rho_0,0)$, $H( \rho_0,0)=0$ and $\frac{\partial
  H}{\partial r}(\rho_0,0)=-1\neq0$; furthermore 
\[
H(\widetilde
\rho(x'),x')=0\quad\text{for $|x'|$ small}.
\]
 By the Implicit Function Theorem we
can then conclude that $\widetilde\rho$ is of class $C^{1,1}$ in a neighbourhood~$0$.

\smallskip \noindent 
(ii)
Since $\mathcal V$ is of class $C^{1,1}$, there exists a function
$G\in C^{1,1}(\R^N)$ such that 
\[
\mathcal V=\{x\in \R^N:G(x)<0\},\quad 
\partial\mathcal V=\{x\in \R^N:G(x)=0\}\quad\text{and}\quad
\nabla G(x)\neq0\text{ for all $x\in\partial\mathcal V$.}
\]
In particular we have that $\nnu(x)=\frac{\nabla G(x)}{\|\nabla
  G(x)\|}$ for all $x\in\partial\mathcal V$, hence
assumption \eqref{eq:V-star-shaped} can be reformulated as follows:
\begin{equation}\label{eq:9}
\text{there exists
$\tilde \sigma>0$ such that 
 }x\cdot\nabla G(x)\geq\tilde \sigma  \quad\text{for every }x\in\partial \mathcal V.
\end{equation}
 We observe that 
\begin{equation}\label{eq:18}
\partial\left(\Phi(\e\mathcal
V)\right)=\{x\in\R^N:G_\e(x)=0\},
\end{equation}
where 
\[
G_\e(x)=G\bigg(\frac{\Phi^{-1}(x)}{\e}\bigg).
\]
From \eqref{eq:10} we deduce that (by choosing $r_1$ smaller if necessary) there exists some positive constant
$C>0$ independent of $\e$ such that 
\begin{equation}\label{eq:12}
|x|\leq C\e\quad\text{for all $x\in \partial\left(\Phi(\e\mathcal
V)\right)$}.
\end{equation}
Since the exterior unit normal at $x\in\partial\left(\Phi(\e\mathcal
V)\right)$ has the same direction as $\nabla G_\e(x)$, 
to prove assertion (ii) we have to show that, if $\e$ is sufficiently small, 
\begin{equation}\label{eq:8}
 \inf_{x\in\partial\left(\Phi(\e\mathcal
V)\right) } x\cdot \nabla G_\e(x)>0.
\end{equation}
From \eqref{eq:10}, \eqref{eq:9}, and \eqref{eq:12} it follows that, for all $x\in \partial\left(\Phi(\e\mathcal
V)\right)$, 
\begin{align*}
   x\cdot \nabla G_\e(x)&=
\frac{x}\e\cdot \nabla G\bigg(\frac{\Phi^{-1}(x)}{\e}\bigg)
  J_{\Phi^{-1}}(x)\\
&=\left(\frac{\Phi^{-1}(x)}{\e}+\frac{x-\Phi^{-1}(x)}{\e}\right)\cdot \nabla G\bigg(\frac{\Phi^{-1}(x)}{\e}\bigg)
  (I_N+O(|x|)) \\
&=\frac{\Phi^{-1}(x)}{\e}\cdot \nabla
  G\bigg(\frac{\Phi^{-1}(x)}{\e}\bigg)+\frac1\e\,O(|x|^2)
\geq \tilde \sigma-O(\e)>\frac{\tilde\sigma}2
\end{align*}
 provided $\e$ is sufficiently small, thus proving claim \eqref{eq:8}.
\end{proof}
We now prove that sections of strictly star-shaped sets are strictly
star-shaped.
\begin{lemma}\label{l:section}
Let $\omega$ be a $C^{1,1}$  bounded open set in $\R^N$ such that
$0\in \omega$ and 
$\omega$ is 
strictly star-shaped with respect to the origin (i.e. $\omega$
satisfies \eqref{eq:V-star-shaped}). Then the set $\omega'=\{x'\in\R^{N-1}:(x',0)\in \omega\}$ is a $C^{1,1}$ strictly
star-shaped open  subset of $\R^{N-1}$.  
\end{lemma}
\begin{proof}
Since $\omega$ is of class $C^{1,1}$, there exists 
$\tilde G\in C^{1,1}(\R^N)$ such that $\omega=\{x\in \R^N: \tilde G (x)<0\}$ and
$\partial\omega=
\{x\in \R^N: \tilde G (x)=0\}$. Then 
\[
\omega'=\{x'\in  \R^{N-1}:\tilde g(x')<0\}
\]
where
$\tilde g (x')=\tilde G (x',0)$,
$\tilde g\in C^{1,1}(\R^{N-1})$. We claim that
\begin{equation}\label{eq:13}
  \partial \omega'=\{x'\in\R^{N-1}:(x',0)\in\partial\omega\}.
\end{equation}
It is easy to verify that, if $x'\in  \partial \omega'$, then
$(x',0)\in\partial \omega$. Thus, to prove claim \eqref{eq:13} it
is enough to show that, if 
$(x',0)\in \partial\omega$, then $x'\in\partial\omega'$. To
this aim, the assumption of strict star-shapedness of $\omega$ plays a crucial
role. Let $(x',0)\in \partial\omega$. Then $\tilde G (x',0)=0$;
hence for every $n\in\N_*$ there  exists $\xi_n\in\big[0,\frac1n\big]$ such that 
\[
\tilde G\big(\big(1-\tfrac1n\big)x',0\big)=\nabla \tilde G
((1-\xi_n)x',0)\cdot(-\tfrac{x'}n,0)
=-\tfrac1n\, \Big(\nabla \tilde G (x',0)\cdot (x',0)+o(1)\Big)
\]
as $n\to+\infty$. The assumption that $\omega$ is 
strictly star-shaped with respect to the origin yields that $\nabla
\tilde G (x',0)\cdot (x',0)>0$, hence we conclude that
$\tilde G\big(\big(1-\tfrac1n\big)x',0\big)<0$ for $n$ sufficiently large, so
that $\big(1-\tfrac1n\big)x'\in \omega'$ and
$\big(1-\tfrac1n\big)x'\to x'$ as $n\to+\infty$. In a similar way, we
can prove that $\big(1+\tfrac1n\big)x'\not\in \omega'$ and
$\big(1+\tfrac1n\big)x'\to x'$ as $n\to+\infty$. Hence we conclude
that $x'\in\partial\omega'$, thus proving claim \eqref{eq:13}.

From \eqref{eq:13} it follows that 
\begin{equation*}
  \partial \omega'=\{x'\in\R^{N-1}: \tilde g (x')=0\}.
\end{equation*}
We observe that, for $x'\in \partial  \omega'$,
\begin{equation}\label{eq:14}
\nabla \tilde g (x')\cdot x'=\nabla \tilde G (x',0) \cdot (x',0)>0
\end{equation}
by strict star-shapedness of $\omega$. In particular $\nabla
\tilde g (x')\neq 0$ for all $x'\in \partial \omega'$, hence by 
the  Implicit Function Theorem the boundary of $\omega'$ 
can be locally parametrized as the graph of a $C^{1,1}$-function,
i.e. $\omega'$ is of class $C^{1,1}$. The strict star-shapedness
of $\omega'$ directly follows from \eqref{eq:14}. 
\end{proof}

From Lemmas  \ref{l:eV-star-shaped} and \ref{l:section} we may directly
conclude the following result.
\begin{corollary}\label{cor:starshaped}
The set $\widetilde\Sigma_\e$ defined in
\eqref{eq:sigmatilde} is of class $C^{1,1}$ and 
strictly star-shaped in $\R^{N-1}$ with respect to $0$ for $\e$ sufficiently small.
\end{corollary}

Corollary \ref{cor:starshaped} achieves our aim of proving
preservation of star-shapedness after the action of the diffeomorphism
$\Phi$. The following Lemma \ref{l:rho-eps} provides a quantitative
estimate of the size of the transformed Neumann region, proving that
its size remains of order $\e$.

From Lemma \ref{l:eV-star-shaped} (ii) we have that there exists
$\e_0\in(0,1)$ such that  $\Phi(\e\mathcal
V)$ is strictly star-shaped with respect to $0$ for all $\e\in(0,\e_0)$. Then, applying  Lemma
\ref{l:eV-star-shaped} (i) to $\Phi(\e\mathcal V)$, for all
$\e\in(0,\e_0)$ there exists a function $\rho_\e:\S^{N-1}\to\R$ of class $C^{1,1}$
such that $\rho_\e\geq 0$, 
\begin{equation}\label{eq:rho-eps}
\Phi(\e\mathcal V)=\{r\theta:\theta\in
\S^{N-1}\text{ and }0\leq r< \rho_\e(\theta)\},
\quad \text{and}\quad  \partial (\Phi(\e\mathcal V))=\{\rho_\e(\theta)\theta:\theta\in
\S^{N-1}\}.
\end{equation}

\begin{lemma}\label{l:rho-eps}
  For every $\e\in(0,\e_0)$, let $\rho_\e$ be as in
  \eqref{eq:rho-eps}. Then there exist $\e_1\in(0,\e_0)$ and 
 $\kappa>1$ (independent of $\e$) such that, for all $\e\in(0,\e_1)$, 
\begin{equation}\label{eq:15}
\frac\e \kappa\leq \rho_\e(\theta)\leq \kappa\e\quad\text{for all
}\theta\in\S^{N-1}
\end{equation}
and
\begin{equation}\label{eq:17}
|\nabla_{\S^{N-1}}\rho_\e(\theta)|\leq \kappa\e\quad\text{for all
}\theta\in\S^{N-1}.
\end{equation}
\end{lemma}
\begin{proof}
Estimate \eqref{eq:15} follows from \eqref{eq:11} and the fact that,
being $\mathcal V$ a bounded open set containing $0$,
$0<\inf_{x\in\partial\mathcal V}|x|\leq\sup_{x\in\partial\mathcal V}|x|<+\infty$.

To prove \eqref{eq:17}, we observe that, by \eqref{eq:rho-eps} and \eqref{eq:18},
\begin{equation*}
 \partial (\Phi(\e\mathcal
  V))=\left\{x\in\R^N\setminus\{0\}:|x|-\rho_\e\bigg(\frac{x}{|x|}\bigg)=0\right\}
= \left\{x\in\R^N\setminus\{0\}:G\bigg(\frac{\Phi^{-1}(x)}{\e}\bigg)=0\right\},
\end{equation*}
with $G$ being as in the proof of Lemma \ref{l:eV-star-shaped}.
Therefore, for every $\e\in(0,\e_0)$ and $x\in \partial (\Phi(\e\mathcal
  V))$, there exists $c_\e(x)\in\R$ such that 
\begin{equation}\label{eq:19}
\frac{1}\e\, \nabla G\bigg(\frac{\Phi^{-1}(x)}{\e}\bigg)
  J_{\Phi^{-1}}(x)=c_\e(x)\left(\frac{x}{|x|}-\frac1{|x|}\nabla_{\S^{N-1}}\rho_\e\bigg(\frac{x}{|x|}\bigg)\right).
\end{equation}
Hence, from \eqref{eq:19}, \eqref{eq:9}, \eqref{eq:10}, and \eqref{eq:15} we deduce 
\begin{align}\label{eq:20}
c_\e(x)|x|&= \nabla G\bigg(\frac{\Phi^{-1}(x)}{\e}\bigg)
  J_{\Phi^{-1}}(x)\cdot \frac{x}\e\\
\notag&=\nabla G\bigg(\frac{\Phi^{-1}(x)}{\e}\bigg)\cdot \frac{\Phi^{-1}(x)}{\e}+
\frac1\e\,\nabla G\bigg(\frac{\Phi^{-1}(x)}{\e}\bigg)\cdot
\left(
  J_{\Phi^{-1}}(x)x-\Phi^{-1}(x)\right)\\
\notag&\geq\tilde \sigma +
\frac1\e\,O(|x|^2)=\tilde\sigma +O(\e)>\frac{\tilde\sigma}{2}
\quad\text{for all $x\in \partial (\Phi(\e\mathcal
  V))$}
\end{align}
if $\e$ is sufficiently small.
On the other hand, multiplying both sides of \eqref{eq:19} by
$\nabla_{\S^{N-1}}\rho_\e\big(\tfrac{x}{|x|}\big)$ we obtain that 
\begin{equation*}
\frac{1}\e\, \nabla G\bigg(\frac{\Phi^{-1}(x)}{\e}\bigg)
  J_{\Phi^{-1}}(x)\cdot
  \nabla_{\S^{N-1}}\rho_\e\bigg(\frac{x}{|x|}\bigg)
=-\frac{c_\e(x)}{|x|}\bigg|\nabla_{\S^{N-1}}\rho_\e\bigg(\frac{x}{|x|}\bigg)\bigg|^2
\end{equation*}
and hence, in view of \eqref{eq:20}, 
\begin{align*}
\frac{\tilde\sigma}{2|x|^2}\bigg|\nabla_{\S^{N-1}}\rho_\e\bigg(\frac{x}{|x|}\bigg)\bigg|^2
\leq
  \frac{|c_\e(x)|}{|x|}\bigg|\nabla_{\S^{N-1}}\rho_\e\bigg(\frac{x}{|x|}\bigg)\bigg|^2\leq
  {\rm const\,}\frac1\e\, \bigg|\nabla_{\S^{N-1}}\rho_\e\bigg(\frac{x}{|x|}\bigg)\bigg|
\end{align*}
for all $x\in \partial (\Phi(\e\mathcal
  V))$ and for some ${\rm const\,}>0$ independent of $x$ and
  $\e$. Therefore, taking into account \eqref{eq:15}, we have that,
for all $x\in \partial (\Phi(\e\mathcal
  V))$, 
\[
\bigg|\nabla_{\S^{N-1}}\rho_\e\bigg(\frac{x}{|x|}\bigg)\bigg|\leq 
{\rm const\,}\frac{2|x|^2}{\tilde \sigma\,\e}\leq {\rm const\,} \e,
\]
thus proving \eqref{eq:17}.
\end{proof}

\begin{remark}\label{rem:gis}
Lemma \ref{l:rho-eps} implies that $\widetilde\Sigma_\e\subset
B'_{\kappa\e}$ for all $\e\in(0,\e_1)$.
\end{remark}

We conclude this section with the following refined Poincar\'e-type inequality  for function vanishing on the
crack $\widetilde \Gamma_{\e,r_1}$.
\begin{lemma}\label{l:poin-refined}
   For all $\tau\in(0,1)$ there exists $M_\tau>1$ such that, for 
all $r>0$ and 
$\e<\min\{\e_1, \frac{r}{\kappa M_\tau}\}$,
\begin{equation*}
    \frac{1-\tau}{r}\int_{\partial B_r}u^2\ds\leq
    \int_{B_r}|\nabla u|^2\dx\quad\text{for all } u\in  H^1_{0, \widetilde\Gamma_{\e,r}}(B_{r})
  \end{equation*}
and 
  \begin{equation*}
    \frac{N-1}{r^2}\int_{B_r}u^2\dx\leq\left(1+\frac{1}{1-\tau}\right)
    \int_{B_r}|\nabla u|^2\dx\quad\text{for all } u\in  H^1_{0, \widetilde\Gamma_{\e,r}}(B_{r}).
  \end{equation*}
\end{lemma}
\begin{proof}
It follows directly from \cite[Lemma 4.2, Corollaries 4.3 and
4.4]{FO}, recalling that \Cref{l:rho-eps} implies
$B_r'\setminus B_{\kappa\e}'\subseteq \widetilde\Gamma_{\e,r}$ for
$\e<\min\{\e_1, \frac{r}{\kappa}\}$, as also observed in Remark  \ref{rem:gis}.
\end{proof}

\section{A Pohozaev-type inequality}\label{sec:pohoz-type-ineq}

 Pohozaev-type identites play a pivotal role in the differentiation of the frequency function; indeed, by the coarea formula,
\[
	\frac{\d}{\dr}\int_{B_r}\widetilde{A}\nabla v\cdot\nabla v\dx=\int_{\partial B_r}\widetilde{A}\nabla v\cdot\nabla v\dx
\]
and the Pohozaev identity allows to rewrite the latter boundary
integral in terms of volume integrals and boundary integrals of normal
derivatives. Therefore, this section is devoted to the proof of a
Pohozaev-type inequality for solutions to problem
\eqref{eq:reflected}.

Let $\widetilde A$ be the matrix-valued function introduced 
 in \eqref{eq:A}--\eqref{eq:tilde}. 
From \eqref{eq:5} and \eqref{eq:4} it follows easily that 
\begin{equation}\label{eq:33}
\widetilde A(0)= I_{N},\quad 
\widetilde A(y)=I_{N}+O(|y|)\quad\text{as }|y|\to0.
\end{equation}
Recalling that $r_1$ was defined in \eqref{eq:r1_def}, we define, for all $y\in B_{r_1}$,
\begin{equation}\label{eq:mu-beta}
\mu(y)=\frac{\widetilde A(y)y\cdot y}{|y|^2}\quad\text{and}\quad 
\bm{b}(y)=\frac1{\mu(y)}\, \widetilde A(y)y.
\end{equation}
From  \eqref{eq:4} and \eqref{eq:33} we have that
\begin{equation}\label{eq:38}
	\mu\in C^{0,1}(B_{r_1}), \quad \mu(y)=1+O(|y|),
	\quad \nabla\mu(y)=O(1)\quad\text{as }|y|\to0,
\end{equation}
and
\begin{equation}\label{eq:37}
	\bm{b}(y)=y+O(|y|^2), \quad 
	J_{\bm{b}}(y)=I_N+O(|y|), \quad
	\dive{\bm{b}}(y)=N+O(|y|)
	\quad\text{as }|y|\to0.
\end{equation}
Furthermore we have that, possibly choosing
$r_1$ smaller, for every $y\in B_{r_1}$
\begin{gather}
  \frac{1}{2}|\xi|^2\leq \widetilde{A}(y)\xi \cdot \xi	\leq
  \frac{3}{2}|\xi|^2~
  \text{for all }\xi\in\R^N, \label{eq:bound_A} \\
	 \frac{1}{2}\leq \mu(y)\leq \frac{3}{2}, \label{eq:bound_mu} \\
	 \frac{1}{2}\leq \widetilde{p}(y)\leq \frac{3}{2}. \label{eq:bound_p}
\end{gather}
Finally we have that
\begin{equation}\label{eq:34}
  {\bm{b}}(y)\cdot\frac{y}{|y|}=|y|\quad\text{for all }y\in
  B_{r_1}.
\end{equation}

\begin{proposition}\label{p:poho}
Let $n_0$ be as in \eqref{eq:simple_hp} and let $\e_1, \kappa$ be as in Lemma \ref{l:rho-eps}.
For $i=1,\dots,n_0$ and $\e\in (0,\e_1)$, let $v_i^\e$ solve
\eqref{eq:reflected}. There exists $\tilde{r}\in (0,r_1)$ such
that, for all $\e\in (0,\min\{\e_1,\tilde r/\kappa\})$ and a.e. $r\in(\kappa \e,\tilde{r})$,
  \begin{multline}
 r \!\int_{\partial B_r}\!\!\widetilde A\nabla v_i^\e \!\cdot \!\nabla v_i^\e\ds
-\!\int_{B_r}\!\!\Big((\dive\bm{b}) \widetilde A\nabla v_i^\e \!\cdot \!\nabla
  v_i^\e-2J_{\bm{b}}(\widetilde A\nabla v_i^\e) \!\cdot \! \nabla v_i^\e+(d \widetilde A \nabla
  v_i^\e \nabla v_i^\e) \!\cdot \!\bm{b}\Big)\dy\\
  \geq 2r
\int_{\partial B_r}\frac1{\mu}{|\widetilde A\nabla v_i^\e\cdot\nnu|^2}\ds+
2\lambda_i^\e\int_{B_r}\widetilde p(\bm{b}\cdot \nabla v_i^\e) v_i^\e\dy. \label{eq:35}
\end{multline}
 \end{proposition}
 The proof of a Pohozaev-type identity for an equation of type
\eqref{eq:reflected} is classically based on the integration 
by Divergence Theorem of the following Rellich-Ne\u{c}as identity
\begin{multline}\label{eq:rellich_necas}
  \dive\big((\widetilde A\nabla w\cdot\nabla w)\bm{b}-
  2(\bm{b}\cdot\nabla w) \widetilde A \nabla w\big)=
  (\dive\bm{b}) \widetilde A\nabla w\cdot\nabla
  w-2(J_{\bm{b}}\widetilde A\nabla w)
    \cdot\nabla w\\
    +(d \widetilde A \nabla w\nabla w)\cdot\bm{b}
-2(\bm{b}\cdot\nabla w)\dive(\widetilde A\nabla w),
\end{multline}
which holds in a distributional sense for any $H^2$-function $w$.
Nevertheless, the highly non-smoothness of the cracked domain
$B_{r_1}\setminus \widetilde \Gamma_{\e,r_1}$, on which equation
\eqref{eq:reflected} is satisfied, prevents us from the direct use of the
Divergence Theorem, because of  lack of regularity of the solution $v_i^\e$, which could
indeed fail to be $H^2$.  

In order to overcome this regularity issue, we perform an approximation process.

\subsection{Regular approximation of the cracked
domain}\label{sec:regul-appr-crack}

The first step of our regularization procedure relies in the
construction of a family of sets which approximate the cracked
domain, being of class $C^{1,1}$ and star-shaped with respect to the origin.

To this aim, let us first consider the sequence of functions
\[
h_n:\R\to\R,\quad 
h_n(t)=\left(n^2t^2+\frac1{n^2}\right)^{\!\!1/8}.
\]
  By direct computations it is easy to verify that 
\begin{equation}\label{eq:22}
  h_n(t)\geq 4th_n'(t)\quad\text{for all }t\in \R
\end{equation}
and 
\begin{equation}\label{eq:21}
  \text{if $(y',y_N)\in B_{r}$ and $|y'|\geq h_n(y_N)$, \quad then
    $|y_N|\leq \frac{r^4}n$.}
\end{equation}
Recall the definition of $\rho_\e$ in \eqref{eq:rho-eps}.
For every $\e\in (0,\min\{\e_1,r_1/\kappa\})$,  $r\in(\kappa\e,r_1]$, and
$n\in\N_*$ we define 
\begin{align*}
&D_{\e,r}^n=\big\{x=(x',x_N)\in
  B_r: |x'|< \rho_\e\big(\tfrac{x'}{|x'|}\big)+ h_n(x_N) \big\} ,\\
&\Gamma_{\e,r}^n=\big\{x=(x',x_N)\in
  B_r:|x'|= \rho_\e\big(\tfrac{x'}{|x'|}\big)+ h_n(x_N) \big\} \subset\partial
  D_{\e,r}^n,\\
&S_{\e,r}^n=\partial D_{\e,r}^n\setminus \Gamma_{\e,r}^n.
\end{align*}
We note that, for   $\e\in (0,\min\{\e_1,r_1/\kappa\})$ and  $r\in(\kappa\e, r_1]$
fixed, $D_{\e,r}^n\neq B_r$ and $\Gamma_{\e,r}^n$ is not empty provided $n$ is sufficiently large.

\begin{lemma}\label{l:ay-nu}
There exists $r_2\in(0,r_1)$ such that, 
for every $\e\in (0,\min\{\e_1,r_2/\kappa\})$ (with $\kappa$ as in Lemma \ref{l:rho-eps}), there exists
 $n_\e\in\N_*$ such that $A(y)y\cdot\nnu(y)> 0$ on $\Gamma_{\e,r_2}^n$
 for all $n\geq n_\e$, where $\nnu(y)$ is the exterior unit normal at $y\in\partial D_{\e,r_2}^n$. 
\end{lemma}
\begin{proof}
Because of the definition of $\Gamma_{\e,r}^n$, we have that, if $y\in
\Gamma_{\e,r}^n$, 
\[
\nnu(y)=\frac{\nabla G_\e^n(y)}{\|\nabla G_\e^n(y)\|},
\]
where
$G_\e^n(y',y_N)=|y'|-h_n(y_N)-\rho_\e\big(\tfrac{y'}{|y'|}\big)$.
Then, to prove the lemma it is enough to show that, for some $r_2\in (0,r_1)$ sufficiently small and all $\e\in (0,\min\{\e_1,r_2/\kappa\})$,
$A(y)y\cdot\nabla G_\e^n(y)> 0$ for all $y\in \Gamma_{\e,r_2}^n$,
provided $n$ is sufficiently large.

We observe that 
\[
\nabla
G_\e^n(y',y_N)=\left(\frac{y'}{|y'|}-\frac1{|y'|}\nabla_{\S^{N-1}}
  \rho_\e\big(\tfrac{y'}{|y'|}\big),-h_n'(y_N)\right).
\]
From \eqref{eq:5} and \eqref{eq:alpha} it follows that, as $|y'|\to0$, $|y_N|\to0$, 
\begin{align*}
  A(y',y_N)(y',y_N)=\Big(
 y'+O(|y'|^3)+|y'|O(|y_N|)+O(y_N^2),\, 
  y_N+|y'|O(|y_N|) +O(y_N^2) \Big),
\end{align*}
so that 
\begin{align*}
  A(y',y_N)(y',y_N)\cdot \nabla
&G_\e^n(y',y_N)
=|y'|+O(|y'|^3)+O(|y'|^2)\big|\nabla_{\S^{N-1}}
  \rho_\e\big(\tfrac{y'}{|y'|}\big)\big|+|y'|O(|y_N|)\\
&+O(|y_N|) \big|\nabla_{\S^{N-1}}
  \rho_\e\big(\tfrac{y'}{|y'|}\big)\big|+O(y_N^2)\\
&+\frac{O(y_N^2)}{|y'|} \big|\nabla_{\S^{N-1}}
  \rho_\e\big(\tfrac{y'}{|y'|}\big)\big|-y_N h_n'(y_N)\big(1+O(|y'|)+O(y_N)\big),
\end{align*}
uniformly with respect to $\e\in(0,\e_1)$.
Therefore, in view of \eqref{eq:17},
\begin{multline}\label{eq:25}
  A(y',y_N)(y',y_N)\cdot \nabla
G_\e^n(y',y_N)
=F_1(y',y_N,\e)+|y'|O(|y_N|)\\
\quad +\e\, O(|y_N|)+O(y_N^2)+\e\,
  \frac{O(y_N^2)}{|y'|}-y_N h_n'(y_N)
F_2(y',y_N),
\end{multline}
where 
\[
F_1(y',y_N,\e)=
|y'|+O(|y'|^3)+\e\, O(|y'|^2)
\quad\text{and}\quad 
F_2(y',y_N)=1+O(|y'|)+O(y_N)
\]
as $|y'|\to0$ and $|y_N|\to0$ uniformly with respect to $\e\in(0,\e_1)$.

Let us choose $r_2\in (0,r_1)$ such that
\begin{equation}\label{eq:23}
F_1(y',y_N,\e)\geq\frac12|y'|\quad\text{and}\quad F_2(y',y_N)\leq 2\quad\text{for all
$(y',y_N)\in B_{r_2}$ and $\e\in(0,\e_1)$}.
\end{equation}
We note that, if $\e\in (0,\min\{\e_1,r_2/\kappa\})$, then
$\widetilde\Sigma_\e\subset B_{r_2}'$ and $\Gamma_{\e,r_2}^n$ is not
empty for $n$ sufficiently large. Moreover by \eqref{eq:15} we have that, if $(y',y_N)\in
\Gamma_{\e,r_2}^n$, then $|y'|\geq \rho_\e\big(\tfrac{y'}{|y'|}\big)\geq
\frac\e\kappa$, so that  
\begin{equation}\label{eq:24}
  \frac\e{|y'|}\leq\kappa\quad\text{for all }\e\in
  (0,\min\{\e_1,r_2/\kappa\}) \text{ and }(y',y_N)\in
\Gamma_{\e,r_2}^n.
\end{equation}
From \eqref{eq:21}, \eqref{eq:25}, \eqref{eq:23}, \eqref{eq:24}, 
the definition of $\Gamma_{\e,r_2}^n$, \eqref{eq:22}, and \eqref{eq:15}, we conclude that, 
for all $\e\in
  (0,\min\{\e_1,\frac{r_2}\kappa\})$ and $(y',y_N)\in \Gamma_{\e,r_2}^n$,
\begin{align*}
  A(y',y_N)(y',y_N)\cdot &\nabla G_\e^n(y',y_N)\geq \frac12|y'|-2y_N h_n'(y_N)+O\left(\frac1n\right)\\
                         &=\frac12\left(h_n(y_N)+\rho_\e\big(\tfrac{y'}{|y'|}\big)\right) -2y_N
                           h_n'(y_N)+O\left(\frac1n\right)\\
                         &\geq \frac12\left(h_n(y_N)-4y_N
                           h_n'(y_N) \right)+\frac{\e}{2\kappa}+O\left(\frac1n\right) \geq \frac{\e}{2\kappa}+O\left(\frac1n\right)
                           \quad \text{as }n\to\infty.
\end{align*}
From the above estimate it follows that, choosing $r_2$ as above,
for every $\e\in
  (0,\min\{\e_1,\frac{r_2}\kappa\})$ 
there exists $n_\e\in\N_*$  such that $A(y)y\cdot \nabla
G_\e^n(y)>0$ for all $y\in \Gamma_{\e,r_2}^n$ with $n\geq n_\e$, thus
completing the proof.
\end{proof}

Let $\tilde\alpha\in (0,1)$ and $\mu$ as in \eqref{eq:mu-beta}. In view of \eqref{eq:alpha}, \eqref{eq:5},
\eqref{eq:sigmatilde}, \eqref{eq:det-jac}, 
\eqref{eq:tilde}, \eqref{eq:38} and Lemma \ref{l:poin}, there exists
\begin{equation}\label{eq:scelta-tilder}
\tilde{r}\in \bigg(0,\min \bigg\{r_2,\frac{\tilde\alpha}{2\|\mu\|_{L^\infty(B_{r_2})}}\bigg\}\bigg)
\end{equation}
 such that
\begin{equation}\label{eq:uniforme-ell}
\widetilde A(y)\xi\cdot\xi\geq \tilde\alpha|\xi|^2\quad\text{for all
}\xi\in\R^N\text{ and }y\in B_{\tilde r}
\end{equation}
and,  for all $r\in(0,\tilde r]$,
\begin{equation}\label{eq:hp_coerc-bordo}
 \int_{B_r}\Big(\widetilde A\nabla w\cdot \nabla w-\lambda_{n_0}
 \widetilde p w^2\Big)\dx+\int_{\partial B_r}\mu w^2\ds\geq\tilde\alpha \int_{B_r}|\nabla
 w|^2\dx\quad\text{for all }w\in H^1(B_r).
\end{equation}
In particular, from \eqref{eq:hp_coerc-bordo} it follows that, for all $r\in (0,\tilde{r}]$
\begin{equation}\label{eq:hp_coerc}
 \int_{B_r}\widetilde A\nabla w\cdot \nabla w\dx-\lambda_{n_0}
 \int_{B_r}\widetilde p w^2\dx\geq\tilde\alpha \int_{B_r}|\nabla
 w|^2\dx\quad\text{for all }w\in H^1_0(B_r).
\end{equation}
We denote 
\[
D_{\e}^n=D_{\e,\tilde r}^n, \quad 
\Gamma_{\e}^n=\Gamma_{\e,\tilde r}^n,\quad 
S_{\e}^n=S_{\e,\tilde r}^n,\quad \widetilde \Gamma_{\e} =\overline{B_{\tilde r}'}\setminus\widetilde\Sigma_\e
.
\]

\subsection{\texorpdfstring{Regular approximation of $v_i^\e$}{Regular approximation of $v_i^\e$}}

Let $v_i^\e$ be the functions defined in \eqref{eq:veps_def}. Let us fix 
\[
i\in\{1,\dots,n_0\}\quad\text{and}\quad\e\in (0,\min\{\e_1,\tilde
r/\kappa\}).
\]
Since $v_i^\e\in H^1_{0, \widetilde\Gamma_\e}(B_{\tilde r})$, 
there exists a sequence of functions $v_n=v_{i,n}^\e$ such that  
\begin{equation}\label{eq:30}
  v_n\in
C_c^\infty(\overline{B_{\tilde r}}\setminus \widetilde\Gamma_\e)\quad \text{and}\quad
v_n\to v_i^\e\text{ in }H^1(B_{\tilde r})\text{ as }n\to\infty.
\end{equation}
The functions $v_n$ can be chosen in such a way that 
\begin{equation}\label{vn=0}
v_n(x',x_N)=0\quad \text{if $x'\in \widetilde\Gamma_\e$ and $|x_{N}|\leq\frac{\tilde{r}^4}{n}$}.
\end{equation}

\begin{remark}\label{rem:v_n_zero}
 We observe that $v_n\equiv0$ in
  $B_{\tilde r}\setminus D^n_\e$ (in particular $v_n$ has null trace
  on $\Gamma_\e^n$). Indeed, let
  $x=(x',x_N)\in B_{\tilde r}\setminus D^n_\e$. Then
\[
|x'|\geq h_n(x_N)+\rho_\e\big(\tfrac{x'}{|x'|}\big)>\rho_\e\big(\tfrac{x'}{|x'|}\big),
\]
so that $x'\in \widetilde\Gamma_\e$; moreover 
\[
\tilde r\geq |x'|\geq h_n(x_N)\geq n^{1/4}|x_N|^{1/4},
\]
so that $|x_{N}|\leq \tilde r^4/n$. Then  $v_n(x)=0$ in view of \eqref{vn=0}.
\end{remark}

We are going to construct a sequence of approximated solutions
$\{w_n\}_{n\in\mathbb{N}}$ on the sets $D_\e^n$.
Let $\tilde n_\e\geq n_\e$ be such that $\Gamma_{\e}^n$ is not
empty for all $n\geq \tilde n_\e$.

\begin{lemma}\label{l:approx}
 For every $n\geq \tilde n_\e$, there exists a unique weak
solution $w_n \in H^1(D_\e^n)$ to the problem
	\begin{equation}\label{eq:approx_n}
	\begin{cases}
-\dive (\widetilde{A}(y)\nabla w_n(y))=\lambda_i^\e
        \widetilde p(y) w_n(y),&\text{in }D_\e^n,\\
w_n=v_n,&\text{on }\partial D_\e^n.
\end{cases}
	\end{equation}
Furthermore, extending $w_n$ trivially to zero in
$B_{\tilde{r}}\setminus D_\e^n$, we have that
 \begin{equation}\label{eq:26}
w_n\to v_i^\e \quad\text{in} \quad H^1(B_{\tilde r})\quad\text{as $n\to+\infty$}.
\end{equation}
\end{lemma}
\begin{proof}
 Letting $W_n:=w_n-v_n$, we observe that $w_n$ is a weak solution to 
  \eqref{eq:approx_n}
 if and only if $W_n\in H^1_0(D_\e^n)$ weakly solves
\begin{equation}\label{eq:W_n}
	\begin{cases}
-\dive (\widetilde{A}\nabla W_n)-\lambda_i^\e\,
        \widetilde p \,W_n=\lambda_i^\e
        \widetilde p\, v_n+\dive (\widetilde{A}\nabla v_n)&\text{in }D_\e^n,\\
W_n=0&\text{on }\partial D_\e^n.
\end{cases}\end{equation}
i.e. 
\begin{equation}\label{eq:27}
a_n(W_n,\phi)=\langle F_n,\phi\rangle\quad\text{for all }\phi\in
H^1_0(D_\e^n)
\end{equation}
where 
\begin{align*}
&a_n:H^1_0(D_\e^n) \times H^1_0(D_\e^n) \to\R,\quad a_n(\phi_1,\phi_2)=
\int_{D_\e^n}\widetilde A\nabla
\phi_1\cdot\nabla\phi_2\dy-\lambda_i^\e\int_{D_\e^n}\widetilde p
  \,\phi_1\,\phi_2\dy,\\
&F_n\in H^{-1}(D_\e^n),\quad {}_{H^{-1}(D_\e^n)}\langle F_n,\phi\rangle_{H^{1}_0(D_\e^n)}=\int_{D_\e^n}
       \left(\lambda_i^\e \widetilde p\, v_n\,\phi-\widetilde A\nabla
v_n\cdot\nabla\phi\right)\dy.
\end{align*}
Since $\widetilde p\in L^\infty(D_\e^n)$, by the Poincaré inequality
the bilinear form $a_n$ is    continuous, whereas estimate \eqref{eq:hp_coerc} implies that $a_n$ is coercive on
$H^1_0(D_\e^n)$. The Lax-Milgram Theorem ensures the existence of
a unique weak solution $W_n$ to \eqref{eq:W_n} and, consequently, of a
unique weak solution $w_n=W_n+v_n$ to \eqref{eq:approx_n}, for all
$n\geq\tilde n_\e$.

From the Poincar\'e inequality and boundedness of $\{v_n\}$ in
$H^1(B_{\tilde r})$ we can easily deduce that 
\[
|\langle F_n,\phi\rangle|\leq c\,\|\phi\|_{H^1_0(D_\e^n)}\quad\text{for all }\phi\in
H^1_0(D_\e^n),
\]
for some constant $c>0$ which may depend on $i,\e,N,\tilde r$ but
is independent of $n$. Therefore, choosing $\phi=W_n$ in \eqref{eq:27}
and using estimates \eqref{eq:hp_coerc} and \eqref{eq:28},
we obtain that 
\[
\tilde\alpha \|W_n\|_{H^1_0(D_\e^n)} ^2\leq a_n(W_n,W_n)=\langle F_n,W_n\rangle \leq  c\,\|W_n\|_{H^1_0(D_\e^n)},
\]
so that 
\begin{equation*}\label{eq:29}
  \|W_n\|_{H^1_0(B_{\tilde r})}\leq \frac c{\tilde\alpha}\quad \text{for all
  }n\geq\tilde n_\e, 
\end{equation*}
 where $W_n$ is extended trivially to zero in $B_{\tilde{r}}\setminus D_\e^n$. 
Therefore there exist $W\in H^1_0(B_{\tilde r})$ and  a subsequence
$\{W_{n_k}\}$ of
$\{W_n\}$ such that 
\begin{equation}\label{eq:31}
W_{n_k}\rightharpoonup W\quad\text{weakly in
$H^1_0(B_{\tilde r})$}.
\end{equation}
We observe that, for any $\delta>0$ small, the set
$\Lambda_\delta^\e=\left\{(x',0)\in B_{\tilde r}:|x'|\geq\delta+\rho_\e\big(\tfrac{x'}{|x'|}\big)\right\}$
is contained in $B_{\tilde r}\setminus \overline{D_n^\e}$ for $n$ sufficiently
large (how large it should be depends on $\delta$); hence, for any
$\delta>0$ small fixed,  $W_n\in
H^1_{0,\Lambda_\delta^\e\cup \partial B_{\tilde r}} ( B_{\tilde r})$  for $n$ sufficiently
large. Since  $H^1_{0,\Lambda_\delta^\e\cup \partial B_{\tilde
    r}} ( B_{\tilde r})$ is weakly closed in $H^1(B_{\tilde r})$, we deduce that $W\in H^1_{0,\Lambda_\delta^\e\cup \partial B_{\tilde
    r}} ( B_{\tilde r})$ for all $\delta>0$ small; therefore we
conclude that $W\in 
H^1_{0,\partial B_{\tilde r}\cup \widetilde \Gamma_{\e}}(B_{\tilde r})$.

Then, from \eqref{eq:vie}, \eqref{eq:30},   \eqref{eq:31}, and \eqref{eq:27} we deduce that 
\begin{align*}
0&=-\int_{B_{\tilde r}}\left(\widetilde A\nabla
v_i^\e\cdot\nabla W\dy-\lambda_i^\e\widetilde p\,
v_i^\e\, W\right)\dy\\
&=-\lim_{k\to+\infty}\int_{B_{\tilde r}}\left(\widetilde A\nabla
v_{n_k}\cdot\nabla W_{n_k}\dy-\lambda_i^\e\widetilde p\,
v_{n_k}\, W_{n_k}\right)\dy\\
&=\lim_{k\to+\infty}
\langle F_{n_k},W_{n_k}\rangle=\lim_{k\to+\infty}a_{n_k}(W_{n_k},W_{n_k}) 
\end{align*}
thus concluding that $\|W_{n_k}\|_{H^1_0(B_{\tilde r})}\to 0$ as
$k\to+\infty$ in view of \eqref{eq:hp_coerc}. Hence $W_{n_k}\to 0$ in
$H^1_0(B_{\tilde r})$ and $w_{n_k}=W_{n_k}+v_{n_k}\to v_i^\e$ in
$H^1(B_{\tilde r})$ as $k\to+\infty$ thanks to \eqref{eq:30}. By
Urysohn’s subsequence principle, we finally conclude that
$w_n\to v_i^\e$ in $H^1(B_{\tilde r})$ as $n\to+\infty$.
\end{proof} 

\subsection{Proof of Proposition \ref{p:poho}}

Let us fix $i=1,\dots,n_0$ and $\e\in (0,\min\{\e_1,\tilde
r/\kappa\})$ with $\tilde r$ being as in section
\ref{sec:regul-appr-crack} (see
\eqref{eq:scelta-tilder}--\eqref{eq:hp_coerc}). Let $v_i^\e$ solve
\eqref{eq:reflected} and, for all $n\geq \tilde n_\e$, let $w_n \in H^1(D_\e^n)$ be as in
 Lemma \ref{l:approx}. 

Let $r\in (\kappa \e, \tilde{r})$.
By classical elliptic regularity theory
 (see e.g. \cite[Theorem 2.2.2.3]{grisvard}) we have that $w_n\in
 H^2(D_{\e,r}^n)$. Then 
\[
(\widetilde A\nabla w_n\cdot\nabla w_n)\bm{b}-
  2(\bm{b}\cdot\nabla w_n) \widetilde A \nabla w_n\in
  W^{1,1}(D_{\e,r}^n)
\]
so that we can use the integration by parts formula for Sobolev
functions on the Lipschitz domain $D_{\e,r}^n$ and  obtain, in view of
\eqref{eq:rellich_necas}, \eqref{eq:approx_n}, and \eqref{eq:34},  
\begin{multline*}
 r\int_{S_{\e,r}^n} \widetilde A\nabla w_n\cdot\nabla
 w_n\ds
-\int_{D_{\e,r}^n}  \left( (\dive\bm{b}) \widetilde A\nabla w_n\cdot\nabla
  w_n-2(J_{\bm{b}}\widetilde A\nabla w_n)
    \cdot\nabla w_n
    +(d \widetilde A \nabla w_n\nabla w_n)\!\cdot \!\bm{b}
\right)\dy,
\\=
2r\int_{S_{\e,r}^n}
  \frac{1}{\mu}|\widetilde A \nabla w_n\cdot\nnu|^2\ds+2 \lambda_i^\e
\int_{D_{\e,r}^n} (\bm{b}\cdot\nabla w_n)
        \widetilde p\, w_n\dy\\
+\int_{\Gamma_{\e,r}^n} \big(-(\widetilde A\nabla w_n\cdot\nabla
 w_n)\bm{b}\cdot\nnu+
  2(\bm{b}\cdot\nabla w_n) \widetilde A \nabla w_n\cdot\nnu\big)\ds
\end{multline*}
where $\nnu=\nnu(y)$ is the exterior unit normal at $y\in\partial
D_{\e,r}^n=\Gamma_{\e,r}^n\cup S_{\e,r}^n$. 
On $\Gamma_{\e,r}^n$ we have that $w_n=0$ so that $\nabla
w_n=\frac{\partial w_n}{\partial\nnu}\,\nnu$; hence 
\begin{equation}\label{eq:starshape_approx}
-(\widetilde A\nabla w_n\cdot\nabla
 w_n)\bm{b}\cdot\nnu+
  2(\bm{b}\cdot\nabla w_n) \widetilde A \nabla w_n\cdot\nnu
=\frac1\mu \bigg|\frac{\partial w_n}{\partial\nnu}\bigg|^2( \widetilde A y\cdot\nnu)
( \widetilde A \nnu\cdot\nnu)\geq0\quad\text{on }\Gamma_{\e,r}^n
\end{equation}
thanks to Lemma \ref{l:ay-nu} and \eqref{eq:bound_A}. Then we obtain the following inequality
\begin{multline}\label{eq:32}
 r\int_{\partial B_r} \widetilde A\nabla w_n\cdot\nabla
 w_n\ds\\
-\int_{B_r}  \left( (\dive\bm{b}) \widetilde A\nabla w_n\cdot\nabla
  w_n-2(J_{\bm{b}}\widetilde A\nabla w_n)
    \cdot\nabla w_n
    +(d \widetilde A \nabla w_n\nabla w_n)\cdot\bm{b}
\right)\dy,
\\\geq
2r\int_{\partial B_r}
  \frac{1}{\mu}|\widetilde A \nabla w_n\cdot\nnu|^2\ds+2 \lambda_i^\e
\int_{B_r} (\bm{b}\cdot\nabla w_n)
        \widetilde p\, w_n\dy
\end{multline}
for all $n\geq \tilde n_\e$, where  $w_n$ is extended trivially to zero in
$B_{r}\setminus D_\e^n$.

For $i$ and $\e$ fixed as above, we  now intend to pass to the limit in
\eqref{eq:32} as $n\to+\infty$ for every $r\in (\kappa \e, \tilde{r})$.
The strong $H^1$-convergence of $w_n$ to $v_i^\e$ stated in
\eqref{eq:26} directly implies that 
\begin{multline*}
\lim_{n\to+\infty}\int_{B_r}  \left( (\dive\bm{b}) \widetilde A\nabla w_n\cdot\nabla
  w_n-2(J_{\bm{b}}\widetilde A\nabla w_n)
    \cdot\nabla w_n
    +(d \widetilde A \nabla w_n\nabla w_n)\cdot\bm{b}
\right)\dy\\
=\int_{B_r}  \left( (\dive\bm{b}) \widetilde A\nabla v_i^\e\cdot\nabla
  v_i^\e-2(J_{\bm{b}}\widetilde A\nabla v_i^\e)
    \cdot\nabla v_i^\e
    +(d \widetilde A \nabla v_i^\e\nabla v_i^\e)\cdot\bm{b}
\right)\dy
\end{multline*}
and 
\[
\lim_{n\to+\infty}\int_{B_r} (\bm{b}\cdot\nabla w_n)
        \widetilde p\, w_n\dy=\int_{B_r} (\bm{b}\cdot\nabla v_i^\e)
        \widetilde p\, v_i^\e\dy,
\]
for all $r\in (\kappa \e, \tilde{r})$. 
In order to deal with the boundary integrals in \eqref{eq:32}, we observe that,
by the strong $H^1$-convergence \eqref{eq:26} of $w_n$ to $v_i^\e$,
\begin{equation*}
\lim_{n\rightarrow +\infty}\int_0^{\tilde r}\biggl(\int_{\partial B_r}|\nabla(w_n-v_i^\e)|^2\ds\biggr)\dr= 0,
\end{equation*}
i.e., letting
$F_n(r)=\int_{\partial B_r}|\nabla(w_n-v_i^\e)|^2\ds$, 
$F_n\to 0$ in $L^1(0,\tilde r)$. Then there exists a subsequence
$F_{n_k}$ such that $F_{n_k}(r)\rightarrow 0$ 
for a.e. $r\in (0,\tilde r)$, hence 
\[
\nabla w_{n_k}\to \nabla v_i^\e\quad\text{in }L^2(\partial
B_r)\quad\text{as }k\to+\infty 
\quad\text{for a.e. $r\in (0,\tilde r)$}.
\]
Therefore, for a.e. $r\in (0,\tilde r)$,
\begin{align*}
&\lim_{k\to+\infty}\int_{\partial B_r} \widetilde A\nabla w_{n_k}\cdot\nabla
 w_{n_k}\ds
=\int_{\partial B_r} \widetilde A\nabla v_i^\e\cdot\nabla
 v_i^\e\ds,\\
&\lim_{k\to+\infty}\int_{\partial B_r}
  \frac{1}{\mu}|\widetilde A \nabla w_n\cdot\nnu|^2\ds=
\int_{\partial B_r}
  \frac{1}{\mu}|\widetilde A \nabla v_i^\e\cdot\nnu|^2\ds.
\end{align*}
Hence we can pass to the limit in
\eqref{eq:32} as $n\to+\infty$ for a.e. $r\in (\kappa \e, \tilde{r})$,
thus obtaining \eqref{eq:35}.\qed

\section{Energy estimates via an Almgren-type frequency function}\label{sec:energy-estimates-via}

\subsection{Monotonicity formula}

 For every $\lambda\in\R$, $r>0$, and $v\in H^1(B_r)$ we define
\[
E(v,r,\lambda)=r^{2-N}\int_{B_r}\left(
\widetilde{A}\nabla v\cdot\nabla v-\lambda
        \widetilde p v^2\right)\dy,
 \]
where $\widetilde A$ and $\widetilde p$ have been introduced in
\eqref{eq:tilde}, and 
\[
 H(v,r) =r^{1-N}\int_{\partial B_r}\mu\,v^2\ds,
\]
where $\mu$ has been introduced in \eqref{eq:mu-beta}.
We observe that from \eqref{eq:hp_coerc-bordo} it follows that 
\begin{equation}\label{eq:44}
E(v,r,\lambda)+r H(v,r)\geq0\quad\text{for all $r\in(0,\tilde r]$,
  $\lambda\leq\lambda_{n_0}$ and
  $v\in  H^1(B_r)$}.
\end{equation}
 We also define the Almgren-type frequency function as
\begin{equation}\label{eq:73}
N(v,r,\lambda):=\frac{E(v,r,\lambda)}{H(v,r)}.
\end{equation}

\begin{lemma}\label{l:Hpos}Let $v_i^\e$ be as in \eqref{eq:veps_def}.

\begin{enumerate}[\rm (i)]
\item $H(v_i^\e,r)>0$ for all 
$\e\in (0,\min\{\e_1,\tilde r/\kappa\})$, $r\in[\kappa \e,\tilde{r}]$,
and $1\leq i\leq n_0$.
\item 
For every $r\in(0,\tilde r]$, there exist $C_r > 0$ and $\alpha_r \in
(0,r/\kappa)$ such that $H(v_i^\e,r)
 \geq C_r$ for all $0<\e<\min\{\alpha_r,\e_1\}$ and $1 \leq i\leq n_0$.
  \end{enumerate}
\end{lemma}
\begin{proof} 
 To prove (i) we argue by contradiction and assume that there exist 
$\e\in (0,\min\{\e_1,\tilde r/\kappa\})$, $r\in[\kappa \e,\tilde{r}]$,
and $1\leq
  i\leq n_0$ such that $H(v_i^\e,r)=0$, 
  i.e. $v_i^\e=0$ on $\partial B_r$. Testing \eqref{eq:reflected} with
  $v_i^\e$, integrating
  over $B_r$, and using estimates \eqref{eq:hp_coerc} and \eqref{eq:28},
 we obtain that 
	\begin{equation*}
0=\int_{B_{r}}\widetilde A\nabla
v_i^\e\cdot\nabla v_i^\e\dy-\lambda_i^\e\int_{B_{r}}\widetilde p|
v_i^\e|^2\dy\geq \tilde\alpha \int_{B_r}|\nabla v_i^\e|^2\dy
	\end{equation*}
and hence $v_i^\e\equiv 0$ in
$B_r$. It follows that  $u_i^\e\equiv 0$ in
$B_{r}^+$, i.e. $\varphi_i^\e\equiv 0$ in $F(B_r^+)$, with $F$ as in \eqref{eq:F_def}, so that
from the classical unique
continuation principle for elliptic equations we may conclude
that $\varphi_i^\e\equiv 0$ in $\Omega$, a contradiction.

 In order to prove (ii), suppose by contradiction that there exist 
$0<r\leq\tilde r$, $\e_\ell\to0$, and $i_\ell\in \{1,\ldots,n_0\}$ such that
\begin{equation*}
\lim_{\ell\to+\infty}H(v_{i_\ell}^{\e_\ell},r)=0.
\end{equation*}
From \eqref{eq:eqphiie}, \eqref{eq:ortonor}, and \eqref{eq:28} we have that $\{\varphi_{i_\ell}^{\e_\ell}\}_\ell$ is bounded in $H^1(\Omega)$.
Then there exist $\lambda\in[0,\lambda_{n_0}]$ and $\varphi\in H^{1}(\Omega)$ such that, along a subsequence,
$\lambda_{i_\ell}^{\e_\ell}\to\lambda$ and
$\varphi_{i_\ell}^{\e_\ell} \to \varphi$ weakly in $H^1(\Omega)$ and
strongly in $L^2(\Omega)$.
 From
\eqref{eq:ortonor} it follows  that $\int_\Omega
\varphi^2\dx=1$ and then $\varphi\not\equiv 0$ in $\Omega$. Moreover
$\varphi$ weakly satisfies $-\Delta \varphi=\lambda \varphi$ in $\Omega$.

The weak convergence $\varphi_{i_\ell}^{\e_\ell}\rightharpoonup\varphi$ in
$H^1(\Omega)$ implies tha $v_{i_\ell}^{\e_\ell}\rightharpoonup v$ weakly in $H^1(B_{r_1})$, where $v$ is 
the  even reflection through the hyperplane $\{y_N=0\}$  of
$u=\varphi\circ F$. Since $v_{i_\ell}^{\e_\ell} \in H^1_{0, \widetilde
  \Gamma_{\e_\ell,r_1} }(B_{r_1})$, we have that $v\in H^1_{0,\overline{B_{r_1}'}}(\Omega)$.
Moreover, $v$ weakly solves 
  \begin{equation}\label{eq:36}
	\begin{cases}
-\dive (\widetilde{A}\nabla v)=\lambda
        \widetilde p\, v,&\text{in }B_{r_1}\setminus B_{r_1}',\\
v=0,&\text{on }B_{r_1}'.
\end{cases}
\end{equation}
By compactness of the trace embedding
$H^{1}(B_{r})\hookrightarrow L^2(\partial B_r)$, we also have
that
\[
0=\lim_{\ell\to\infty}H(v_{i_\ell}^{\e_\ell},r)
=\lim_{\ell\to\infty}
r^{1-N}\int_{\partial B_r}\mu\,|v_{i_\ell}^{\e_\ell} |^2\ds
=r^{1-N}\int_{\partial B_r}\mu\,|v |^2\ds
,
\]
which implies that $v=0$ on $\partial B_r$.  Testing
\eqref{eq:36} by $v$ in $B_r$, from \eqref{eq:hp_coerc} we deduce that
  \begin{equation*}
    0=\int_{B_r}
    \big(
\widetilde{A}\nabla v\cdot\nabla v    - \lambda\widetilde p v^2 \big)\dx
    \geq \tilde\alpha\int_{B_r} 
    |\nabla v|^2\dx.
\end{equation*}
Then $v\equiv 0$ in $B_r$ and, consequently,
 $\varphi\equiv 0$ in $F(B_r^+)$, so that
from the classical unique
continuation principle for elliptic equations we may conclude
that $\varphi\equiv 0$ in $\Omega$, a contradiction.
\end{proof}

As a consequence of Lemma \ref{l:Hpos} the function $r\mapsto
N(v_i^\e,r,\lambda_i^\e)$ is well defined in the interval
$[\kappa\e,\tilde r]$ for all $\e\in (0,\min\{\e_1,\tilde r/\kappa\})$ and $1 \leq i\leq n_0$.
Furthermore, estimate \eqref{eq:44} implies that 
\begin{equation*}\label{eq:46}
N(v_i^\e,r,\lambda_i^\e)+1\geq 1-\tilde r>0\quad\text{for all }r\in
[\kappa\e,\tilde r].
\end{equation*}
In order to differentiate the function $N$, we need to differentiate
both $E$ and $H$. We start here by deriving a formula for the
derivative of $H$, which turns out to be expressible in terms of the
function $E$, see \eqref{eq:derH2}.
\begin{lemma}\label{l:derh}
  For all $\e\in
(0,\min\{\e_1,\tilde r/\kappa\})$ and $1 \leq i\leq n_0$,
$H(v_i^\e,\cdot)\in W^{1,1}(\kappa\e,\tilde r)$,
\begin{align}\label{eq:derH1}
 & \frac{d}{dr}H(v_i^\e,r)=2r^{1-N}\int_{\partial B_r}\mu
                           v_i^\e\frac{\partial
  v_i^\e}{\partial\nnu}\ds+O(1) H(v_i^\e,r)\quad\text{as $r\to0$},\\
\label{eq:derH3} & \frac{d}{dr}H(v_i^\e,r)=2r^{1-N}\int_{\partial B_r}
(\widetilde{A}\nabla v_i^\e\cdot\nnu)v_i^\e\ds+O(1) H(v_i^\e,r) \quad\text{as $r\to0$},
\end{align}
and
\begin{equation}\label{eq:derH2}
  \frac{d}{dr}H(v_i^\e,r)=\frac2r E(v_i^\e,r,\lambda_i^\e)+O(1)
  H(v_i^\e,r) \quad\text{as $r\to0$},
\end{equation}
 where 
the derivative is meant in a distributional sense and a.e. in $(\kappa\e,\tilde r)$,
$\nnu=\nnu(y)=\frac{y}{|y|}$ is the unit outer normal vector to
$\partial B_r$,
and
$O(1)$ denotes terms which are bounded for $r$ in a neighbourhood of $0$
uniformly with respect to~$\e$.
\end{lemma}
\begin{proof}
By direct calculations we have that
$H(v_i^\e,\cdot)\in W^{1,1}(\kappa\e,\tilde r)$ and 
\begin{equation*}
  \frac{d}{dr}H(v_i^\e,r)=2r^{1-N}\int_{\partial B_r}\mu
                           v_i^\e\frac{\partial v_i^\e}{\partial\nnu}\ds+r^{1-N}\int_{\partial B_r}|v_i^\e|^2
\frac{\partial \mu}{\partial\nnu}\ds
\end{equation*}
in a weak sense,
from which \eqref{eq:derH1} follows in view of \eqref{eq:38}.

To prove \eqref{eq:derH3}  we define 
\[
\bm{a}(y)=\frac{\mu(y)(\bm{b}(y)-y)}{|y|}
\]
and observe that 
\begin{equation}\label{eq:41}
\int_{\partial B_r}
(\widetilde{A}\nabla v_i^\e\cdot\nnu)v_i^\e\ds
=\int_{\partial B_r}\mu
                           v_i^\e\frac{\partial
  v_i^\e}{\partial\nnu}\ds+\frac12\int_{\partial B_r}\bm{a}\cdot\nabla((v_i^\e)^2)\ds.
\end{equation}
By \eqref{eq:34} we have that 
\begin{equation}\label{eq:39}
\bm{a}(y)\cdot y=0.
\end{equation}
From
\eqref{eq:38},  \eqref{eq:37}, and \eqref{eq:39} we deduce that
\begin{equation}\label{eq:42}
  \dive\bm{a}=\frac{\nabla\mu}{|y|}\cdot
  (\bm{b}-y)
  +\frac{\mu}{|y|}(\dive{\bm{b}}-N)=O(1)\quad\text{ as
  }|y|\to 0.
\end{equation}
From \eqref{eq:41}, \eqref{eq:39}, and \eqref{eq:42} it follows that 
 \begin{align}\label{eq:43}
\int_{\partial B_r}
(\widetilde{A}\nabla v_i^\e\cdot\nnu)v_i^\e\ds
&=\int_{\partial B_r}\mu
                           v_i^\e\frac{\partial
  v_i^\e}{\partial\nnu}\ds-\frac12\int_{\partial
   B_r}(\dive\bm{a})|v_i^\e|^2\ds\\
\notag&=
\int_{\partial B_r}\mu
                           v_i^\e\frac{\partial
  v_i^\e}{\partial\nnu}\ds+O(1)r^{N-1}H(v_i^\e,r)
\end{align}
as $r\to0$ (uniformly in $\e$). 
Combining \eqref{eq:derH1} and \eqref{eq:43} we obtain estimate \eqref{eq:derH3}.

To prove \eqref{eq:derH2}  we test \eqref{eq:reflected} with $v_i^\e$
and integrate over $B_r$ thus obtaining
\begin{equation*}
r^{N-2}E(v_i^\e,r,\lambda_i^\e)=
\int_{B_r}\left(
\widetilde{A}\nabla v_i^\e\cdot\nabla v_i^\e-\lambda_i^\e
        \widetilde p |v_i^\e|^2\right)\dy=\int_{\partial
                                 B_r}(\widetilde{A}\nabla
                                 v_i^\e\cdot\nnu) v_i^\e\ds,
                               \end{equation*}
whose combination with \eqref{eq:derH3} immediately yields \eqref{eq:derH2}.
\end{proof}

\begin{remark}\label{rem:Hprimo-E}
We observe that  \eqref{eq:derH2} implies that there exist $\bar r_0\in (0,\tilde r)$ and $C_0>0$ (independent
of $\e$) such that, for all $\e\in(0,\bar r_0/\kappa)$,
\begin{equation*}
\frac2r E(v_i^\e,r,\lambda_i^\e)-\frac1{C_0}\,
  H(v_i^\e,r)\leq
 \frac{d}{dr}H(v_i^\e,r)\leq\frac2r E(v_i^\e,r,\lambda_i^\e)+C_0
  H(v_i^\e,r)\quad\text{a.e. } r\in (\e\kappa,\bar r_0).
\end{equation*}
\end{remark}

\begin{lemma}\label{l:derE}
  For all $\e\in
(0,\min\{\e_1,\tilde r/\kappa\})$ and $1 \leq i\leq n_0$,
$E(v_i^\e,\cdot,\lambda_i^\e)\in W^{1,1}(\kappa\e,\tilde r)$
and 
\[
\frac{d}{dr}E(v_i^\e,r,\lambda_i^\e)\geq  2r^{2-N}
\int_{\partial B_r}\frac1{\mu}{|\widetilde A\nabla
  v_i^\e\cdot\nnu|^2}\ds +O(1)E(v_i^\e,r,\lambda_i^\e)+rO(1)H(v_i^\e,r)
\]
as $r\to0$, where the derivative is meant in a distributional sense and a.e. in
$(\kappa\e,\tilde r)$ and 
$O(1)$ denotes terms which are bounded for $r$ in a neighbourhood of $0$
uniformly with respect to~$\e$.
\end{lemma}
\begin{proof}
By direct calculations we have that
$E(v_i^\e,\cdot,\lambda_i^\e)\in W^{1,1}(\kappa\e,\tilde r)$ and 
\begin{multline*}
 \frac{d}{dr}E(v_i^\e,r,\lambda_i^\e)=
 (2-N)r^{1-N}\int_{B_r}\left(
 \widetilde{A}\nabla v_i^\e\cdot\nabla v_i^\e-\lambda_i^\e
         \widetilde p |v_i^\e|^2\right)\dy\\+
 r^{2-N}\int_{\partial B_r}\left(
 \widetilde{A}\nabla v_i^\e\cdot\nabla v_i^\e-\lambda_i^\e
         \widetilde p |v_i^\e|^2\right)\ds.
    \end{multline*}
Hence, in view of Proposition \ref{p:poho},
\begin{align*}
&  \frac{d}{dr}E(v_i^\e,r,\lambda_i^\e)\\
&\geq  2r^{2-N}
\int_{\partial B_r}\frac1{\mu}{|\widetilde A\nabla
   v_i^\e\cdot\nnu|^2}\ds
-\lambda_i^\e r^{2-N}\int_{\partial B_r}\widetilde p |v_i^\e|^2\ds\\
&\quad+r^{1-N}
\int_{B_r}\!\!\Big((2-N) \widetilde{A}\nabla v_i^\e\cdot\nabla v_i^\e+
(\dive\bm{b}) \widetilde A\nabla v_i^\e \!\cdot \!\nabla
  v_i^\e-2J_{\bm{b}}(\widetilde A\nabla v_i^\e) \!\cdot \! \nabla v_i^\e+(d \widetilde A \nabla
  v_i^\e \nabla v_i^\e) \!\cdot \!\bm{b}\Big)\dy\\
&\quad+\lambda_i^\e r^{1-N}\int_{B_r}\left(2 \widetilde p(\bm{b}\cdot \nabla v_i^\e) v_i^\e+(N-2) \widetilde p |v_i^\e|^2\right)\dy.
\end{align*}
Therefore, in view  of \eqref{eq:4}, \eqref{eq:33}, \eqref{eq:37}, \eqref{eq:38}, 
Lemma \ref{l:poin}, and \eqref{eq:hp_coerc-bordo}
\begin{align*}
 \frac{d}{dr}E(v_i^\e,\cdot,\lambda_i^\e)&\geq  2r^{2-N}
\int_{\partial B_r}\frac1{\mu}{|\widetilde A\nabla
   v_i^\e\cdot\nnu|^2}\ds\\
&\quad+O(1)r^{2-N}
\int_{B_r}|\nabla v_i^\e|^2\dy+O(1)r^{2-N}
\int_{\partial B_r}\mu|v_i^\e|^2\ds\\
&= 2r^{2-N}
\int_{\partial B_r}\frac1{\mu}{|\widetilde A\nabla
   v_i^\e\cdot\nnu|^2}\ds +O(1)E(v_i^\e,r,\lambda_i^\e)+rO(1)H(v_i^\e,r)
\end{align*}
thus proving the lemma.
\end{proof}

\begin{lemma}\label{lemma:freq_monot}
 There exist $\bar r\in(0,\bar r_0)$ and $C>0$ such that, for all $\e\in
(0,\min\{\e_1,\bar r/\kappa\})$ and $1 \leq i\leq n_0$,
$N(v_i^\e,\cdot,\lambda_i^\e)\in W^{1,1}(\kappa\e,\bar r)$
and 
\begin{equation}\label{eq:45}
   \frac{d}{dr}N(v_i^\e,r,\lambda_i^\e)\geq
   -C(N(v_i^\e,r,\lambda_i^\e)+1)\quad\text{for a.e. }r\in (\kappa\e,\bar r).
 \end{equation}
Furthermore, for all  $\e\in
(0,\min\{\e_1,\bar r/\kappa\})$, $1 \leq i\leq n_0$, and
$\kappa\e\leq r<R\leq \bar r$
\begin{equation}\label{eq:50}
N(v_i^\e,r,\lambda_i^\e)+1\leq e^{C(R-r)}(N(v_i^\e,R,\lambda_i^\e)+1).
\end{equation}
\end{lemma}
\begin{proof}
The fact that $N(v_i^\e,\cdot,\lambda_i^\e)\in W^{1,1}(\kappa\e,\bar r)$ follows directly from  the fact that
$H(v_i^\e,\cdot)$ and $E(v_i^\e,\cdot,\lambda_i^\e)$ belong to $W^{1,1}(\kappa\e,\tilde r)$ and Lemma \ref{l:Hpos}.

  From \eqref{eq:derH2} it follows that, as $r\to0$, 
\[
E(v_i^\e,r,\lambda_i^\e)=\frac r2 \frac{d}{dr}H(v_i^\e,r)+rO(1)
H(v_i^\e,r),
\]
 which, together with \eqref{eq:derH3} and again \eqref{eq:derH2}, yields
\begin{align*}
 &E(v_i^\e,r,\lambda_i^\e) \frac{d}{dr}H(v_i^\e,r)\\
\notag&\!=\!\left(r^{2-N}\!\!\int_{\partial B_r}\!\!
(\widetilde{A}\nabla v_i^\e\cdot\nnu)v_i^\e\ds+rO(1)
H(v_i^\e,r)\right)\!\!\left(2r^{1-N}\!\!\int_{\partial B_r}\!\!
(\widetilde{A}\nabla v_i^\e\cdot\nnu)v_i^\e\ds+O(1) H(v_i^\e,r)
  \right)\\
\notag&=2r^{3-2N}\left(\int_{\partial B_r}
(\widetilde{A}\nabla v_i^\e\cdot\nnu)v_i^\e\ds\right)^{\!\!2}+rO(1)
  (H(v_i^\e,r))^2\\
\notag&\qquad\qquad \qquad+r^{2-N}O(1) H(v_i^\e,r)\left(\int_{\partial B_r}
(\widetilde{A}\nabla v_i^\e\cdot\nnu)v_i^\e\ds\right) \\
\notag&=2r^{3-2N}\left(\int_{\partial B_r}
(\widetilde{A}\nabla v_i^\e\cdot\nnu)v_i^\e\ds\right)^{\!\!2}+rO(1)
  (H(v_i^\e,r))^2\\
\notag&\qquad \qquad \qquad+rO(1) H(v_i^\e,r)\left(\frac12 \frac{d}{dr}H(v_i^\e,r)+O(1) H(v_i^\e,r)\right) \\
\notag&=2r^{3-2N}\left(\int_{\partial B_r}
(\widetilde{A}\nabla v_i^\e\cdot\nnu)v_i^\e\ds\right)^{\!\!2}+rO(1)
  (H(v_i^\e,r))^2\\
\notag&\qquad \qquad \qquad+rO(1) H(v_i^\e,r)\left(\frac1r E(v_i^\e,r,\lambda_i^\e)+O(1) H(v_i^\e,r)
\right).
\end{align*}
The above estimate, Lemma \ref{l:derE}, the Cauchy-Schwarz inequality,
and \eqref{eq:44} imply that 
\begin{align*}\label{eq:47}
   \frac{d}{dr}N(&v_i^\e,r,\lambda_i^\e)=\frac{H(v_i^\e,r)\frac{d}{dr}E(v_i^\e,r,\lambda_i^\e)-
                                         E(v_i^\e,r,\lambda_i^\e)
                                         \frac{d}{dr}H(v_i^\e,r)}{(H(v_i^\e,r))^2}\\
\notag&\geq 2r^{3-2N}\frac{\left(\int_{\partial B_r}\frac1{\mu}{|\widetilde A\nabla
   v_i^\e\cdot\nnu|^2}\ds\right)\left(\int_{\partial B_r}\mu |v_i^\e|^2\ds\right)-\left(\int_{\partial B_r}
(\widetilde{A}\nabla
  v_i^\e\cdot\nnu)v_i^\e\ds\right)^2}{(H(v_i^\e,r))^2}\\
\notag&\quad +O(1)
        N(v_i^\e,r,\lambda_i^\e)+r\,O(1)\\[5pt]
\notag &\geq -C\left( N(v_i^\e,r,\lambda_i^\e)+1\right)
\end{align*}
a.e. in $(\e\kappa,\bar r)$, for some $\bar r\in (0,\bar r_0)$ and $C>0$ independent
of $\e$.

Finally, estimate \eqref{eq:50} follows by integration of
\eqref{eq:45} over the interval $[r,R]$.
\end{proof}

\begin{lemma}\label{lemma_stima_H_sotto}
  For $\tau\in\big(0,\frac12\big)$, let $ M_\tau$ be as in \Cref{l:poin-refined}.
  Let $i\in\{1,\dots,n_0\}$.
 If
$\e< \min\{\e_1,\frac{\bar r}{\kappa M_\tau}\}$  and
$\kappa M_\tau\e\leq s_1<s_2\leq\bar r$, then
\begin{equation*}
  \frac{H(v_i^\e,s_2)}{H(v_i^\e,s_1)} 
  \geq e^{-(4+C_0^{-1})\bar r}\left(\dfrac{s_2}{s_1}\right)^{\frac{2\tilde\alpha(1-\tau)}{\|\mu\|_{L^\infty(B_{\bar r})}}}.
\end{equation*}
\end{lemma}
\begin{proof}
Let $\tilde{\alpha}\in (0,1)$ be as in \eqref{eq:uniforme-ell}. By \eqref{eq:hp_coerc-bordo}, \eqref{eq:28} and \Cref{l:poin-refined}, for every
$\e< \min\{\e_1,\frac{\bar r}{\kappa M_\tau}\}$  and
$r\in[\kappa M_\tau\e,,\bar r]$
\begin{align*}
\tilde\alpha \int_{B_r}|\nabla
 v_i^\e|^2\dx&\leq
\int_{B_r}\left(\widetilde A\nabla v_i^\e\cdot \nabla v_i^\e-\lambda_{i}^\e
 \widetilde p |v_i^\e|^2\right)\dx
+\|\mu\|_{L^\infty(B_{\bar r})}\int_{\partial B_r}|v_i^\e|^2\ds\\
&\leq
\int_{B_r}\left(\widetilde A\nabla v_i^\e\cdot \nabla v_i^\e-\lambda_{n_0}
\widetilde p |v_i^\e|^2\right)\dx
+\|\mu\|_{L^\infty(B_{\bar r})}\int_{\partial B_r}|v_i^\e|^2\ds\\
&\leq r^{N-2}E(v_i^\e,r,\lambda_i^\e)+\frac{\|\mu\|_{L^\infty(B_{\bar r})}r}{1-\tau} \int_{B_r}|\nabla
 v_i^\e|^2\dx
\end{align*}
so that, using again \Cref{l:poin-refined} 
and recalling that $\tilde\alpha-  2\|\mu\|_{L^\infty(B_{\bar
  r})}r>0$ in view of \eqref{eq:scelta-tilder}, we obtain that  
\begin{align}\label{eq:40}
 r^{N-2}E(v_i^\e,r,\lambda_i^\e)&\geq \left(\tilde\alpha-  \frac{\|\mu\|_{L^\infty(B_{\bar r})}r}{1-\tau}\right) \int_{B_r}|\nabla
 v_i^\e|^2\dx\geq \left(\tilde\alpha-  2\|\mu\|_{L^\infty(B_{\bar
  r})}r\right)  \int_{B_r}|\nabla
 v_i^\e|^2\dx
\\
\notag&\geq \left(\tilde\alpha-  2\|\mu\|_{L^\infty(B_{\bar
  r})}r\right)\frac{1-\tau}{r}\int_{\partial
        B_r}|v_i^\e|^2\ds\\
&\notag\geq \left(\tilde\alpha-  2\|\mu\|_{L^\infty(B_{\bar
  r})}r\right) \frac{1-\tau}{\|\mu\|_{L^\infty(B_{\bar
  r})}}r^{N-2}H(v_i^\e,r)\\
&\notag \geq 
\frac{\tilde\alpha(1-\tau)}{\|\mu\|_{L^\infty(B_{\bar
  r})}}r^{N-2}H(v_i^\e,r)-2r^{N-1}H(v_i^\e,r).
\end{align}
From \Cref{rem:Hprimo-E} and \eqref{eq:40} it follows that 
\[
 \frac{d}{dr}H(v_i^\e,r)
\geq \frac2r E(v_i^\e,r,\lambda_i^\e)-\frac1{C_0}\,
  H(v_i^\e,r)\geq
 \frac{2\tilde\alpha(1-\tau)}{\|\mu\|_{L^\infty(B_{\bar
  r})}}\frac{H(v_i^\e,r)}r-(4+C_0^{-1})H(v_i^\e,r)
\]
in $[\kappa M_\tau\e,,\bar r]$. The conclusion then follows by
integration between $s_1$ and $s_2$.
\end{proof}

\subsection{Energy estimates}

By a combination of Lemma \ref{l:poin} with estimates \eqref{eq:bound_A}--\eqref{eq:bound_p} it is possible to prove the following perturbed Poincar\'e-type inequality.
\begin{lemma}\label{lemma:poinc_pert}
  For any $r\leq r_1$ and for any $u\in H^1(B_r)$ there holds
\[
  \frac{N-1}{r^2}\int_{B_r}\widetilde{p}u^2\dy\leq 3\left(
    \int_{B_r}\widetilde{A}\nabla u\cdot\nabla
    u\dy+\frac{1}{r}\int_{\partial B_r}\mu u^2\ds \right).
\]
\end{lemma}

\begin{proposition}\label{prop:energy_estim}
	For any $R,K$ such that $R\geq K\geq  \kappa$ there holds
	\begin{gather}
		\int_{B_{R\e}}\widetilde{A}\nabla v_i^\e\cdot\nabla v_i^\e\dy=O(\e^{N-2}H(v_i^\e,K\e)) \quad\text{as }\e\to 0, \label{eq:energy_estim_th1}\\
		\int_{B_{R\e}}\widetilde{p}\abs{v_i^\e}^2\dy=O(\e^{N}H(v_i^\e,K\e)) \quad\text{as }\e\to 0, \label{eq:energy_estim_th2} \\
		\int_{\partial B_{R\e}}\mu\abs{v_i^\e}^2\ds=O(\e^{N-1}H(v_i^\e,K\e))\quad\text{as }\e\to 0, \label{eq:energy_estim_th3}
	\end{gather}
	for all $i\in\{1,\dots,n_0\}$.

\end{proposition}
\begin{proof}
  First of all, we prove that
	\begin{equation}\label{eq:energy_estim_1}
		\mathcal{N}(v_i^\e,\bar{r},\lambda_i^\e)=O(1)\quad\text{as }\e\to 0.
	\end{equation}
We notice that 
\begin{align*}
  E(v_i^\e,\bar{r},\lambda_i^\e)&\leq
  \bar{r}^{2-N}\int_{B_{\bar{r}}}\widetilde{A}\nabla v_i^\e\cdot\nabla
  v_i^\e\dy\\
  &=2\bar{r}^{2-N}\int_{\Phi^{-1}(B_{\bar{r}}^+)}
  \abs{\nabla\varphi_i^\e}^2\dx\leq
  2\bar{r}^{2-N}\int_{\Omega}\abs{\nabla\varphi_i^\e}
  =2\bar{r}^{2-N}\lambda_i^\e.
\end{align*}
Since $\lambda_i^\e\leq \lambda_{n_0}$ for all $\e\in(0,1)$ and all
$1\leq i\leq n_0$, we have that $E(v_i^\e,\bar{r},\lambda_i^\e)$ is
bounded for $\e\in(0,\min\{\e_1,\bar{r}/K\})$. From Lemma \ref{l:Hpos}
 (ii) we know that there exists $C_{\bar{r}}>0$ and
$\alpha_{\bar{r}}\in(0,\bar{r}/K)$ such that
$H(v_i^\e,\bar{r})\geq C_{\bar r}$ for all
$\e\in(0,\min\{\alpha_{\bar{r}},\e_1\})$. Therefore
\eqref{eq:energy_estim_1} is proved. Hence from 
  estimate \eqref{eq:50} we deduce that there exists
  $c_0>0$ such
that
	\begin{equation}\label{eq:energy_estim_2}
		\mathcal{N}(v_i^\e,r,\lambda_i^\e)\leq c_0
	\end{equation}
	for all $\e\in (0,\min\{\e_1,\bar{r}/K,\alpha_{\bar{r}}\})$ and all $K\e\leq r\leq \bar{r}$.
	
	By Lemma \ref{l:derh} and \eqref{eq:energy_estim_2} there
        exist $R_1\in(0,\bar r)$, $c_1>0$ and
        $\bar{\e}\in(0,\min\{\e_1,R_1/R\})$
        such that, for any $\e\in(0,\bar{\e})$ and for any
        $K\e\leq r\leq R_1$
\begin{equation}\label{eq:energy_estim_5}
  \frac{\frac{d}{d r} H(v_i^\e,r)}{H(v_i^\e,r)}\leq c_1\left( \frac{1}{r} +1\right).
\end{equation}
	By integration of \eqref{eq:energy_estim_5} in $(K\e,R\e)$ we obtain that
\[
  \frac{H(v_i^\e,R\e)}{H(v_i^\e,K\e)}\leq
  \left(\frac{R}{K}\right)^{\!\!c_1} e^{c_1\e(R-K)}
\]
which, in turn, implies \eqref{eq:energy_estim_th3}.
	
 From \eqref{eq:hp_coerc} and \eqref{eq:bound_A} we have that there exists $R_2\in (0,\tilde{r})$ such that
\[
  \int_{B_{R\e}}(\widetilde{A}\nabla v_i^\e\cdot\nabla
  v_i^\e-\lambda_{n_0}\widetilde{p}\abs{v_i^\e}^2 )\dy \geq
  \frac{\tilde\alpha}{2}\int_{B_{R\e}}\widetilde{A}\nabla v_i^\e\cdot\nabla
  v_i^\e\dy
\]
        for all $\e\in(0,\min\{\e_1,R_2/R\})$. Therefore, since
        $\lambda_i^\e\leq \lambda_{n_0}$, there exists $c_2>0$ such
        that
\[
  \int_{B_{R\e}}\widetilde{A}\nabla v_i^\e\cdot\nabla v_i^\e\dy\leq
  c_2\e^{N-2}E(v_i^\e,R\e,\lambda_i^\e)
\]
for all $\e\in(0,\min\{\e_1,R_2/R\})$. Then \eqref{eq:energy_estim_2}
(with $r=R\e$) yields
\begin{equation}\label{eq:energy_estim_3}
  \int_{B_{R\e}}\widetilde{A}\nabla v_i^\e\cdot\nabla v_i^\e\dy\leq c_0 c_2\e^{N-2}H(v_i^\e,R\e),
\end{equation}	
 for all
  $\e\in (0,\min\{\e_1, R_2/R,\alpha_{\bar{r}}\})$.  This fact,
together with \eqref{eq:energy_estim_th3}, proves
\eqref{eq:energy_estim_th1}.  Applying Lemma \ref{lemma:poinc_pert}
with $r=R\e$, for $\e$ sufficiently small, and $u=v_i^\e$, in view of
\eqref{eq:energy_estim_th1} and \eqref{eq:energy_estim_th3} we obtain
\eqref{eq:energy_estim_th2}, thus concluding the proof.
	\end{proof}

Hereafter, we denote
\begin{equation}\label{eq:def_beta}
	\beta:=2\tilde\alpha/\norm{\mu}_{L^\infty(B_{\bar{r}})}.
\end{equation}

\begin{proposition}\label{prop:rough_estim}
	Let $\tau\in(0,1/2)$, $ M_\tau>1$ as in Lemma \ref{l:poin-refined} and $\beta$ as in \eqref{eq:def_beta}. Then, for any $R\geq  M_\tau\kappa$, there holds
	\begin{gather}
		\int_{B_{R\e}}\widetilde{A}\nabla v_i^\e\cdot\nabla v_i^\e\dy=O(\e^{N-2+\beta(1-\tau)})\quad\text{as }\e\to 0, \label{eq:rough_estim_th1} \\
		\int_{B_{R\e}}\widetilde{p}\abs{v_i^\e}^2\dy=O(\e^{N+\beta(1-\tau)})\quad\text{as }\e\to 0,\label{eq:rough_estim_th2} \\
		\int_{\partial B_{R\e}}\mu\abs{v_i^\e}^2\ds=O(\e^{N-1+\beta(1-\tau)})\quad\text{as }\e\to 0,\label{eq:rough_estim_th3}
	\end{gather}
	for all $i\in\{1,\dots,n_0\}$.
\end{proposition}
\begin{proof}
 From Lemma \ref{lemma_stima_H_sotto} we know that there exists a constant $C>0$ such that
\begin{equation}\label{eq:rough_estim_1}
  H(v_i^\e,R\e)\leq C\e^{\beta(1-\tau)}H(v_i^\e,\bar{r}) \quad\text{for all }\e\in(0,\bar{r}/R).
\end{equation}
Combining estimates \eqref{eq:bound_A} and \eqref{eq:bound_mu} with Lemma
        \ref{l:poin-refined}, we obtain that
\begin{equation}\label{eq:rough_estim_2}
  H(v_i^\e,\bar{r})\leq
  \frac{3\bar{r}^{2-N}}{(1-\tau)}
  \int_{B_{\bar{r}}}\widetilde{A}\nabla v_i^\e\cdot\nabla v_i^\e\dy. 
\end{equation}
By definition of $\widetilde{A}$ and monotonicity of eigenvalues, we have that
\[
  \int_{B_{\bar{r}}}\widetilde{A}\nabla v_i^\e\cdot\nabla
  v_i^\e\dy=2\int_{\Phi^{-1}(B_{\bar{r}}^+)}\abs{\nabla\varphi_i^\e}^2\dx\leq
 2\lambda_i^\e\leq 2\lambda_{n_0}.
\]
This, together with \eqref{eq:rough_estim_2} and
\eqref{eq:rough_estim_1}, implies
\eqref{eq:rough_estim_th3}. Moreover, \eqref{eq:rough_estim_th1}
follows from \eqref{eq:energy_estim_3} and \eqref{eq:rough_estim_th3},
while \eqref{eq:rough_estim_th2} comes as a consequence of Lemma
\ref{lemma:poinc_pert}, \eqref{eq:rough_estim_th1} and
\eqref{eq:rough_estim_th3}.
\end{proof}

The following result is a straightforward consequence of the previous two propositions.

\begin{corollary}\label{cor:up_low_H}
Let $\tau\in(0,1/2)$. Then for any $K\geq \kappa$
there exist $\bar{C},q,\tilde \e>0$ such that
	\begin{equation}\label{eq:up_low_H_th1}
		H(v_{n_0}^\e,K\e)\geq \bar C\e^q\quad\text{for all
                }\e\in(0,\tilde \e).
	\end{equation}
	Moreover, letting $ M_\tau$ be as in Lemma
        \ref{l:poin-refined}, $\beta$ as in \eqref{eq:def_beta} and
        $K\geq  M_\tau \kappa$, we have that, for all $i\in\{1,\dots,n_0\}$, 
	\begin{equation}\label{eq:up_low_H_th2}
		H(v_i^\e,K\e)=O(\e^{\beta(1-\tau)})\quad\text{as }\e\to 0.
	\end{equation}
\end{corollary}
\begin{proof}
  If we integrate \eqref{eq:energy_estim_5} between $K\e$ and
  $R_1$ we obtain that
\[
  \frac{H(v_{n_0}^\e,R_1)}{H(v_{n_0}^\e,K\e)}\leq \left(
    \frac{ R_1 e^{R_1}}{K} \right)^{\!\!c_1}\e^{-c_1}.
\]
Then, in view of Lemma \ref{l:Hpos} point $(ii)$, \eqref{eq:up_low_H_th1} follows with
	\[
		\bar{C}:=C_{R_1}\left( \frac{K}{R_1e^{R_1}} \right)^{c_1}\quad\text{and}\quad q:=c_1.
	\]
	Finally \eqref{eq:up_low_H_th2} directly comes from Proposition \ref{prop:rough_estim}.
\end{proof}

\section{Upper bound on \texorpdfstring{$\lambda_{n_0}-\lambda_{n_0}^\varepsilon$}{eigenvalue variation} }\label{sec:upper-bound-texorpdf}

Hereafter we fix $\tau\in(0,1/2)$ and
\begin{equation*}\label{eq:def_K_rho}
	K_\tau>2\kappa M_\tau
      \end{equation*}
with $\kappa$ as in Lemma \ref{l:rho-eps} and $ M_\tau$ as in Lemma \ref{l:poin-refined}. For convenience in the exposition, hereafter we denote
\begin{equation}\label{eq:62}
	\Theta_{r}:=\Phi^{-1}(B_r^+)
\end{equation}
for any $r\in(0,r_1)$, with $\Phi$ as in \eqref{eq:6}.

For every $i\in\{1,\dots,n_0\}$, $R\geq K_\tau$ and
$\e\in(0,\min\{\e_1,\tilde{r}/R\})$ we consider the following
minimization problem
\begin{equation}\label{eq:min_xi_int}
	\min\left\{ \int_{\Theta_{R\e}}\abs{\nabla u}^2\dx\colon u\in H^1(\Theta_{R\e}),~u-(\eta_{R\e}\circ\Phi) \varphi_i^\e\in H^1_0(\Theta_{R\e}) \right\},
\end{equation}
where $\eta_{R\e}(x)=\eta_R(x/\e)$ and $\eta_R$ is as in
\eqref{eq:def_cutoff}. By standard variational methods, it is easy to
prove that this problem has a unique solution $\xint{i}$, which weakly
satisfies
\[
\begin{cases}
  -\Delta\xint{i}=0, &\text{in }\Theta_{R\e}, \\
  \xint{i}=\varphi_i^\e, &\text{on }(\partial\Theta_{R\e})^+, \\
  \xint{i}=0, &\text{on }(\partial\Theta_{R\e})^0,
\end{cases}
\]
where
\begin{equation*}
	(\partial\Theta_{R\e})^+:=\partial\Theta_{R\e}\cap\Omega\quad\text{and}\quad (\partial\Theta_{R\e})^0:=\partial\Theta_{R\e}\cap\partial\Omega.
\end{equation*}

\begin{lemma}\label{lemma:estim_xint}
	For any $R\geq K_\tau$ the following estimates hold as $\e\to 0$
\begin{gather}
          \int_{\Theta_{R\e}}\abs{\nabla\xint{i}}^2\dx=O(\e^{N-2}H({v_i^\e},K_\tau\e)), \label{eq:estim_xint_th1}\\
          \int_{\Theta_{R\e}}\abs{\xint{i}}^2\dx=O(\e^{N}H({v_i^\e},K_\tau\e)),  \label{eq:estim_xint_th2} \\
          \int_{(\partial\Theta_{R\e})^+}\abs{\xint{i}}^2\ds=O(\e^{N-1}H({v_i^\e},K_\tau\e)), \label{eq:estim_xint_th3}
	\end{gather}
	together with
	\begin{gather}
		\int_{\Theta_{R\e}}\abs{\nabla\xint{i}}^2\dx=O(\e^{N-2+\beta(1-\tau)}),  \label{eq:estim_xint_th4}\\
		\int_{\Theta_{R\e}}\abs{\xint{i}}^2\dx=O(\e^{N+\beta(1-\tau)}),  \label{eq:estim_xint_th5}\\
		\int_{(\partial\Theta_{R\e})^+}\abs{\xint{i}}^2\ds=O(\e^{N-1+\beta(1-\tau)}), \label{eq:estim_xint_th6}
	\end{gather}
	for all $i\in\{1,\dots,n_0\}$, where $\beta$ is defined in \eqref{eq:def_beta}.
\end{lemma}
\begin{proof}
	By the change of variable induced by the diffeomorphism
        $\Phi$, problem \eqref{eq:min_xi_int} is equivalent to
	\[
          \min\left\{ \int_{B_{R\e}^+}A\nabla u\cdot\nabla u\dy\colon
            u\in H^1(B_{R\e}^+),~ u-\eta_{R\e} u_i^\e\in
            H^1_0(B_{R\e}^+) \right\},
	\]
	with $A$ as in \eqref{eq:A},
	and the minimum is attained by $\xint{i}\circ\Phi^{-1}$. If
        one tests the problem above with $u=\eta_{R\e} u_i^\e$, the
        following is obtained, in view also of \eqref{eq:bound_A},
        \eqref{eq:bound_p} and \eqref{eq:def_cutoff},  
\begin{align*}
  \int_{\Theta_{R\e}}\abs{\nabla\xint{i}}^2\dx & =\int_{B_{R\e}^+}A\nabla (\xint{i}\circ\Phi^{-1})\cdot
                                                 \nabla (\xint{i}\circ\Phi^{-1})\dy \\
                                               &\leq \frac{3}{2}\int_{B_{R\e}^+}\abs{\eta_{R\e}\nabla u_i^\e
                                                 +u_i^\e\nabla\eta_{R\e}}^2\dy \\
                                               &\leq
                                                 96\int_{B_{R\e}^+}\left(A\nabla
                                                 u_i^\e\cdot\nabla
                                                 u_i^\e
                                                 +\frac{1}{(R\e)^2}\,p\abs{u_i^\e}^2\right)\dy.
\end{align*}
Combining this estimate with \eqref{eq:energy_estim_th1} and
\eqref{eq:energy_estim_th2} proves \eqref{eq:estim_xint_th1}, while
combining it with \eqref{eq:rough_estim_th1} and
\eqref{eq:rough_estim_th2} proves \eqref{eq:estim_xint_th4}. Since
$\xint{i}=\varphi_i^\e$ on $(\partial\Theta_{R\e})^+$, estimates
\eqref{eq:estim_xint_th3} and \eqref{eq:estim_xint_th6} are trivial in
view of \eqref{eq:energy_estim_th3} and
\eqref{eq:rough_estim_th3}. Finally, \eqref{eq:estim_xint_th2} and
\eqref{eq:estim_xint_th5} come from 
the other estimates, Lemma
\ref{lemma:poinc_pert}, and the change of variable induced by the diffeomorphism
        $\Phi$.
\end{proof}
Using the functions $\xint{i}$ that solve \eqref{eq:min_xi_int}, we
construct a family of competitors (see \eqref{eq:def_hat_xi}) to test
the Rayleigh quotient for $\lambda_{n_0}$ and obtain a sharp estimate
from above of the eigenvalue variation
$\lambda_{n_0}-\lambda_{n_0}^\e$. To this aim, we also provide
suitable energy estimates for such competitors, see
\eqref{eq:estim_hat_xi_th1}--\eqref{eq:estim_hat_xi_th6}.  For all
$i\in\{1,\dots,n_0\}$, $R\geq K_\tau$ and
$\e\in(0,\min\{\e_1,\tilde{r}/R\})$ we define
\begin{equation}\label{eq:def_xi}
	\xi_{i,R,\e}(x):=\begin{cases}
						\varphi_i^\e(x),&\text{if }x\in\Omega\setminus \Theta_{R\e}, \\
						\xint{i}(x), &\text{if }x\in\Theta_{R\e}.
					\end{cases}
\end{equation}
We observe that $\xi_{i,R,\e}\in H^1_0(\Omega)$ thanks to the fact
that $R\geq K_\tau>2\kappa$, which guarantees that
$\e\mathcal V\subset \Theta_{\frac{R\e}2}$. Moreover it is easy to verify that
the family $\{\xi_{1,R,\e},\dots,\xi_{n_0,R,\e}\}$ is linearly
independent in $H^1_0(\Omega)$, for $\e$ sufficiently small. We also
define  
\begin{equation}\label{eq:def_Z_Ups}
  Z_R^\e(x):=
  \frac{(\xint{n_0}\circ \Phi^{-1})(\e
    x)}{\sqrt{H(\e)}},
  \quad \Upsilon^\e(x):=\frac{u_{n_0}^\e(\e x)}{\sqrt{H(\e)}},
\end{equation}
where we denote
\begin{equation}\label{eq:61}
	H(\e):=\frac12H(v_{n_0}^\e,K_\tau\e)=\frac{1}{(K_\tau\varepsilon)^{N-1}}\int_{S_{K_\tau\e}^+}\mu |u_{n_0}^\varepsilon|^2\ds.
\end{equation}
As a consequence of the estimates given in Propositions 
\ref{prop:energy_estim} and \ref{prop:rough_estim}
and Lemma \ref{lemma:estim_xint}, we are able
to prove the following result.

\begin{lemma}\label{lemma:estim_xi}
	For any $R\geq K_\tau$ we have that, as $\e\to 0$,
	\begin{gather}
          \int_{\Omega}\abs{\nabla \xi_{n_0,R,\e}}^2\!\dx=
          \lambda_{n_0}^\e+\e^{N-2}{H(\e)}\left(
            \int_{B_R^+}\!\!A(\e y)\nabla Z_R^\e\!\cdot\!\nabla
            Z_R^\e\!\dy-\int_{B_R^+}\!\!A(\e y)
            \nabla \Upsilon^\e\!\cdot\!\nabla \Upsilon^\e\!\dy \right), \label{eq:estim_xi_th1} \\
          \int_\Omega\abs{\nabla
            \xi_{i,R,\e}}^2\dx=\lambda_i^\e+O(\e^{N-2+\beta(1-\tau)}),
          \quad\text{for all }i=1,\dots,n_0, \label{eq:estim_xi_th2}\\
          \int_\Omega\nabla\xi_{i,R,\e}\cdot\nabla\xi_{n_0,R,\e}\dx=O\big(\e^{N-2+\frac{\beta}{2}(1-\tau)}
          \sqrt{H(\e)}\big),\quad\text{for all }i=1,\dots,n_0-1, \label{eq:estim_xi_th3}\\
          \int_\Omega\nabla\xi_{i,R,\e}\cdot\nabla\xi_{j,R,\e}\dx=O(\e^{N-2+\beta(1-\tau)}),\quad\text{for all }i,j=1,\dots,n_0,~i\neq j,\label{eq:estim_xi_th4} \\
          \int_\Omega\abs{\xi_{n_0,R,\e}}^2\dx=1+O(\e^N {H(\e)}),\label{eq:estim_xi_th5} \\
          \int_\Omega\abs{\xi_{i,R,\e}}^2\dx=1+O(\e^{N+\beta(1-\tau)}),\quad\text{for all }i=1,\dots,n_0, \label{eq:estim_xi_th6}\\
          \int_\Omega\xi_{i,R,\e}\,\xi_{n_0,R,\e}\dx=O(\e^{N+\frac{\beta}{2}(1-\tau)}\sqrt{ {H(\e)}}),\quad\text{for all }i=1,\dots,n_0-1,,\label{eq:estim_xi_th7} \\
          \int_\Omega\xi_{i,R,\e}\,\xi_{j,R,\e}\dx=O(\e^{N+\beta(1-\tau)}),\quad\text{for
            all }i,j=1,\dots,n_0,~i\neq j. \label{eq:estim_xi_th8}
	\end{gather}
\end{lemma}
\begin{proof}
	 By definition of $\xi_{n_0,R,\e}$ we have
	\[
		\int_{\Omega}\abs{\nabla
                  \xi_{n_0,R,\e}}^2\dx=\int_\Omega\abs{\nabla\varphi_{n_0}^\e}^2\dx-
                \int_{\Theta_{R\e}}\abs{\nabla{\varphi_{n_0}^\e}}^2\dx
                +\int_{\Theta_{R\e}}\abs{\nabla\xint{n_0}}^2\dx.
	\]
        Since, by {\eqref{eq:ortonor}--~\eqref{eq:eqphiie},}
        $\int_\Omega\abs{\nabla\varphi_{n_0}^\e}^2\dx=\lambda_{n_0}^\e$,
	by the change of variable $y=\Phi(x)$ and the definition of
        $Z_R^\e$ and $\Upsilon^\e$ given in \eqref{eq:def_Z_Ups}, we
        obtain \eqref{eq:estim_xi_th1}.
  Similarly, for any $i=1,\dots,n_0$,
	\[
		\int_{\Omega}\abs{\nabla\xi_{i,R,\e}}^2\dx=\lambda_i^\e-\int_{\Theta_{R\e}}\abs{\nabla\varphi_i^\e}^2\dx+\int_{\Theta_{R\e}}\abs{\nabla\xint{i}}^2\dx.
	\]
	Then \eqref{eq:estim_xi_th2} follows from
        \eqref{eq:rough_estim_th1} and \eqref{eq:estim_xint_th4}.
        
        For all $i=1,\dots,{n_0-1}$ we have that
	\[
		\int_\Omega\nabla\xi_{i,R,\e}\cdot\nabla\xi_{n_0,R,\e}\dx=-\int_{\Theta_{R\e}}\nabla\varphi_i^\e\cdot\nabla\varphi_{n_0}^\e\dx+\int_{\Theta_{R\e}}\nabla\xint{i}\cdot\nabla\xint{n_0}\dx,
	\]
	since the perturbed eigenfunctions are orthogonal, therefore
        \eqref{eq:estim_xi_th3} follows from Cauchy-Schwartz inequality
        and estimates \eqref{eq:energy_estim_th1},
        \eqref{eq:rough_estim_th1}, \eqref{eq:estim_xint_th1} and
        \eqref{eq:estim_xint_th4}. Finally, again by orthogonality,
        for any $i,j=1,\dots,n_0$, $i\neq j$, we have that
	\[
		\int_\Omega\nabla\xi_{i,R,\e}\cdot\nabla\xi_{j,R,\e}\dx=-\int_{\Theta_{R\e}}\nabla\varphi_i^\e\cdot\nabla\varphi_j^\e\dx+\int_{\Theta_{R\e}}\nabla\xint{i}\cdot\nabla\xint{j}\dx
	\]
	and so \eqref{eq:estim_xi_th4} easily follows from Cauchy-Schwartz inequality and estimates \eqref{eq:rough_estim_th1} and \eqref{eq:estim_xint_th4}. The proof of \eqref{eq:estim_xi_th5}--\eqref{eq:estim_xi_th8} is completely analogous and it is therefore omitted.
\end{proof}

We now construct an orthogonal basis
$\{\hat{\xi}_{1,R,\e},\dots,\hat{\xi}_{n_0,R,\e}\}$ of the space
$\textrm{span}\,\{\xi_{1,R,\e},\dots,\xi_{n_0,R,\e}\}$. To this aim we
recursively define
\begin{equation}\label{eq:def_hat_xi}
	\hat{\xi}_{n_0,R,\e}:=\xi_{n_0,R,\e}\quad\text{and}\quad
        \hat{\xi}_{i,R,\e}:=\xi_{i,R,\e}-\sum_{j=i+1}^{n_0}d_{i,j}^{R,\e}\hat{\xi}_{j,R,\e}
        \quad\text{for }i=1,\dots,n_0-1,
\end{equation}
where
\begin{equation*}
	d_{i,j}^{R,\e}:=\frac{\int_\Omega \xi_{i,R,\e}\hat{\xi}_{j,R,\e}\dx}{\int_\Omega|\hat{\xi}_{j,R,\e}|^2\dx}.
\end{equation*}
 The functions
$\{\hat{\xi}_{1,R,\e},\dots,\hat{\xi}_{n_0,R,\e}\}$ are orthogonal in
$L^2(\Omega)$. Moreover they satisfy the following estimates:
\begin{gather}
	\int_{\Omega}|\nabla \hat{\xi}_{n_0,R,\e}|^2\!\dx=\lambda_{n_0}^\e\!+\e^{N-2}{H(\e)}\!\left( \int_{B_R^+}\!\!A(\e y)\nabla Z_R^\e\!\cdot\!\nabla Z_R^\e\dy-\!\int_{B_R^+}\!\!A(\e y)\nabla \Upsilon^\e\cdot\nabla \Upsilon^\e\dy \right), \label{eq:estim_hat_xi_th1} \\
	\int_\Omega|\nabla \hat{\xi}_{i,R,\e}|^2\dx=\lambda_i^\e+O(\e^{N-2+\beta(1-\tau)}),\quad\text{for all }i=1,\dots,n_0, \label{eq:estim_hat_xi_th2}\\
	\int_\Omega\nabla\hat{\xi}_{i,R,\e}\cdot\nabla\hat{\xi}_{n_0,R,\e}\dx=O(\e^{N-2+\frac{\beta}{2}(1-\tau)}\sqrt{{H(\e)}}),\quad\text{for all }i=1,\dots,n_0-1, \label{eq:estim_hat_xi_th3}\\
	\int_\Omega\nabla\hat{\xi}_{i,R,\e}\cdot\nabla\hat{\xi}_{j,R,\e}\dx=O(\e^{N-2+\beta(1-\tau)}),\quad\text{for all }i,j=1,\dots,n_0,~i\neq j,\label{eq:estim_hat_xi_th4} \\
	\int_\Omega|\hat\xi_{n_0,R,\e}|^2\dx=1+O(\e^N {H(\e)}),\label{eq:estim_hat_xi_th5} \\
	\int_\Omega|\hat{\xi}_{i,R,\e}|^2\dx=1+O(\e^{N+\beta(1-\tau)}),\quad\text{for all }i=1,\dots,n_0. \label{eq:estim_hat_xi_th6}
\end{gather}
The proof of estimates \eqref{eq:estim_hat_xi_th1}--\eqref{eq:estim_hat_xi_th6}
consists in direct computations and comes from Lemma
\ref{lemma:estim_xi} and the following estimates on the coefficients $d_{i,j}^{R,\e}$
\begin{gather*}
	d_{j,n_0}^{R,\e}=O(\e^{N+\frac{\beta}{2}(1-\tau)}\sqrt{{H(\e)}})\quad\text{for all }j=1,\dots,n_0-1, \\
	d_{j,k}^{R,\e}=O(\e^{N+\beta(1-\tau)})\quad\text{for all }k=2,\dots,n_0, ~j<k.
\end{gather*}
We are now ready to prove an upper bound of the eigenvalue variation $\lambda_{n_0}-\lambda_{n_0}^\e$.

\begin{proposition}\label{prop:upper_bound}
	For any $R\geq K_\tau$ we have that
	\begin{equation}\label{eq:upper_bound_th1}
          \lambda_{n_0}-\lambda_{n_0}^\e\leq \e^{N-2}
{H(\e)}(f_R(\e)+o(1))\quad\text{as }\e\to 0,
	\end{equation}
	where
	\begin{equation}\label{eq:def_f_R_eps}
		f_R(\e)=\int_{B_R^+}A(\e y)\nabla Z_R^\e\cdot\nabla Z_R^\e\dy-\int_{B_R^+}A(\e y)\nabla \Upsilon^\e\cdot\nabla \Upsilon^\e\dy.
	\end{equation}
	Moreover
	\begin{equation}\label{eq:upper_bound_O}
		f_R(\e)=O(1)\quad\text{as }\e\to 0.
	\end{equation}
\end{proposition}
\begin{proof}
 By the Courant-Fischer Min-Max variational characterization of the
eigenvalues, see \eqref{eq:min_max}, we have that 
\begin{equation*}
		\lambda_{n_0} =\min \left\{
		\max_{\substack{a_1,\ldots,a_{n_0}\in\R \\ \sum_{i=1}^{n_0} a_i^2=1 }} 
		\frac{\left\|\nabla \left( \sum_{i=1}^{n_0} a_i u_i
			\right)\right\|_{L^2(\Omega)}^2}{\|\sum_{i=1}^{n_0} a_i u_i
			\|_{L^2(\Omega)}^2}:\ 
		\begin{minipage}{3.5cm}
			\quad\\[5pt]
			$\{u_1\ldots,u_{n_0}\} \subset  H^1_0(\Omega)$ \\ linearly independent
		\end{minipage}
		\right\}.
	\end{equation*}
	We test the above minimization problem with the orthonormal family
	\[
		\left\{ u_i:=\frac{\hat{\xi}_{i,R,\e}}{\lVert\hat{\xi}_{i,R,\e}\rVert_{L^2(\Omega)}} \right\}_{i=1,\dots,n_0},
	\]
where $\hat{\xi}_{i,R,\e}$ is defined in \eqref{eq:def_hat_xi}; we thus obtain
\begin{equation*}
  \lambda_{n_0}-\lambda_{n_0}^\e\leq
  \max_{\substack{a_1,\ldots,a_{n_0}\in\R \\
      \sum_{i=1}^{n_0} a_i^2=1 }}\int_\Omega
  \bigg|\nabla\bigg(\sum_{i=1}^{n_0}
  a_i\tfrac{\hat{\xi}_{i,R,\e}}{\lVert\hat{\xi}_{i,R,\e}\rVert_{L^2(\Omega)}}
  \bigg)\bigg|^2\!\dx-\lambda_{n_0}^\e=\max_{\substack{a_1,\ldots,a_{n_0}\in\R
      \\ \sum_{i=1}^{n_0} a_i^2=1 }}\sum_{i,j=1}^{n_0}M_{i,j}^\e a_ia_j,
\end{equation*}
	with
\[
  M_{i,j}^\e=\frac{\int_{\Omega}\nabla \hat \xi_{i,R,\e}\cdot\nabla
    \hat \xi_{j,R,\e}\dx}{\|\hat \xi_{i,R,\e}\|_{L^2(\Omega)}\|\hat
    \xi_{j,R,\e}\|_{L^2(\Omega)}} -{\lambda_{n_0}^\e}\delta_i^j.
\]
From estimates \eqref{eq:estim_hat_xi_th1}--\eqref{eq:estim_hat_xi_th6} we deduce the behaviour of the coefficients $M_{i,j}^\e$'s as $\e\to 0$, that is
\begin{align*}
  &M_{n_0,n_0}^\e=\e^{N-2}{H(\e)}(f_R(\e)+o(1)), \\
  &M_{i,i}^\e=\lambda_i^\e-\lambda_{n_0}^\e+o(1),\quad\text{for all }i=1,\dots,n_0-1, \\
  &M_{i,n_0}^\e=O(\e^{N-2+\frac{\beta}{2}(1-\tau)}\sqrt{
    H(\e) }),
  \quad\text{for all }i=1,\dots,n_0-1, \\
  &M_{i,j}^\e=O(\e^{N-2+\beta(1-\tau)}),\quad\text{for all
  }i,j=1,\dots,n_0-1,~i\neq j.
\end{align*}
Moreover, \eqref{eq:energy_estim_th1} and \eqref{eq:estim_xint_th1}
yield that $f_R(\e)=O(1)$ as $\e\to 0$, while from Corollary
\ref{cor:up_low_H} we have that, for all $\e \in (0,\tilde\e)$,
$H(\e)\geq \bar{C}\e^q$ for some
$\bar{C},q,\tilde \e>0$. Therefore the assumptions of
Lemma \ref{lemma:quadratic_form} are fulfilled with
\begin{gather*}
  \sigma(\e)=\e^{N-2}H(\e),\quad\mu(\e)=f_R(\e)+o(1), \quad\text{as }\e\to 0,\\
  \quad a=\frac{1}{2}(N-2+\beta(1-\tau))>0,\quad
  M>\frac{N-2+q}{a}-2.
\end{gather*}
Hence
\[
  \max_{\substack{a_1,\ldots,a_{n_0}\in\R \\
      \sum_{i=1}^{n_0} a_i^2=1
    }}\sum_{i,j=1}^{n_0}M_{i,j}^\e a_ia_j=\e^{N-2}H(\e)(f_R(\e)+o(1))
\]
as $\e\to 0$ and the proof of \eqref{eq:upper_bound_th1} is
complete. As already observed, \eqref{eq:upper_bound_O} is a consequence of
\eqref{eq:energy_estim_th1} and \eqref{eq:estim_xint_th1}.	
\end{proof}

\section{Lower bound on \texorpdfstring{$\lambda_{n_0}-\lambda_{n_0}^\varepsilon$}{eigenvalue variation}}\label{sec:lower-bound-texorpdf}

In this section we provide a lower bound for the eigenvalue variation
$\lambda_{n_0}-\lambda_{n_0}^\e$. In order to do this, we first
construct a family of competitors for the Rayleigh quotient of
$\lambda_{n_0}^\e$; then, exploiting the local energy estimates stated
in \Cref{lem:w_n_0_est}, we prove a blow-up result for their scaling,
see \Cref{lem:U_R_conv}.

Recalling the definition of $\psi_{i}$ given in \eqref{eq:psi_i_def}, from
\eqref{eq:gamma_i_def} 
and 
\eqref{eq:11}--\eqref{eq:10}
we easily deduce that 
\begin{equation}\label{eq:52}
  \frac{(\varphi_{i}\circ \Phi^{-1})(rx)}{r^{\gamma_i}} \to \psi_i \quad\text{in } H^1(B_{R})\text{ as } r\to 0, \text{ for every } R>0.
\end{equation}
From \eqref{eq:52} and \Cref{l:poin} we deduce the following estimates for $\varphi_i\circ\Phi^{-1}$, $i=1,\ldots,n_0$.

\begin{lemma}\label{lemma:phi_i_est}
There exists $\tilde{C}>0$ such that, for all $i\in \{1,\ldots,n_0-1\}$,
for all $R>1$ and $\varepsilon\in(0,\frac{r_1}{R})$,
\begin{align*}
  &\|\nabla(\varphi_i \circ \Phi^{-1})\|^2_{L^2(B_{R\varepsilon}^+)}
    \leq \tilde{C} (R\varepsilon)^N,\\
  &\|\varphi_i \circ \Phi^{-1}\|^2_{L^2(B_{R\varepsilon}^+)}
    \leq \tilde{C} (R\varepsilon)^{N+2},\\
  &\|\varphi_i \circ \Phi^{-1}\|^2_{L^2(S_{R\varepsilon}^+)}
    \leq \tilde{C} (R\varepsilon)^{N+1},
\end{align*}
where, for $r>0$, $B_r^+$ and $S_r^+$ are defined in \eqref{eq:notation},
and
\begin{align*}
  &\|\nabla(\varphi_{n_0} \circ \Phi^{-1})\|^2_{L^2(B_{R\varepsilon}^+)}
    \leq \tilde{C} (R\varepsilon)^{N+2\gamma-2},\\
  &\|\varphi_{n_0} \circ \Phi^{-1}\|^2_{L^2(B_{R\varepsilon}^+)}
    \leq \tilde{C} (R\varepsilon)^{N+2\gamma},\\
  &\|\varphi_{n_0} \circ \Phi^{-1}\|^2_{L^2(S_{R\varepsilon}^+)}
    \leq \tilde{C} (R\varepsilon)^{N-1 +2\gamma},
\end{align*}
where $\gamma$ is defined in \eqref{eq:notation-n_0}.
\end{lemma}
\noindent For every $R>\kappa$ and $0<\e<\frac{r_1}{R}$, we define
\begin{equation}\label{eq:w_j0_def}
w_{n_0,R,\varepsilon}(x)=
\begin{cases}
\varphi_{n_0}(x), \quad & \text{if } x\in \Omega\setminus
{\Theta_{R\varepsilon}}, \\
w_{n_0,R,\varepsilon}^{\text{int}}(x), \quad & \text{if } x\in { \Theta_{R\varepsilon}},
\end{cases}
\end{equation}
where the notation $\Theta_{R\varepsilon}$ has been
  introduced in \eqref{eq:62},
\[
w_{n_0,R,\varepsilon}^{\text{int}}(x)=
w_{R,\varepsilon}(\Phi(x))
\]
and 
 $w_{R,\varepsilon}$ is the unique $H^1(B_{R\varepsilon}^+)$-function satisfying
$w_{R,\varepsilon} -(\varphi_{n_0}\circ\Phi^{-1})\in 
H^1_{0,\partial B_{R\varepsilon}^+\setminus \widetilde
  \Sigma_\e}(B_{R\varepsilon}^+)$ and achieving
\[
\min \left\{ \|\nabla w\|^2_{L^2(B_{R\varepsilon}^+)} :  w \in
  H^1(B_{R\varepsilon}^+),\ w-(\varphi_{n_0}\circ\Phi^{-1})\in 
H^1_{0,\partial B_{R\varepsilon}^+\setminus \widetilde \Sigma_\e}(B_{R\varepsilon}^+) \right\}.
\]
By classical variational methods, $w_{R,\varepsilon}$ exists and it is
the unique solution to 
\begin{equation}\label{eq:w_int_eq}
  \begin{cases}
w_{R,\varepsilon} -(\varphi_{n_0}\circ\Phi^{-1})\in 
H^1_{0,\partial B_{R\varepsilon}^+\setminus \widetilde
  \Sigma_\e}(B_{R\varepsilon}^+),\\
    \int_{B_{R\varepsilon}^+} \nabla w_{R,\varepsilon}\cdot \nabla
    \phi \, \dx =0 \text{ for every } \phi \in H^1_{0,\partial
      B_{R\varepsilon}^+\setminus \widetilde
      \Sigma_\e}(B_{R\varepsilon}^+).
  \end{cases}
\end{equation}
As a consequence of Lemma \ref{lemma:phi_i_est} and  \Cref{l:poin} we obtain the following estimates.

\begin{lemma}\label{lem:w_n_0_est}
Letting $\tilde{C}>0$ be as in \Cref{lemma:phi_i_est},
  we have that,
for all $R>\kappa$ and $\varepsilon\in(0,\frac{r_1}{R})$,
\begin{align*}
&\|\nabla w_{R,\varepsilon}\|^2_{L^2(B_{R\varepsilon}^+)} \leq \tilde{C} (R\varepsilon)^{N+2\gamma-2},\\
&\|w_{R,\varepsilon}\|^2_{L^2(B_{R\varepsilon}^+)} \leq \tilde{C} (R\varepsilon)^{N+2\gamma},\\
&\|w_{R,\varepsilon}\|^2_{L^2(S_{R\varepsilon}^+)} \leq \tilde{C} (R\varepsilon)^{N+2\gamma-1}.
\end{align*}
\end{lemma}
\noindent For every $R>\kappa$, $0<\e<\frac{r_1}{R}$, and $x\in B_{R}^+$, we let
\begin{equation}\label{eq:URe}
U_{R,\varepsilon}(x)=\frac{w_{R,\e}(\varepsilon x)}{\varepsilon^{\gamma}}.
\end{equation}

\begin{lemma}\label{lem:U_R_conv}
For all $R>\kappa$, $\lim_{\varepsilon\to0} \|U_{R,\varepsilon}-U_R \|_{H^1(B_R^+)}=0$, where $U_R$ is as in Lemma \ref{lemma:U_R}.
\end{lemma}
\begin{proof}
Let $R>\kappa$. From a change of variables and Lemma \ref{lem:w_n_0_est} we have
\[
\|\nabla U_{R,\varepsilon}\|^2_{L^2(B_R^+)} =
\varepsilon^{-N-2\gamma+2} \|\nabla w_{R,\varepsilon}
\|^2_{L^2(B_{R\varepsilon}^+)}\leq
\tilde{C}R^{N+2\gamma-2}
\]
and 
\[
\int_{S_R^+}U_{R,\varepsilon}^2\ds=\e^{1-N-2\gamma}\int_{S_{R\e}^+}w_{R,\e}^2\ds\leq \tilde{C}
R^{N+2\gamma-1}
\]
for all $\varepsilon\in(0,r_1/R)$,
so that the family $\{U_{R,\varepsilon}\}_{\varepsilon\in(0,r_1/R)}$ is bounded
in $H^1(B_R^+)$ in view of  \Cref{l:poin}.

We deduce that there exist $W\in H^1(B_R^+)$ and a sequence $\varepsilon_n\to0$ such that
$U_{R,\varepsilon_n} \rightharpoonup W$ weakly in $H^1(B_R^+)$ as
$n\to\infty$. Letting 
 \begin{equation}\label{eq:Ve}
V_\e(x)= \frac{(\varphi_{n_0}\circ \Phi^{-1})(\e x)}{\e^{\gamma}},
\end{equation}
from \eqref{eq:52} we have that 
\begin{equation}\label{eq:56}
V_\e\to \psi \quad\text{in $H^1(B_R^+)$ as }
\e\to 0.
\end{equation}
 Hence 
\[
U_{R,\e_n}-V_{\e_n} \rightharpoonup W-\psi\quad\text{as
  $n\to\infty$ in $H^1(B_R^+)$}.
\]
Since \eqref{eq:w_int_eq} yields that 
$U_{R,\e}-V_\e\in  H^1_{0,\partial B_{R}^+\setminus (\frac1\e\widetilde
  \Sigma_\e)}(B_{R}^+)$, the above convergence and 
\Cref{rem:Mosco} imply that 
\[
W-\psi\in H^1_{0,\partial B_{R}^+\setminus \Sigma}(B_{R}^+).
\]
We observe that the equation satisfied by $U_{R,\varepsilon}$ is 
\begin{equation}\label{eq:U_R_eps_eq}
\int_{B_R^+} \nabla U_{R,\varepsilon}\cdot\nabla\phi \dx=0
\quad \text{for every } \phi \in H^1_{0,\partial B_{R}^+\setminus (\frac1\e\widetilde
  \Sigma_\e)}(B_{R}^+).
\end{equation}
Let $\phi\in C^\infty_{\rm c}(B_R^+\cup \Sigma)$.
  From \eqref{eq:11}--\eqref{eq:10} we easily deduce that
$\phi\in C^\infty_{\rm c}(B_R^+\cup (\frac1\e\widetilde\Sigma_\e))$
for $\e$ sufficiently small, hence $\phi\in H^1_{0,\partial B_{R}^+\setminus (\frac1\e\widetilde
  \Sigma_\e)}(B_{R}^+)$ and \eqref{eq:U_R_eps_eq} is satisfied for
$\e=\e_n$ and large $n$. Therefore
we can pass to the limit to infer that $\int_{B_R^+} \nabla W
\cdot\nabla\phi \dx=0$ for every $\phi \in C^\infty_{\rm c}(B_R^+\cup
\Sigma)$ and then, by density, for all $\phi\in H^1_{0,\partial B_{R}^+\setminus \Sigma}(B_{R}^+)$.
By uniqueness of the solution to \eqref{eq:U_R_eq}, we conclude that
$W=U_R$. Since the limit $U_R$ is the same along every subsequence,
the Urysohn's Subsequence Principle implies that the whole family
$U_{R,\varepsilon}$ weakly converges to $U_R$ in $H^1(B_R^+)$ as $\e\to0$.

It remains to show the strong $H^1$-convergence. To this aim, we choose in \eqref{eq:U_R_eps_eq}
$\phi= U_{R,\varepsilon}-V_\e$ for all $\varepsilon\in(0,r_1/R)$. We obtain
\[
\int_{B_R^+} |\nabla U_{R,\varepsilon}|^2 \dx =\int_{B_R^+} \nabla U_{R,\varepsilon} \cdot \nabla V_\varepsilon \dx\to
\int_{B_R^+} \nabla U_R\cdot \nabla \psi \dx 
=\int_{B_R^+} |\nabla U_R|^2\dx,
\]
as $\varepsilon\to0$, where the convergence above is justified by
the fact that $U_{R,\varepsilon_n} \rightharpoonup U_R$ weakly in
$H^1(B_R^+)$ and $V_\varepsilon \to \psi$ strongly
in $H^1(B_R^+)$, and the last equality follows from \eqref{eq:U_R_eq}.
\end{proof}

From \eqref{eq:w_j0_def}, \eqref{eq:URe}, and \eqref{eq:Ve} it follows
that, for all $R>\kappa$ fixed, 
\begin{align}\label{eq:48}
  \int_\Omega &|\nabla w_{n_0,R,\e}(x)|^2\dx=
 \lambda_{n_0}-\int_{\Theta_{R\e}}|\nabla\varphi_{n_0}(x)|^2\dx+
 \int_{\Theta_{R\e}}|\nabla w_{n_0,R,\varepsilon}^{\text{int}}(x)|^2\dx\\
  \notag&=\lambda_{n_0}-\e^{N-2+2\gamma}\int_{B_R^+}\Big(A(\e
          y)\nabla V_\e(y)\cdot \nabla V_\e(y)-A(\e
          y)\nabla U_{R,\e}(y)\cdot \nabla U_{R,\e}(y)\Big)\,dy,\\
                        \notag&=\lambda_{n_0}-\e^{N-2+2\gamma}\bigg(
                                \int_{B_R^+}\Big(|\nabla V_\e(y)|^2-|\nabla U_{R,\e}(y)|^2\Big)\,dy+o(1)\bigg),
\end{align}
as $\e\to 0$, where the last estimate follows from
  \eqref{eq:33} and boundedness of $\{V_\e\}_\e$ and $\{U_{R,\e}\}_\e$
  in $H^1(B_R^+)$ (see \Cref{lem:U_R_conv}).

The main goal of this section is to prove the following result.

\begin{proposition}\label{prop:lower_bound}
For all $R>\kappa$, we have that
\begin{equation}\label{eq:lower_bound}
\lambda_{n_0}^\varepsilon-\lambda_{n_0} \leq 
\varepsilon^{N+2\gamma-2}(g_R(\varepsilon)+o(1))
\quad\text{as } \varepsilon\to0,
\end{equation}
where
\begin{equation}\label{eq:55}
g_R(\e)=\int_{B_R^+}\Big(|\nabla U_{R,\e}(y)|^2-|\nabla
  V_\e(y)|^2\Big)\,dy
\end{equation}
with $V_\e$ and $U_{R,\e}$ defined in \eqref{eq:Ve} and \eqref{eq:URe} respectively.
Furthermore $\lim_{\varepsilon\to0} g_R(\varepsilon)=g_R$, where
\begin{equation}\label{eq:def_g_R}
g_R=\|\nabla U_R\|^2_{L^2(B_R^+)} - \|\nabla \psi\|^2_{L^2(B_R^+)}.
\end{equation}
\end{proposition}

\begin{proof}
  We use the  Min-Max characterization of the eigenvalue $\lambda_{n_0}^\varepsilon$ recalled in \eqref{eq:min_max}, that we rewrite as follows
\begin{equation}\label{eq:min_max2}
\lambda_{n_0}^\varepsilon =\min \left\{
\max_{\substack{a_1,\ldots,a_{n_0}\in\R \\ \sum_{i=1}^{n_0} a_i^2=1 }} 
\frac{\left\|\nabla \left( \sum_{i=1}^{n_0} a_i u_i
    \right)\right\|_{L^2(\Omega)}^2}{\|\sum_{i=1}^{n_0} a_i u_i
  \|_{L^2(\Omega)}^2}:\ 
\begin{minipage}{4.5cm}
\quad\\[5pt]
$\{u_1\ldots,u_{n_0}\} \subset  H^1_{0,\partial\Omega\setminus \Sigma_\e}(\Omega)$ \\ linearly independent
\end{minipage}
\right\}.
\end{equation}
Let us fix $R>\kappa$. We define 
\[
\hat{w}_{i,R,\varepsilon}=\varphi_i\quad\text{for all
}i=1,2,\dots,n_0-1,
\]
and 
\[
\hat{w}_{n_0,R,\varepsilon} =w_{n_0,R,\varepsilon} - \sum_{i=1}^{n_0-1} c_i^\varepsilon \varphi_i, 
\]
where 
\[
c_i^\varepsilon=\int_\Omega w_{n_0,R,\varepsilon} \varphi_i \dx.
\]
Let 
\[
\tilde{w}_{i,R,\varepsilon}=\frac{\hat{w}_{i,R,\varepsilon}}{\|\hat{w}_{i,R,\varepsilon}\|_{L^2(\Omega)}},\quad i=1,\dots,n_0.
\]
We note that the family $\{\tilde{w}_{i,R,\varepsilon}\}_{i=1,\ldots,n_0}$ is orthonormal in $L^2(\Omega)$ and linearly independent in $H^1_{0,\partial\Omega\setminus \Sigma_\varepsilon}(\Omega)$.

 Lemmas \ref{lemma:phi_i_est} and \ref{lem:w_n_0_est},
\eqref{eq:w_j0_def} and \eqref{eq:10} imply that, for
  all $i\in\{1,2,\dots,n_0-1\}$, 
\begin{equation}\label{eq:49}
c_i^\varepsilon =O(\e^{N+1+\gamma})
\quad\text{and}\quad
\int_\Omega \nabla w_{n_0,R,\varepsilon}\cdot \nabla\varphi_i \dx =O(\e^{N+\gamma-1})
 \quad\text{as }\e\to0.
\end{equation}
Then, from \eqref{eq:48} we deduce that 
\begin{align}\label{eq:51}
 & \int_\Omega |\nabla \hat w_{n_0,R,\e}(x)|^2\dx\\
  \notag&\quad=
\lambda_{n_0}-\e^{N-2+2\gamma}\bigg(
                                \int_{B_R^+}\Big(|\nabla V_\e(y)|^2-|\nabla U_{R,\e}(y)|^2\Big)\,dy+o(1)\bigg)+O(\e^{2(N+\gamma)})\\
\notag&\quad=\lambda_{n_0}-\e^{N-2+2\gamma}\left(\int_{B_R^+}\Big(|\nabla
  V_\e(y)|^2-|\nabla U_{R,\e}(y)|^2\Big)\,dy+o(1)\right)
\end{align}
as $\e\to0$. Furthermore \eqref{eq:49} implies that, for all $i=1,\dots,n_0-1$, 
\begin{equation}\label{eq:53}
\int_\Omega \nabla \hat w_{n_0,R,\e}(x)\cdot \nabla \hat w_{i,R,\e}(x)\dx=
O(\e^{N+\gamma-1})
 \quad\text{as }\e\to0,
\end{equation}
while \eqref{eq:w_j0_def}, \Cref{lemma:phi_i_est}, \Cref{lem:w_n_0_est},
and \eqref{eq:49} yield
\begin{equation}\label{eq:54}
 \int_\Omega |\hat w_{n_0,R,\e}(x)|^2\dx=1+O(\e^{N+2\gamma}),
 \quad\text{as }\e\to0.
\end{equation}
Choosing as test functions in \eqref{eq:min_max2}
$u_i=\tilde{w}_{i,R,\varepsilon}$ we obtain the following estimate
 \[
    \lambda_{n_0}^\e-\lambda_{n_0}\leq
    \max_{\substack{a_1,\dots,a_{n_0}\in \R \\
        \sum_{i=1}^{n_0}\abs{a_i}^2=1}}
\int_{\Omega}\abs{\nabla \left(\sum_{i=1}^{n_0} a_i \tilde{w}_{i,R,\varepsilon}\right)}^2\dx-\lambda_{n_0}
    =\max_{\substack{a_1,\dots,a_{n_0}\in \R \\
        \sum_{i=1}^{n_0}\abs{a_i}^2=1}}\sum_{i,j=1}^{n_0} L_{i,j}^\e a_ia_j ,
\]
  where
  \[
    L_{i,j}^\e=\frac{\int_{\Omega}\nabla
\hat w_{i,R,\e}\cdot\nabla
\hat w_{j,R,\e}\dx}{\|\hat w_{i,R,\e}\|_{L^2(\Omega)}\|\hat w_{j,R,\e}\|_{L^2(\Omega)}}-\lambda_{n_0}\delta_i^j.
  \]
From estimates \eqref{eq:51}, \eqref{eq:53}, and \eqref{eq:54} it
follows that
  \begin{gather*}
   L_{n_0,n_0}^\e=\e^{N+2\gamma-2}(g_R(\e)+o(1)),\quad
   L_{i,n_0}^\e=L_{n_0,i}^\e=O(\e^{N+\gamma-1})\quad\text{for all }i<n_0, \\
   L_{i,i}^\e=\lambda_i-\lambda_{n_0}<0  \quad\text{for all }i<n_0, \quad
   L_{i,j}^\e=0\quad\text{for all }i,j<n_0,\quad i\neq j,
  \end{gather*}
  as $\e\to 0$. We observe that $g_R(\e)=O(1)$ as $\e \to
  0$ by \eqref{eq:56} and \Cref{lem:U_R_conv}. Therefore estimate
  \eqref{eq:lower_bound} follows from Lemma \ref{lemma:quadratic_form}. Finally
  the limit $\lim_{\varepsilon\to0} g_R(\varepsilon)=g_R$ is a
  direct consequence of \eqref{eq:55}, \Cref{lem:U_R_conv},
  and \eqref{eq:56}.
\end{proof}

\noindent
With the purpose of deducing from \eqref{eq:lower_bound} a more
precise estimate from above of the eigenvalue variation
$\lambda_{n_0}^\e-\lambda_{n_0}$ and, in particular, of recognizing a
sign in the right-hand side of \eqref{eq:lower_bound}, we are now
going to compute the limit of the function $g_R$, as $R$ diverges.  To
this aim we define
\begin{align}
  \chi_R(r):=\int_{S_1^+} U_R(r\theta) \Psi(\theta)\ds,\quad
  &\text{for }R>2\text{ and }1\leq r\leq R, \label{eq:def_chi_R} \\
  \chi(r):=\int_{S_1^+} U(r\theta) \Psi(\theta)\ds,\quad
  &\text{for }r\geq 1. \label{eq:def_chi}
\end{align}
In addition, hereafter we denote
\begin{equation}\label{eq:def_pi_0}
	\pi_0=\pi_0(n_0):=\int_{S_1^+}\Psi^2\ds.
\end{equation}
We first establish the following preliminary result.

\begin{lemma}\label{lemma:lim_chi_R}
  Let $m_{n_0}(\Sigma)$ be the constant defined in \eqref{eq:def_m},
  let $\chi_R$ be as in \eqref{eq:def_chi_R}, $\chi$ as in
  \eqref{eq:def_chi} and $\pi_0$ as in \eqref{eq:def_pi_0}. Then
	\begin{equation}\label{eq:lim_chi_R_th1}
		\chi(1)=\lim_{R\to+\infty}\chi_R(1)=\pi_0-\frac{2 m_{n_0}(\Sigma)}{N+2\gamma-2}.
	\end{equation}
\end{lemma}
\begin{proof}
 The fact that $\chi_R(1)\to \chi(1)$ as $R\to+\infty$ is a
  consequence of Lemma \ref{lemma:U_R_conv} and continuity 
 of the trace map from $H^1(B_R^+)$ to $L^2(S_1^+)$. We now claim that
	\begin{equation}\label{eq:lim_chi_R_1}
		r^{-\gamma}\chi(r)\to \pi_0\quad\text{as }r\to+\infty.
	\end{equation}
To prove \eqref{eq:lim_chi_R_1}, we first observe that
	\begin{equation}\label{eq:lim_chi_R_2}
		\chi(r)=\pi_0 r^{\gamma}+\int_{S_1^+}w_0 (r\theta)\Psi(\theta)\ds,
	\end{equation}
 since $U(x)=w_0(x)+\abs{x}^{\gamma}\Psi(x/\abs{x})$ by \eqref{eq:def_U}.
By considering  the Kelvin transform of
   the restriction of $w_0$ on $\R^N_+\setminus B_1^+$ and observing
   that it must vanish at $0$ at least with vanishing order $1$ (see
   \cite{FF-proca}), we deduce  that 
\begin{equation*}
|w_0(x)|=O(|x|^{-N+1})\quad\text{as }|x|\to+\infty.
\end{equation*}
Combining the above estimate with \eqref{eq:lim_chi_R_2} we obtain
claim \eqref{eq:lim_chi_R_1}, being $\gamma\geq 1$.

  In view of the definition of $\chi$ given in \eqref{eq:def_chi}, the equation satisfied by $U$ and the fact that $\Psi$ is a spherical harmonic of
degree $\gamma$, it's easy to prove that $\chi(r)$
solves the following differential equation
\begin{equation*}
  (r^{N+2\gamma-1}(r^{-\gamma}\chi(r))')'=0 \quad\text{in }[1,+\infty).
\end{equation*}
Integration of the the above equation yields that
\begin{equation}\label{eq:lim_chi_R_3}
  r^{-\gamma}\chi(r)=\chi(1)+C\frac{1-r^{-N-2\gamma+2}}{N+2\gamma-2},\quad\text{for
    all }r\geq 1
\end{equation}
and for some $C\in\R$. Taking into account \eqref{eq:lim_chi_R_1}, we
obtain the exact value of the constant $C$, i.e.
	\[
		C=(N+2\gamma-2)(\pi_0-\chi(1)).
	\]
	Then \eqref{eq:lim_chi_R_3} can be rewritten as
	\begin{equation}\label{eq:lim_chi_R_6}
		\chi(r)=\pi_0 r^{\gamma}-(\pi_0-\chi(1))r^{-N-\gamma+2},
	\end{equation}
	whose derivative is
\begin{align}\label{eq:lim_chi_R_7}
	\chi'(r)&=\pi_0 \gamma
                  r^{\gamma-1}-(N+\gamma-2)(\chi(1)-\pi_0)r^{-N-\gamma+1}
  \\
  \notag&=(N+2\gamma-2)\pi_0
          r^{\gamma-1}-\frac{N+\gamma-2}{r}\chi(r),\quad\text{for all }r\geq1.
\end{align}
	Then, by computing the derivative in \eqref{eq:def_chi} as well, and evaluating it at $r=1$, we have that
	\begin{equation}\label{eq:lim_chi_R_4}
		\int_{S_1^+}\Psi\partial_{\nnu}U\ds=(N+2\gamma-2)\pi_0-(N+\gamma-2)\chi(1).
	\end{equation}
	Thanks to the harmonicity of the function
        $\psi$, the definition of $U$ given in \eqref{eq:def_U} and
        \eqref{eq:m_int_by_parts}, one can see that
	\[
		\int_{B_1^+}\nabla \psi\cdot\nabla U\dx=\int_{S_1^+} U\partial_{\nnu}\psi+2m_{n_0}(\Sigma).
	\]
	On the other hand, being $\psi$ a $\gamma$-homogeneous
        polynomial $\partial_{\nnu}\psi =\gamma\psi$ on
        $S_1^+$ 
	and so
	\begin{equation}\label{eq:lim_chi_R_5}
		\int_{B_1^+}\nabla \psi\cdot\nabla U\dx=\gamma\chi(1)+2m_{n_0}(\Sigma).
	\end{equation}
	Moreover, by \eqref{eq:63} and integration by parts we have that
	\[
		\int_{B_1^+}\nabla \psi\cdot\nabla U\dx=\int_{S_1^+}\Psi\partial_{\nnu} U\ds.
	\]
	Combining the identity above with \eqref{eq:lim_chi_R_5} and
        \eqref{eq:lim_chi_R_4} and rearranging the terms, we finally
        obtain \eqref{eq:lim_chi_R_th1} and complete the proof.
\end{proof}

As a byproduct of the proof of the previous Lemma, we obtain the following result, which is needed in the Section \ref{sec:proof_main}.

\begin{corollary}\label{cor:chi}
	For all $R>1$ there holds
	\begin{equation}\label{eq:chi_th1}
		\int_{S_R^+}\psi\partial_{\nnu}U\ds=\pi_0\gamma R^{N+2\gamma-2}+\frac{2(N+\gamma-2)}{N+2\gamma-2}m_{n_0}(\Sigma)
	\end{equation}
	as well as
        \begin{equation}\label{eq:chi_th2}
		\chi(R)=\pi_0R^{\gamma}-\frac{2m_{n_0}(\Sigma)}{N+2\gamma-2}R^{-N-\gamma+2}.
	\end{equation}
\end{corollary}
\begin{proof}
	By definition \eqref{eq:def_chi} we have that
	\begin{equation*}
		\int_{S_R^+}\psi\partial_{\nnu} U\ds=R^{N+\gamma-1}\chi'(R).
	\end{equation*}
	Plugging \eqref{eq:lim_chi_R_7}, \eqref{eq:lim_chi_R_6} and
        \eqref{eq:lim_chi_R_th1} into the previous identity, one can
        deduce \eqref{eq:chi_th1}. On the other hand
        \eqref{eq:chi_th2} can be easily proved by plugging
        \eqref{eq:lim_chi_R_th1} into \eqref{eq:lim_chi_R_6}.
\end{proof}

We are now able to compute the limit of $g_R$ as $R$ diverges.

\begin{lemma}\label{lemma:lim_g_R}
	Let $g_R$ be as in \eqref{eq:def_g_R} and $m_{n_0}(\Sigma)$ as in \eqref{eq:def_m}. Then $\lim_{R\to\infty}g_R=2m_{n_0}(\Sigma)$.
\end{lemma}
\begin{proof}
  From \eqref{eq:60}, harmonicity of $\psi$ and the
    fact that $\psi=0$ on $\partial\R^N_+$ it follows that
	\begin{equation}\label{eq:lim_g_R_1}
		g_R=\int_{S_R^+}\psi\partial_{\nnu} U_R\ds-\int_{S_R^+}\psi\partial_{\nnu}\psi\ds.
	\end{equation}
        Let us compute the two terms at the right hand side of
        \eqref{eq:lim_g_R_1}. If $\chi_R$ is the function defined in
        \eqref{eq:def_chi_R}, then
\begin{equation}\label{eq:lim_g_R_2}
          \chi_R'(R)=\int_{S_1^+}\Psi(\theta)\partial_{\nnu}U_R(R\theta)\ds=R^{-N-\gamma+1}\int_{S_R^+}{\psi}\partial_{\nnu} U_R\ds.
	\end{equation}
	On the other hand, one can easily prove that the function $\chi_R$ solves the following ODE
	\begin{equation*}
		(r^{N+2\gamma-1}(r^{-\gamma}\chi_R(r))')'=0\quad\text{in }[1, R],
	\end{equation*}
so that, by integration, there exists $C\in \R$ such that 
	\begin{equation*}
		r^{-\gamma}\chi_R(r)=\chi_R(1)+C\frac{1-r^{-N-2\gamma+2}}{N+2\gamma-2}\quad\text{for
                all }r\in [1,R].
	\end{equation*}
	Since $U_R=\psi=R^{\gamma}\Psi$ on $S_R^+$, then, by
        \eqref{eq:def_chi_R} and \eqref{eq:def_pi_0},
        $\chi_R(R)=\pi_0 R^{\gamma}$. Therefore the
        constant $C$ above is explicitly given by
	\[
		C=\frac{(N+2\gamma-2)(\pi_0-\chi_R(1))}{1-R^{-N-2\gamma+2}}.
	\]
Hence, in view of \eqref{eq:lim_g_R_2}, we can rewrite  the first term
in \eqref{eq:lim_g_R_1} as
\begin{equation}
  \int_{S_R^+}\psi\partial_{\nnu}U_R\ds=R^{N+\gamma-1}\chi_R'(R) =\frac{\pi_0(N+\gamma-2)-\chi_R(1)(N+2\gamma-2)+\pi_0\gamma R^{N+2\gamma-2}}{1-R^{-N-2\gamma+2}}. \label{eq:lim_r_R_3}
\end{equation}
	Concerning the second term in \eqref{eq:lim_g_R_1}, from
        \eqref{eq:def_pi_0} and the fact that
	\[
          \psi\partial_{\nnu}\psi=\gamma R^{2\gamma-1}\Psi^2
          \quad\text{on }S_R^+,
	\]
	we may easily deduce that
	\[
		\int_{S_R^+}\psi\partial_{\nnu}\psi\ds=\pi_0\gamma R^{N+2\gamma-2}.
	\]
	Plugging the previous identity and \eqref{eq:lim_r_R_3} into \eqref{eq:lim_g_R_1} we obtain that
	\[
		g_R=\frac{(\pi_0-\chi_R(1))(N+2\gamma-2)}{1-R^{-N-2\gamma+2}}.
	\]
	In view of Lemma \ref{lemma:lim_chi_R}, passing to
        the limit as $R\to\infty$ in the previous identity, we draw
        the conclusion.
	\end{proof}

  Combining Proposition \ref{prop:lower_bound} and Lemma
  \ref{lemma:lim_g_R} we directly  obtain the following lower bound for the
  eigenvalue variation $\lambda_{n_0}-\lambda_{n_0}^\e$.
\begin{corollary}\label{cor:lowbound}
  We have that
  \begin{equation*}
    \liminf_{\e\to0}\frac{\lambda_{n_0}-\lambda_{n_0}^\e}{\e^{N+2\gamma-2}}\geq
    -2m_{n_0}(\Sigma)>0.
  \end{equation*}
\end{corollary}

 Combining Proposition \ref{prop:upper_bound} and Corollary
\ref{cor:lowbound} we finally obtain the following result.

\begin{corollary}\label{cor:first_up_low_bound}
For any $R\geq K_\tau$ fixed we have that
	\[
		-2m_{n_0}(\Sigma)+o(1)\leq
                \frac{\lambda_{n_0}-\lambda_{n_0}^\e}
                {\e^{N+2\gamma-2}}\leq
                \frac{H(\e)}{\e^{2\gamma}}(f_R(\e)+o(1))
                \quad\text{as }\e\to0,
	\]
	where $f_R(\e)$ and $H(\e)$ are defined in
        \eqref{eq:def_f_R_eps} and \eqref{eq:61} respectively. In particular
	\begin{equation}\label{eq:cor_up_low_2}
		\frac{\e^{2\gamma}}{H(\e)}=O(1)\quad\text{as }\e\to 0.
	\end{equation}
\end{corollary}

\section{Blow-up analysis}\label{sec:blow-up-analysis}
The analysis performed in the previous sections led, in
\Cref{cor:first_up_low_bound}, to an estimate of the eigenvalue
variation in terms of the normalization factor $H(\e)$. In order to
detect the sharp asymptotic behaviour of $H(\e)$ as $\e\to\ 0$, in the
present section we perform a blow-up analysis for scaled
eigenfunctions. The identification of the limit profile of blown-up
eigenfunctions will be possible thanks to the energy estimate in
\Cref{prop:blow_up_estim} below, which is based on the invertibility
of the Fréchet derivative of the operator $T$, defined as
\begin{align}\label{eq:def_T}
	T\colon H^1_0(\Omega)\times \R & \longrightarrow H^{-1}(\Omega)\times \R \\
\notag	(\varphi,\lambda)&\longmapsto T(\varphi,\lambda):=\left(-\Delta \varphi-\lambda\varphi, \int_{\Omega}\abs{\nabla\varphi}^2\dx-\lambda_{n_0}\right),
\end{align}
where
\[
	_{H^{-1}(\Omega)}\langle-\Delta\varphi -\lambda\varphi,v\rangle_{H^1_0(\Omega)}:=\int_\Omega\left(\nabla\varphi\cdot\nabla v-\lambda\varphi v\right)\dx.
\]
From  the normalization
  \eqref{eq:normaliz_limit_eigenfunction} it easily follows that 
\[
  T(\varphi_{n_0},\lambda_{n_0})=(0,0).
\]
    Additionally, as a consequence of the simplicity
      assumption \eqref{eq:simple_hp} and the Fredholm Alternative,
    it is easy to prove the following invertibility
      result for the Fréchet derivative of $T$ at
      $(\varphi_{n_0},\lambda_{n_0})$. One can see \cite[Lemma
    7.1]{Abatangelo2015} for the proof in a similar framework.

\begin{lemma}
	The functional $T$ defined in \eqref{eq:def_T} is
        Fréchet-differentiable at $(\varphi_{n_0},\lambda_{n_0})$ and
        its Fréchet derivative
\begin{align*}
  &\mathrm{d}T(\varphi_{n_0},\lambda_{n_0})\colon   H^1_0(\Omega)\times\R  \longrightarrow  H^{-1}(\Omega)\times\R, \\
  &\mathrm{d}T(\varphi_{n_0},\lambda_{n_0})(\varphi,\lambda)=\left(-\Delta\varphi-\lambda \varphi_{n_0}-\lambda_{n_0} \varphi,\,2\int_\Omega\nabla\varphi_{n_0}\cdot\nabla\varphi\dx\right),
\end{align*}
	is invertible.
\end{lemma}

The following Lemma states that the function $\xi_{n_0,R,\e}$, defined
in \eqref{eq:def_xi}, is a good approximation of the limit
eigenfunction $\varphi_{n_0}$ for small values of $\e$.

\begin{lemma}\label{lemma:conv_xi}
	Let $R\geq K_\tau$ and let $\xi_{n_0,R,\e}$ be as in \eqref{eq:def_xi}. Then
	\[
		\xi_{n_0,R,\e}\to \varphi_{n_0}\quad\text{in }H^1_0(\Omega),\quad\text{as }\e\to 0. 
	\]
\end{lemma}
\begin{proof}
	We first observe that, by definition,
	\begin{multline}\label{eq:conv_xi_1}
		\int_\Omega\abs{\nabla (\xi_{n_0,R,\e}-\varphi_{n_0})}^2\dx=\int_\Omega\abs{\nabla(\varphi_{n_0}^\e-\varphi_{n_0})}^2\dx\\
		-\int_{\Theta_{R\e}}\abs{\nabla(\varphi_{n_0}^\e-\varphi_{n_0})}^2\dx+\int_{\Theta_{R\e}}\abs{\nabla(\xi_{n_0,R,\e}^{\textup{int}}-\varphi_{n_0})}^2\dx.
	\end{multline}
	In view of Proposition \ref{prop:conv_eigenfunctions},
        estimates \eqref{eq:rough_estim_th1},
        \eqref{eq:estim_xint_th4} and Lemma \ref{lemma:phi_i_est}, we
        can estimate the right hand side of \eqref{eq:conv_xi_1}, thus
        obtaining 
	\[
		\int_\Omega\abs{\nabla (\xi_{n_0,R,\e}-\varphi_{n_0})}^2\dx\leq o(1)+O(\e^{N-2+\beta(1-\tau)})+O(\e^{N+2\gamma-2})=o(1)\quad\text{as }\e\to 0.
	\]
	The proof is thereby complete.
\end{proof}

We now state a crucial energy estimate that quantifies the rate of convergence in Lemma \ref{lemma:conv_xi}.

\begin{proposition}\label{prop:blow_up_estim}
	Let $R\geq K_\tau$. Then
	\begin{equation}\label{eq:blow_up_estim_th1}
	\int_\Omega\abs{\nabla(\xi_{n_0,R,\e}-\varphi_{n_0})}^2\dx=O(\e^{N-2}H(\e))
	\end{equation}
and
	\begin{equation}\label{eq:blow_up_estim_th2}
		\int_{\Omega\setminus \Theta_{R\e}}\abs{\nabla (\varphi_{n_0}^\e-\varphi_{n_0})}^2\dx=O(\e^{N-2}H(\e)),\quad\text{as }\e\to 0.
	\end{equation}
\end{proposition}
\begin{proof}
  Let $T$ be as in \eqref{eq:def_T}. Being $T$ differentiable at
  $(\varphi_{n_0},\lambda_{n_0})$, in view of Lemma
  \ref{lemma:conv_xi} and Proposition \ref{propo:conv_eigenvalues}
  there holds
	\begin{multline}\label{eq:blow_up_estim_1}
		T(\xi_{n_0,R,\e},\lambda_{n_0}^\e)=\mathrm{d} T(\varphi_{n_0},\lambda_{n_0})(\xi_{n_0,R,\e}-\varphi_{n_0},\lambda_{n_0}^\e-\lambda_{n_0})\\
		+o(\norm{\xi_{n_0,R,\e}-\varphi_{n_0}}_{H^1_0(\Omega)}+\abs{\lambda_{n_0}^\e-\lambda_{n_0}})
	\end{multline}
	as $\e\to 0$, where
	\[
		\norm{v}_{H^1_0(\Omega)}:=\left(
                  \int_\Omega\abs{\nabla v}^2\dx
                \right)^{1/2}\quad\text{for all }v\in H^1_0(\Omega).
	\]
	   Applying $(\mathrm{d}T(\varphi_{n_0},\lambda_{n_0}))^{-1}$ to both sides
        in \eqref{eq:blow_up_estim_1} and taking the norms, we obtain
    that
    	\begin{equation*}
		\norm{\xi_{n_0,R,\e}-\varphi_{n_0}}_{H^1_0(\Omega)}  +\abs{\lambda_{n_0}^\e-\lambda_{n_0}}\leq \norm{\mathrm{d} T(\varphi_{n_0},\lambda_{n_0})^{-1}}\norm{T(\xi_{n_0,R,\e},\lambda_{n_0}^\e)}_{H^{-1}(\Omega)\times\R}(1+o(1)),
	\end{equation*}
	as $\e\to 0$, where the norm of $\mathrm{d} T(\varphi_{n_0},\lambda_{n_0})^{-1}$ is intended in the space of linear bounded operators from $H^{-1}(\Omega)\times\R$ to $H^1_0(\Omega)\times\R$ and it is a constant independent of $R$ and $\e$. Therefore
	\begin{multline}\label{eq:blow_up_estim_2}
		\norm{\xi_{n_0,R,\e}-\varphi_{n_0}}_{H^1_0(\Omega)}  +\abs{\lambda_{n_0}^\e-\lambda_{n_0}} \\
		\leq C\left( \norm{-\Delta \xi_{n_0,R,\e}-\lambda_{n_0}^\e\xi_{n_0,R,\e}}_{H^{-1}(\Omega)}+\abs{\norm{\xi_{n_0,R,\e}}_{H^1_0(\Omega)}^2-\lambda_{n_0}} \right)(1+o(1)),
	\end{multline}
	as $\e\to 0$.
We first observe that the definition of $\xi_{n_0,R,\e}$ given in
  \eqref{eq:def_xi}, \eqref{eq:estim_xint_th1}, 
  \eqref{eq:energy_estim_th1}, and Proposition \ref{prop:upper_bound} imply that  
  \begin{equation}\label{eq:blow_up_estim_4}
    \abs{\norm{\xi_{n_0,R,\e}}_{H^1_0(\Omega)}^2-\lambda_{n_0}}\leq \abs{\int_\Omega\abs{\nabla\xi_{n_0,R,\e} }^2\dx-\lambda_{n_0}^\e}+\abs{\lambda_{n_0}^\e-\lambda_{n_0}}=O(\e^{N-2}H(\e)),
	\end{equation}
	as $\e\to 0$.
      Let us now study the other term at the right
      hand side of \eqref{eq:blow_up_estim_2}. For any
      $v\in H^1_0(\Omega)$ we have that, by definition of
      $\xi_{n_0,R,\e}$ as in \eqref{eq:def_xi},
	\begin{align*}
		_{H^{-1}(\Omega)}\langle -&\Delta\xi_{n_0,R,\e}-\lambda_{n_0}^\e \xi_{n_0,R,\e}, v\rangle_{H^1_0(\Omega)}=\int_\Omega(\nabla\xi_{n_0,R,\e}\cdot\nabla v-\lambda_{n_0}^\e\xi_{n_0,R,\e}v)\dx \\
		=& \int_{\Theta_{R\e}}(\nabla \xi^{\textup{int}}_{n_0,R,\e}\cdot\nabla v-\lambda_{n_0}^\e \xi^{\textup{int}}_{n_0,R,\e}v)\dx -\int_{\Theta_{R\e}}(\nabla \varphi_{n_0}^\e\cdot\nabla v	-\lambda_{n_0}^\e\varphi_{n_0}^\e v)\dx.
	\end{align*}
	Now, thanks to the boundedness with respect to $\e$ of
        $\{\lambda_{n_0}^\e\}$, Cauchy-Schwartz inequality and by
        virtue of estimates \eqref{eq:energy_estim_th1},
        \eqref{eq:energy_estim_th2}, \eqref{eq:estim_xint_th1}, 
        \eqref{eq:estim_xint_th2} and Poincar\'e inequality, we have that
        \begin{align*}
  _{H^{-1}(\Omega)}\langle -\Delta\xi_{n_0,R,\e}-\lambda_{n_0}^\e \xi_{n_0,R,\e}, v\rangle_{H^1_0(\Omega)} 
          &=O\left(\e^{\frac{N-2}{2}}\sqrt{H(\e)}\right)\norm{v}_{H^1_0(\Omega)}+O\left(\e^{\frac{N}{2}}\sqrt{H(\e)}\right)\norm{v}_{L^2(\Omega)}\\
        &=O\left(\e^{\frac{N-2}{2}}\sqrt{H(\e)}\right)\norm{v}_{H^1_0(\Omega)}  ,
        \end{align*}
	as $\e\to 0$ and this readily implies that
	\begin{equation*}
		\norm{-\Delta \xi_{n_0,R,\e}-\lambda_{n_0}^\e\xi_{n_0,R,\e}}_{H^{-1}(\Omega)}=O\left(\e^{\frac{N-2}{2}}\sqrt{H(\e)}\right),\quad\text{as }\e\to 0.
	\end{equation*}
	Statement \eqref{eq:blow_up_estim_th1} follows by plugging the
        previous estimate and  \eqref{eq:blow_up_estim_4} into
        \eqref{eq:blow_up_estim_2}.
        
  Since by \eqref{eq:def_xi} we have that 
       \begin{equation*}
	\int_\Omega\abs{\nabla(\xi_{n_0,R,\e}-\varphi_{n_0})}^2\dx\geq
        	\int_{\Omega\setminus \Theta_{R\e}}\abs{\nabla (\varphi_{n_0}^\e-\varphi_{n_0})}^2\dx,
              \end{equation*}
estimate \eqref{eq:blow_up_estim_th2} directly follows
  from \eqref{eq:blow_up_estim_th1}.
\end{proof}

 We are now ready to perform a blow-up analysis for scaled eigenfunctions.

\begin{theorem}[Blow-up]\label{thm:blow_up}
  Let $U$ be as in \eqref{eq:def_U} and $\Upsilon^\e$ be as in
  \eqref{eq:def_Z_Ups}. Then, for all $R\geq K_\tau$, 
  \begin{equation}\label{eq:blow_up_th1}
    \Upsilon^\e\to \frac{1}{\sqrt{\Lambda_\tau}}\,U\quad\text{in
    }H^1(B_R^+)\quad\text{as }\e\to 0
  \end{equation}
  and
\begin{equation}\label{eq:68}
 \frac{H(\e)}{\e^{2\gamma}}\to \Lambda_\tau\quad\text{as }\e\to 0,
\end{equation}
	where
\begin{equation}\label{eq:blow_up_th3}
		\Lambda_\tau:=\frac{1}{K_\tau^{N-1}}\int_{S ^+_{K_\tau}}U^2\ds\,>0.
	\end{equation}
In particular, for all $R\geq K_\tau$, we have that
\begin{equation}\label{eq:blow_up_th4}
	\frac{u_{n_0}^\varepsilon(\varepsilon x)}{\varepsilon^{\gamma}}\to U(x)\quad\text{in }H^1(B_R^+)\quad\text{as }\varepsilon\to 0.
\end{equation}
\end{theorem}
\begin{proof}
  Let $\e_n\to 0$ as $n\to\infty$.  Firstly, from
  \eqref{eq:cor_up_low_2} we deduce that there exists $c\in\R$ such
  that  $c\geq 0$ and, up to a subsequence,
	\begin{equation}\label{eq:blow_up_8}
		q(\e_n):=\frac{\e_n^{\gamma}}{\sqrt{H(\e_n)}}\to c\quad\text{as }n\to\infty.
	\end{equation}
	Secondly, thanks to Proposition \ref{prop:energy_estim}, we have that, for any $R\geq K_\tau$, 
	\[
		\int_{B_R^+}A(\e_n x)\nabla \Upsilon^{\e_n}(x)\cdot\nabla \Upsilon^{\e_n}(x)\dx=O(1)\quad\text{and}\quad \int_{B_R^+}p(\e_n x)\abs{\Upsilon^{\e_n}}^2\dx=O(1)
	\]
	as $n\to\infty$. Therefore, by a diagonal process there exists $\tilde{U}\in H^1_{\textup{loc}}(\overline{\R^N_+})$ such that, up to a subsequence,
\begin{equation}
  \Upsilon^{\e_n}\rightharpoonup \tilde{U}\quad\text{weakly in
  }H^1(B_R^+),\quad \Upsilon^{\e_n}\to \tilde{U}\quad\text{strongly in
  }L^2(B_R^+),\label{eq:blow_up_2}
\end{equation}
and
\begin{equation}
  \Upsilon^{\e_n}\to \tilde{U}\quad\text{strongly in }L^2(S_R^+),\label{eq:blow_up_3}
\end{equation}
as $n\to\infty$ and for all $R\geq K_\tau$. Since, by
definition,
\[ \int_{S_{K_\tau}^+}{\mu}(\e_n
  x)\abs{\Upsilon^{\e_n}}^2\ds=K_\tau^{N-1},
	\]
thanks to \eqref{eq:38} and \eqref{eq:blow_up_3} we can pass to the limit and infer that
	\begin{equation}\label{eq:blow_up_11}
		\int_{S_{K_\tau}^+}\tilde{U}^2\ds=K_\tau^{N-1},
	\end{equation}
	which implies that $\tilde{U}\not\equiv 0$ in $\R^N_+$.
        From the convergence,  as
  $\e\to0$, in the sense of Mosco of
  $\R^{N}\setminus
  \big(\partial\R^N_+\setminus\big(\frac1\e\widetilde\Sigma_\e\big)\big)$
  to
  the set $\R^{N}\setminus \big(\partial\R^N_+\setminus\Sigma\big)$,
       observed in Remark \ref{rem:Mosco}, we derive that
       $\tilde{U}\in H^1_{0,B_R'\setminus\Sigma}(B_R^+)$
       for all $R\geq K_\tau$. In addition, $\tilde{U}$ weakly solves
	\begin{equation}\label{eq:blow_up_10}
		\begin{bvp}
			-\Delta \tilde{U}&=0, &&\text{in }\R^N_+, \\
			\tilde{U}&=0, &&\text{on }\partial\R^N_+\setminus\Sigma, \\
			\partial_{\nnu}\tilde{U}&=0, &&\text{on }\Sigma.
		\end{bvp}
	\end{equation}
In particular
	\begin{equation}\label{eq:blow_up_4}
		\int_{B_R^+}|\nabla \tilde{U}|^2\dx=\int_{S_R^+}\tilde{U}\partial_{\nnu} \tilde{U}\ds,\quad\text{for all }R\geq K_\tau.
	\end{equation}
	 We now aim at proving that
	\begin{equation}\label{eq:blow_up_5}
		\Upsilon^{\e_n}\to \tilde{U}\quad\text{strongly in }H^1(B_R^+),\quad\text{as }n\to \infty,
	\end{equation}
	for all $R\geq K_\tau$. For every $R\geq
          K_\tau$, we have that, for $n$ sufficiently large, $\Upsilon^{\e_n}$ weakly
        solves
\[
\begin{cases}
  -\dive(A(\e_n x)\nabla \Upsilon^{\e_n})(x)=\e_n^2\lambda_{n_0}^{\e_n}p(\e_n x)\Upsilon^{\e_n}(x), &\text{in }B_R^+, \\
  \Upsilon^{\e_n}(x)=0, &\text{on }B_R'\setminus \frac{1}{\e_n} \widetilde{\Sigma}_{\e_n}, \\
  A(\e_n x)\nabla \Upsilon^{\e_n}(x)\cdot\nnu(x)=0, &\text{on }\frac{1}{\e_n}\widetilde{\Sigma}_{\e_n}, \\
  \Upsilon^{\e_n}(x)=\dfrac{u_{n_0}^{\e_n}(\e_n x)}{\sqrt{H(\e_n)}},
  &\text{on }S_R^+.
\end{cases}
\]
For $R\geq K_\tau$, if we
consider the restriction of $\Upsilon^{\e_n}$ to
$B_R^+\setminus B_{R/2}^+$ and we oddly reflect it through the
hyperplane $\{x_N=0\}$, given the equation this function satisfies,
from classical elliptic regularity theory (see e.g. \cite[Theorem
2.3.3.2]{grisvard}) we know that
 $\{\Upsilon^{\e_n}\}_{n}$ is bounded in $H^2(B_R\setminus
B_{R/2})$.
 Therefore, up to a subsequence (still denoted by $\varepsilon_n$), we have that
\begin{equation}\label{eq:blow_up_7}
	\partial_{\nnu}\Upsilon^{\varepsilon_n}\to \partial_{\nnu}\tilde{U}\quad\text{in }L^2(S_R^+),\quad\text{as }n\to\infty.
\end{equation}
Furthermore, from the equation satisfied by $\Upsilon^{\varepsilon_n}$, \eqref{eq:33}, \eqref{eq:bound_p}, \eqref{eq:blow_up_3} and \eqref{eq:blow_up_7} we have, as $n\to\infty$,
\begin{align*}
	\int_{B_R^+}\abs{\nabla \Upsilon^{\varepsilon_n}}^2\dx&=(1+o(1))\int_{B_R^+}A(\varepsilon_n x)\nabla \Upsilon^{\varepsilon_n}(x)\cdot\nabla \Upsilon^{\varepsilon_n}(x)\dx \\
	&=(1+o(1))\left(O(1)\varepsilon_n^2\lambda_{n_0}^{\varepsilon_n}\int_{B_R^+}\abs{\Upsilon^{\varepsilon_n}}^2\dx+\int_{S_R^+}\tilde{U}\partial_{\nnu}\tilde{U}\ds+o(1) \right).
\end{align*}
Therefore, thanks to \eqref{eq:blow_up_2} and
\eqref{eq:blow_up_4}, we conclude that
\begin{equation*}
	\int_{B_R^+}\abs{\nabla \Upsilon^{\varepsilon_n}}^2\dx\to \int_{B_R^+}|\nabla \tilde{U} |^2\dx,
\end{equation*}
which, together with \eqref{eq:blow_up_2}, proves \eqref{eq:blow_up_5}.

Now let us fix $R\geq K_\tau$. From \eqref{eq:blow_up_estim_th2}, we
know that there exist $C_R>0$ and $n_R\in\N$ such that 
	\[
		\int_{\Theta_{\tilde{R}\e}\setminus\Theta_{R\e}}\abs{\nabla
                  (\varphi_{n_0}^{\e_n}-\varphi_{n_0})}^2\dx\leq C_R \e_n^{N-2}H(\e_n),
	\]
	for all $\tilde{R}>R$ and $n>n_R$. In fact, up to a change of
      variable, this is equivalent to
	\begin{equation*}
		\int_{B_{\tilde{R}}^+\setminus B_R^+}A(\e_n x)\nabla
                \left( \Upsilon^{\e_n}-q(\e_n)
                  V_{\e_n}\right) (x) \cdot \nabla \left(
                  \Upsilon^{\e_n}-q(\e_n)
                  V_{\e_n}\right) (x) \dx \leq C_R
              \end{equation*}
              for all $\tilde{R}>R$ and $n>n_R$, where
              $q(\e_n)$ is defined in \eqref{eq:blow_up_8} and
              $V_{\e_n}$ in \eqref{eq:Ve}. Passing to the limit as
              $n\to\infty$ in the above estimate and taking into account
              \eqref{eq:blow_up_8}, \eqref{eq:blow_up_5},
              \eqref{eq:33} and \eqref{eq:56}, we obtain that 
	\begin{equation*}
		\int_{B_{\tilde{R}}^+\setminus B_R^+}|\nabla \tilde{U}-c\nabla \psi|^2\dx\leq C_R
	\end{equation*}
	for all $\tilde{R}>R$ and this readily implies that
	\begin{equation}\label{eq:65}
		\int_{\R^N_+}|\nabla \tilde{U}-c\nabla \psi|^2\dx <\infty.
	\end{equation}
	We now claim that $c>0$.
        Indeed, if this were not the case, then
$\int_{\R^N_+}|\nabla \tilde{U}|^2\dx <\infty$. Then,
  since $U=0$ on $\partial\R^N_+\setminus\Sigma$, in view of
\cite[Lemma 2.3]{FT-tubi} we would have $\tilde{U}\in
\mathcal{D}^{1,2}(\R^N_+\cup \Sigma)$; since $\tilde{U}$ weakly solves
\eqref{eq:blow_up_10}, this would imply that $\tilde{U}\equiv 0$, thus
raising a contradiction.

From \eqref{eq:65} and \cite[Lemma 2.3]{FT-tubi}  it
  follows that $c^{-1}\tilde U-\psi\in \mathcal{D}^{1,2}(\R^N_+\cup
  \Sigma)$; hence, by uniqueness of the limit profile constructed in Lemma
  \ref{lemma:def_w}, we conclude that $c^{-1}\tilde U-\psi=w_0$.
Hence, by the definition of $U$ in \eqref{eq:def_U}, we have that
	\[
		\tilde{U}=c U.	
	\]
	Moreover, in view of \eqref{eq:blow_up_11}, we conclude that
	\[
		c=\frac{1}{\sqrt{\Lambda_\tau}},
	\]
	with $\Lambda_\tau$ as in \eqref{eq:blow_up_th3}.  Since the
        limit of $\Upsilon^{\e_n}$ is independent of the choice of the
        sequence $\{\e_n\}_n$ and of the extracted subsequence, by
        Uryshon Subsequence Principle we may conclude that the
        convergence holds as $\e\to0$. Finally,
          \eqref{eq:blow_up_th4} is a direct consequence of
          \eqref{eq:blow_up_th1} and \eqref{eq:68}.
\end{proof}

As a consequence of the Blow-up Theorem
  \ref{thm:blow_up}, we are able to prove the strong convergence 
 as $\e\to0$ of the family $\{Z_R^\e\}_\e$ defined in
\eqref{eq:def_Z_Ups}.

\begin{corollary}\label{cor:conv_Z_R_eps}
	For any $R>K_\tau$, there holds
	\[
		Z_R^\e\to \frac{1}{\sqrt{\Lambda_\tau}} Z_R\quad \text{in }H^1(B_R^+)\quad\text{as }\e\to 0,
	\]
	where $Z_R$ is defined in Lemma \ref{lemma:Z_R}.
\end{corollary}
\begin{proof}
Let us fix $R>K_\tau$. We observe that $Z_R^\e$ weakly solves 
\begin{equation*}
\begin{cases}
  -\dive(A(\e x)\nabla Z_R^\e)=0, &\text{in }B_R^+, \\
  Z_R^\e=\Upsilon^{\e} , &\text{on }S_R^+ ,\\
  Z_R^\e=0 , &\text{on }B_R',
\end{cases}
\end{equation*}
hence the function
\begin{equation*}
W_R^\e := Z_R^\e-\Lambda_\tau^{-1/2}Z_R-\eta_R(\Upsilon^{\e}-\Lambda_\tau^{-1/2}U),
\end{equation*}
with $\eta_R$ being as in \eqref{eq:def_cutoff},
weakly solves
\begin{equation*}
\begin{cases}
  -\dive(A(\e x)\nabla W_R^\e)=\dive\bigg(\dfrac{A(\e
  x)-I_N}{\sqrt{\Lambda_\tau}} \nabla Z_R+A(\e x)\nabla\Big(\eta_R\Big(\Upsilon^{\e}-\frac{U}{\sqrt{\Lambda_\tau}}\Big)\Big)\bigg), &\text{in }B_R^+, \\
  W_R^\e=0 , &\text{on }\partial B_R^+,
\end{cases}
\end{equation*}
i.e.
\begin{multline*}
  \int_{B_R^+}A(\e x)\nabla W_R^\e\cdot\nabla \phi\dx=-
 \frac1 {\sqrt{\Lambda_\tau}} \int_{B_R^+}(A(\e
 x)-I_N) \nabla Z_R\cdot\nabla\phi\dx\\
-\int_{B_R^+}
A(\e
x)\nabla\Big(\eta_R\Big(\Upsilon^{\e}-\frac{U}{\sqrt{\Lambda_\tau}}\Big)\Big)\cdot\nabla\phi\dx\quad\text{for
every }\phi\in H^1_0(B_R^+).
\end{multline*}
Testing the above equation with $\phi= W_R^\e$ and using
\eqref{eq:33},
we then obtain that 
\begin{equation*}
  \int_{B_R^+}A(\e x)\nabla W_R^\e\cdot\nabla W_R^\e\dx\leq
\mathop{\rm const}\|\nabla W_R^\e\|_{L^2(B_R^+)}
 \bigg(\e+\bigg\|  \eta_R\Big(\Upsilon^{\e}-\frac{U}{\sqrt{\Lambda_\tau}}\Big)\bigg\|_{H^1(B_R^+)}\bigg),
\end{equation*}
which implies that
\begin{equation}\label{eq:66}
W_R^\e\to 0\quad\text{in }H^1_0(B_R^+)\quad\text{as }\e\to0,
\end{equation}
thanks to \eqref{eq:33} and Theorem \ref{thm:blow_up}.
Since
$Z_R^\e-\Lambda_\tau^{-1/2}Z_R=W_R^\e +\eta_R(\Upsilon^{\e}-\Lambda_\tau^{-1/2}U)$,
from \eqref{eq:66} and Theorem \ref{thm:blow_up} we deduce that
$Z_R^\e-\Lambda_\tau^{-1/2}Z_R\to 0$ in $H^1(B_R^+)$.
\end{proof}

We conclude this section with the proof of Theorem \ref{thm:sharp_eigenfunction}.

\begin{proof}[Proof of Theorem \ref{thm:sharp_eigenfunction}]
	It can be easily derived from the change of variable $x=\Phi(y)$, \eqref{eq:blow_up_th4} and Dominated Convergence Theorem.
\end{proof}

\section{Proof of Theorem \ref{thm:main}}\label{sec:proof_main}

From Theorem \ref{thm:blow_up} and Corollary \ref{cor:conv_Z_R_eps} it
follows that, letting $f_R(\e)$ be as in \eqref{eq:def_f_R_eps},
\begin{equation}\label{eq:67}
\lim_{\e\to 0}	f_R(\e)=\frac{1}{\Lambda_\tau}\left(\int_{B_R^+}\abs{\nabla Z_R}^2\dx-\int_{B_R^+}\abs{\nabla U}^2\dx\right)
\end{equation}
for all $R>K_\tau$, 
in view also of \eqref{eq:33}.

Combining \eqref{eq:67} with \eqref{eq:68} and Corollary \ref{cor:first_up_low_bound}, at this point we know that
\begin{equation*}\label{eq:second_up_low_bound}
  -2m_{n_0}(\Sigma) \leq
  \liminf_{\e\to 0}\frac{\lambda_{n_0}-\lambda_{n_0}^\e}{\e^{N+2\gamma-2}} 
  \leq \limsup_{\e\to
    0}\frac{\lambda_{n_0}-\lambda_{n_0}^\e}{\e^{N+2\gamma-2}}
  \leq \int_{B_R^+}\abs{\nabla Z_R}^2\dx-\int_{B_R^+}\abs{\nabla U}^2\dx,
      \end{equation*}
or all $R>K_\tau$. Therefore the proof of Theorem \ref{thm:main} amounts to the proof of the following Lemma, which the rest of the Section is devoted to.

\begin{lemma}\label{lemma:conv_f_R}
	There holds 
	\begin{equation}\label{eq:conv_f_R_th1}
          \lim_{R\to+\infty}\bigg(
          \int_{B_R^+}\abs{\nabla Z_R}^2\dx-
          \int_{B_R^+}\abs{\nabla U}^2\dx\bigg)=-2m_{n_0}(\Sigma).
	\end{equation}
\end{lemma}

A first step in this direction is given by the following lemma.

\begin{lemma}\label{lemma:final_aux}
	There holds
	\begin{equation*}
	\lim_{R\to+\infty}\bigg(\int_{B_R^+}\abs{\nabla
                  Z_R}^2\dx-\int_{B_R^+}\abs{\nabla U}^2\dx-
\int_{S_R^+}\psi \partial_{\nnu}(Z_R-U)\ds\bigg)=0.
	\end{equation*}
\end{lemma}
\begin{proof}
 Integration by parts and equations
    \eqref{eq:63} and \eqref{eq:69} imply that 
	\begin{multline*}
		\int_{B_R^+}\abs{\nabla Z_R}^2\dx-\int_{B_R^+}\abs{\nabla U}^2\dx-\int_{S_R^+}\psi \partial_{\nnu}(Z_R-U)\ds\\
		=\int_{S_R^+}(U-\psi)\partial_{\nnu}(\psi-U)\ds+\int_{S_R^+}(U-\psi)\partial_{\nnu}(Z_R-\psi)\ds.
	\end{multline*}
		Therefore the conclusion follows if we prove that
	\begin{gather}
		\lim_{R\to+\infty}\int_{S_R^+}(U-\psi)\partial_{\nnu}(U-\psi)\ds=0, \label{eq:final_aux_1}\\
		\lim_{R\to+\infty}\int_{S_R^+}(U-\psi)\partial_{\nnu}(Z_R-\psi)\ds=0. \label{eq:final_aux_2}
	\end{gather}
First, we observe that integration by parts and the
  fact that $U-\psi\in  \mathcal{D}^{1,2}(\R^N_+\cup
  \Sigma)$ is harmonic in $\R^N_+$ imply that
	\[
		\int_{S_R^+}(U-\psi)\partial_{\nnu}(U-\psi)\ds=\int_{\R^N_+\setminus B_R^+}\abs{\nabla (U-\psi)}^2\dx.
	\]
Since $U-\psi\in \mathcal{D}^{1,2}(\R^N_+\cup \Sigma)$, the right hand
side vanishes as $R\to+\infty$, thus implying \eqref{eq:final_aux_1}.

In order to prove \eqref{eq:final_aux_2}, we let $R>2$ and consider the equation satisfied by $Z_R-\psi\in H^1(B_R^+)$ in $B_R^+$, i.e.	
\begin{equation*}
\begin{cases}
  -\Delta (Z_R-\psi)=0, &\text{in }B_R^+, \\
Z_R-\psi= 0, &\text{on }B_R', \\
Z_R-\psi= U-\psi, &\text{on }S_R^+.
\end{cases}
\end{equation*}
	If we multiply both sides of the above equation by $\eta_R(U-\psi)$, where $\eta_R$ is as in \eqref{eq:def_cutoff}, and integrate by parts, we obtain that
	\begin{equation*}
		\int_{S_R^+}(U-\psi)\partial_{\nnu}(Z_R-\psi)\ds=\int_{B_R^+}\nabla (Z_R-\psi)\cdot\nabla(\eta_R(U-\psi))\dx.
	\end{equation*}
	Therefore, from the Cauchy-Schwartz inequality and the
        Dirichlet principle it follows that
	\begin{equation}\label{eq:final_aux_3}
		\abs{\int_{S_R^+}(U-\psi)\partial_{\nnu}(Z_R-\psi)\ds}\leq \int_{B_R^+}\abs{\nabla(\eta_R(U-\psi))}^2\dx.
	\end{equation}
Thanks to \eqref{eq:def_cutoff}, we have that
\begin{equation}\label{eq:final_aux_4}
  \int_{B_R^+}\abs{\nabla(\eta_R(U-\psi))}^2\dx\leq 32\left(
  \int_{B_R^+\setminus B_{R/2}^+}\frac{\abs{U-\psi}^2}{\abs{x}^2}\dx+
  \int_{B_R^+\setminus B_{R/2}^+}\abs{\nabla (U-\psi)}^2\dx\right).
\end{equation}
Now, since $U-\psi\in\mathcal{D}^{1,2}(\R^N_+\cup \Sigma)$ and since
the Hardy inequality holds in this space, we have that
\[
  \int_{B_R^+\setminus B_{R/2}^+}\frac{\abs{U-\psi}^2}{\abs{x}^2}\dx
  +\int_{B_R^+\setminus B_{R/2}^+}\abs{\nabla (U-\psi)}^2\dx\to 0\quad\text{as }R\to+\infty.
\]
Combining this fact with \eqref{eq:final_aux_4} and
\eqref{eq:final_aux_3} we obtain \eqref{eq:final_aux_2}, thus
concluding the proof.
\end{proof}

We are now ready to prove Lemma \ref{lemma:conv_f_R}.

\begin{proof}[Proof of Lemma \ref{lemma:conv_f_R}]
  By virtue of Lemma \ref{lemma:final_aux}, to prove
  \eqref{eq:conv_f_R_th1} it is enough to show that
\begin{equation}\label{eq:70}
  \lim_{R\to+\infty}\int_{S_R^+}\psi\partial_{\nnu}(Z_R-U)\ds
  = -2m_{n_0}(\Sigma).
\end{equation}
	For $R>2$ we let
	\[
          \Gamma_R(r):=\int_{S_1^+}Z_R(r\theta)\Psi(\theta)\ds
          \quad\text{for
                  any }0< r\leq R.
	\]
        From \eqref{eq:69}  and the fact that $\Psi$ is a spherical harmonic of
degree $\gamma$ it easily follows that
	\[
		(r^{N+2\gamma-1}(r^{-\gamma}\Gamma_R(r))')'=0\quad\text{in }(0,R).
	\]
	Integrating this ODE in $(r,R)$ we obtain that
	\begin{equation*}
		r^{-\gamma}\Gamma_R(r)=R^{-\gamma}\Gamma_R(R)+\frac{C}{N+2\gamma-2}\left[R^{-N-2\gamma+2}-r^{-N-2\gamma+2}\right]
	\end{equation*}
	for some constant $C\in\R$ and for all
          $r\in(0,R)$. Multiplying both sides by $r^{N+2\gamma-2}$
        leads to
	\[
		r^{N+\gamma-2}\Gamma_R(r)=R^{-\gamma}r^{N+2\gamma-2}\Gamma_R(R)+\frac{C}{N+2\gamma-2}\left[R^{-N-2\gamma+2}r^{N+2\gamma-2}-1\right].
	\]
	 Tanking into account that, by regularity of $Z_R$,
         $\lim_{r\to 0}\Gamma_R(r)$ is finite, thanks to the previous
         identity, we may conclude that $C=0$, thus implying that
	 \begin{equation*}
	 	\Gamma_R(r)=\left(\frac{r}{R}\right)^{\gamma}\Gamma_R(R).
	 \end{equation*}
 	Moreover, since $Z_R=U$ on $S_R^+$, we have that $\Gamma_R(R)=\chi(R)$ and then
 	\begin{equation}\label{eq:conv_f_R_2}
 		\Gamma_R(r)=\left(\frac{r}{R}\right)^{\gamma}\chi(R).
 	\end{equation}
	 By definition of $\Gamma_R$, we know that
	 \[
	 	\int_{S_R^+}\psi\partial_{\nnu} Z_R\ds=R^{N+\gamma-1}\Gamma_R'(R)
	 \]
	 which, in view of \eqref{eq:conv_f_R_2}, becomes
	 \begin{equation*}
	 	\int_{S_R^+}\psi\partial_{\nnu} Z_R\ds=\gamma R^{N+\gamma-2}\chi(R).
	 \end{equation*}
Then, taking into account \eqref{eq:chi_th2}, we have that
 	\begin{equation*}
          \int_{S_R^+}\psi\partial_{\nnu} Z_R\ds
          =\pi_0 \gamma R^{N+2\gamma-2}-\frac{2\gamma m_{n_0}(\Sigma)}{N+2\gamma-2}.
 	\end{equation*}
 	Combining this identity with \eqref{eq:chi_th1} yields
 	\begin{equation*}
 		\int_{S_R^+}\psi\partial_{\nnu}(Z_R-U)\ds=-\frac{2\gamma m_{n_0}(\Sigma)}{N+2\gamma-2}-\frac{2(N+\gamma-2)}{N+2\gamma-2}m_{n_0}(\Sigma)=-2m_{n_0}(\Sigma),
 	\end{equation*}
which implies \eqref{eq:70}. The proof is thereby complete.
	\end{proof}

\section{Proof of Theorem \ref{thm:nonstarshaped}} \label{sec:proof-theor-refthm:m}

In this section, we drop assumptions
\eqref{eq:V-C11}--\eqref{eq:V-star-shaped} on the set $\mathcal V$ and
prove Theorem \ref{thm:nonstarshaped} under the sole assumption
\eqref{eq:7} on $\mathcal V$.
Let  $0<r_{\mathcal V}<R_{\mathcal V}<r_0$ be such that
  $B_{r_{\mathcal V}}\subset\mathcal V\subset B_{R_{\mathcal
      V}}$ (such $r_{\mathcal V},R_{\mathcal V}$ exist because
  $\mathcal V$ is an open bounded set containing $0$). For every
  $\omega\subset\R^N$ bounded open set, we denote
  as $\lambda_{n_0}^\e(\omega)$ the $n_0$-th eigenvalue of problem
  \eqref{eq:eqphiie} with $\Sigma_\e$ given by $(\e\mathcal
  \omega)\cap\partial\Omega$ (i.e. with $\mathcal V$ replaced by
  $\omega$). Then, from \eqref{eq:min_max} and the fact that  
$\e B_{r_{\mathcal V}}\subset\e \mathcal V\subset \e B_{R_{\mathcal
    V}}$ it follows that
\begin{equation}\label{eq:71}
  \lambda_{n_0}^\e( B_{R_{\mathcal
    V}})\leq \lambda_{n_0}^\e(\mathcal V)\leq \lambda_{n_0}^\e(B_{r_{\mathcal
    V}}).
\end{equation}
Since $B_{r_{\mathcal V}}$ and $B_{R_{\mathcal V}}$ satisfy
assumptions \eqref{eq:V-C11} and \eqref{eq:V-star-shaped}, Theorem
\ref{thm:main} and Lemma \ref{l:scaling-m0} yield the
following asymptotic expansions for
$\lambda_{n_0}^\e( B_{R_{\mathcal V}}), \lambda_{n_0}^\e(
B_{r_{\mathcal V}})$:

\begin{align*}
  \lambda_{n_0}^\e( B_{R_{\mathcal V}})& =\lambda_{n_0}-R_{\mathcal V}^{N+2\gamma-2}\mathcal C_{n_0}
  \e^{N+2\gamma-2}+o(\e^{N+2\gamma-2}),\\
    \lambda_{n_0}^\e( B_{r_{\mathcal V}})&=\lambda_{n_0}-r_{\mathcal V}^{N+2\gamma-2}\mathcal C_{n_0}
  \e^{N+2\gamma-2}+o(\e^{N+2\gamma-2}),
\end{align*}
as $\e\to0$, so that, in view of \eqref{eq:71},
\begin{align*}
 r_{\mathcal V}^{N+2\gamma-2}\mathcal C_{n_0}+o(1) 
=  \frac{\lambda_{n_0}-\lambda_{n_0}^\e( B_{r_{\mathcal
    V}})}{ \e^{N+2\gamma-2}}&\leq \frac{\lambda_{n_0}-\lambda_{n_0}^\e(\mathcal
  V)}{ \e^{N+2\gamma-2}}\\
  &\leq
\frac{\lambda_{n_0}-\lambda_{n_0}^\e(B_{R_{\mathcal V}})}{ \e^{N+2\gamma-2}}=R_{\mathcal V}^{N+2\gamma-2}\mathcal C_{n_0}+o(1)
\end{align*}
as $\e\to0$. The above chain of inequalities directly proves Theorem
\ref{thm:nonstarshaped}.

\appendix

\section{}\label{appendix:poincare}

We recall from \cite[Lemma 4.1]{FO} the following Poincar\'e-type inequality on
balls and half-balls.
\begin{lemma}\label{l:poin}
  Let $r>0$. Then
\begin{equation*}
 \frac{N-1}{r^2}\int_{B_r^+}u^2\dx\leq\int_{B_r^+}|\nabla
 u|^2\dx+\frac{1}{r}\int_{S_r^+}u^2\ds\quad\text{for every $u\in
   H^1(B_r^+)$},
 \end{equation*}
and 
\begin{equation*}
 \frac{N-1}{r^2}\int_{B_r}u^2\dx\leq\int_{B_r}|\nabla
 u|^2\dx+\frac{1}{r}\int_{\partial B_r}u^2\ds\quad\text{for every $u\in
   H^1(B_r)$}.
 \end{equation*}
\end{lemma}

	From \cite{Abatangelo2015} we recall the following result, regarding the maximum of quadratic forms with coefficients depending on a parameter (see also \cite{FO}).

  \begin{lemma}\label{lemma:quadratic_form}
	For every $\e>0$ let us consider a quadratic form
	\[
	\begin{aligned}
	&Q_\e\colon \R^{n_0} \longrightarrow \R, \\
	&Q_\e(z_1,\dots,z_{n_0})=\sum_{i,j=1}^{n_0} M_{i,j}(\e)z_i z_j,
	\end{aligned}
	\]
	with real coefficients $M_{i,j}(\e)$ such that
        $M_{i,j}(\e)=M_{j,i}(\e)$. Let us assume that there exist
        $a>0$, $\e\mapsto \sigma(\e)\in\R$ with
        $\sigma(\e)\geq 0$ and $\sigma(\e)=O(\e^{2a})$ as
        $\e\to 0$, and $\e\mapsto \mu(\e)\in\R$ with $\mu(\e)=O(1)$ as
        $\e\to 0$, such that the coefficients $M_{i,j}(\e)$ satisfy
        the following conditions:
	\begin{align*}
	& M_{n_0,n_0}(\e)=\sigma(\e)\mu(\e), \\
	& \text{for all }i<n_0~M_{i,i}(\e)\to M_i<0,~\text{as }\e\to 0, \\
	& \text{for all }i<n_0~M_{i,n_0}(\e)=O(\e^{a}\sqrt{\sigma(\e)})~\text{as }\e\to 0, \\
	& \text{for all }i,j<n_0~\text{with }i\neq j~M_{i,j}=O(\e^{2a})~\text{as }\e\to 0, \\
	& \text{there exists }M\in \mathbb{N}~\text{such that }\e^{(2+M)a}=o(\sigma(\e))~\text{as }\e\to 0.
	\end{align*}
	Then
	\[
	\max_{\substack{z \in\R^n_0 \\ \norm{z}=1}}Q_\e(z)=\sigma(\e)(\mu(\e)+o(1))\quad\text{as }\e\to 0,
	\]
	where $\norm{z}=\norm{(z_1,\dots,z_{n_0})}=\big( \sum_{i=1}^{n_0}z_i^2 \big)^{1/2}$.   
\end{lemma}

\section*{Acknowledgments}
\noindent
The authors would like to thank the  anonymous referee for the careful reading of the paper and many useful suggestions.
 B. Noris was partially supported by the INdAM - GNAMPA Project 2020 “Problemi ai limiti per l’equazione della curvatura media prescritta”. R. Ognibene was partially supported by the project ERC VAREG - \emph{Variational approach to the regularity of the free boundaries} (grant agreement No. 853404) and by the MIUR-PRIN project No. 2017TEXA3H. This work started during a visit of R. Ognibene at Département de mathématiques, Université de Picardie Jules Verne, which he warmly thanks for the hospitality.

\bibliography{biblio}
\bibliographystyle{acm} 

\end{document}